\documentclass[aos,preprint]{imsart}

\RequirePackage{amsthm,amsmath,amsfonts,amssymb}
\RequirePackage[numbers]{natbib}
\RequirePackage[colorlinks,citecolor=blue,urlcolor=blue]{hyperref}

\usepackage{color, multirow, graphicx}
\usepackage{float}
\usepackage{dsfont}
\usepackage{subcaption}
 \usepackage{siunitx}
\startlocaldefs

\newtheorem{lemma}{Lemma}[section]

\newtheorem{theorem}[lemma]{Theorem}

\newtheorem{proposition}[lemma]{Proposition}

\newtheorem{definition}[lemma]{Definition}

\newtheorem{corollary}[lemma]{Corollary}

\newtheorem{example}[lemma]{Example}

\newtheorem{remark}[lemma]{Remark}

 \newcommand{\rbraces}[1]{\left( #1 \right)}
\providecommand{\cov}[2]{\mathrm{Cov} \rbraces{#1 , #2}} 
\providecommand{\Var}[1]{\mathrm{Var} \rbraces{#1}}	
\providecommand{\EQ}[2]{\mathbb{E}_{#2} \rbraces{#1}}
\providecommand{\E}[1]{\EQ{#1}{}} 	
\providecommand{\abs}[1]{\left\lvert#1 \right\rvert}
\providecommand{\Pb}{\mathbb{P}}
\newcommand{\limn}{\lim\limits_{n \rightarrow \infty}}
\providecommand{\distConv}{\xrightarrow{\mathcal{D}}}	
\providecommand{\pConv}{\xrightarrow{\mathbb{P}}}	
\providecommand{\NoD}[2]{\mathcal{N} \rbraces{#1 , #2}}
\newcommand{\1}{\ensuremath{\mathds{1}}}
\newcommand{\N}{{\mathbb N}}
\newcommand{\R}{{\mathbb R}}
\newcommand{\elln}{{\tilde{\ell}_n}}
\newcommand{\bn}{{\tilde{b}_n}}

\endlocaldefs

\begin{document}

\begin{frontmatter}

\title{An Asymptotic Test for Constancy of the Variance under Short-Range Dependence}
\runtitle{Asymptotic Test for Constancy of the Variance}

\begin{aug}

\author[A]{\fnms{Sara K.} \snm{Schmidt,}\ead[label=e1]{sara.schmidt@ruhr-uni-bochum.de}}
\author[B]{\fnms{Max} \snm{Wornowizki,}\ead[label=e2]{wornowizki@statistik.tu-dortmund.de}}
\author[B]{\fnms{Roland} \snm{Fried}\ead[label=e3]{fried@statistik.tu-dortmund.de}}
\and
\author[A]{\fnms{Herold} \snm{Dehling}\ead[label=e4]{herold.dehling@ruhr-uni-bochum.de}}

\address[A]{Department of Mathematics, Ruhr-Universit\"at Bochum,
  Universit\"atsstra\ss e 150, 44780 Bochum, Germany, \printead{e1,e4}}

\address[B]{Department of Statistics, TU Dortmund University, Vogelpothsweg 87, 44221 Dortmund, Germany, \printead{e2,e3}}

\end{aug}

\begin{abstract}
We present a novel approach to test for heteroscedasticity of a non-stationary time series that is based on Gini's mean difference of logarithmic local sample variances. In order to analyse the large sample behaviour of our test statistic, we establish new limit theorems for U-statistics of dependent triangular arrays. We derive the asymptotic distribution of the test statistic under the null hypothesis of a constant variance and show that the test is consistent against a large class of alternatives, including multiple structural breaks in the variance. Our test is applicable even in the case of non-stationary processes, assuming a locally stationary mean function. The performance of the test and its comparatively low computation time are illustrated in an extensive simulation study. As an application, we analyse Google Trends data, monitoring the relative search interest for the topic ``global warming.'' 
\end{abstract}



\end{frontmatter}

\section{Introduction}

Constancy of the variance is a common assumption and several authors have proposed tests for it.
Wichern, Miller and Hsu \citep{Wichern.1976}, Abraham and Wei \citep{AbrahamWei.1984} and Baufays and Rasson \citep{BaufaysRasson.1985} do so
in the parametric framework of autoregressive models.
Incl\'{a}n and Tiao \citep{Inclan.1994} and Gombay, Horv\'{a}th and Hu\v{s}kov\'{a} \citep{Gombay.1996} propose nonparametric tests based on cumulative sums of squares against the alternative of a single structural break in a sequence of independent data.
 Lee and Park \citep{Lee.2001} and Wied et al. \citep{Wied.2012} extend this work to time series data fulfilling different short-range dependence conditions.  Gerstenberger, Vogel and Wendler \citep{Gerstenberger.2019} use Gini's mean difference instead of sums of squares to test against the same family of alternatives. Galeano and Pe\~{n}a \citep{GaleanoPena.2007} and Aue et al. \citep{Aue.2009} consider the multivariate case. Chen and Gupta \citep{ChenGupta.1997} combine binary segmentation and the Schwarz information criterion for detection of multiple change-points in independent Gaussian data.

In these and other papers, the mean of the data is assumed to be constant. Tests of the stationarity of the variance in the presence of a time-varying mean have been derived only recently by Dette, Wu and Zhou \citep{Dette.2019} and  by Gao et al. \citep{Gao.2019}. While the latter authors assume independent Gaussian observations, the former ones apply wild bootstrap to derive critical values because the covariance structure of the limiting Gaussian process depends in a complex manner on the dependencies in the data generating process.

 We construct a nonparametric asymptotic test for the constancy of the variance against the alternative of one or several change-points, allowing the data to be short-range dependent and the mean to be possibly time-varying.  The statistic underlying our test has been proposed by Wornowizki, Fried and Meintanis \citep{Wornowizki.2017}, who use the permutation principle to test the constancy of the variance in a sequence of independent observations.  They illustrate the advantages of this test statistic over several competitors in a series of simulation experiments.

For the construction of our asymptotic test, we develop new theory under the 
assumption that we observe time series data $X_1,\ldots,X_n$ generated by a particular locally stationary process in the sense of Dahlhaus \citep{Dahlhaus.1997}. Formally, we work with a triangular array
\begin{equation}
\label{Eq: Data model}
  X_i:=X_{i,n}=\sigma({i}/{n}) Y_i +\mu({i}/{n}), 
\end{equation}
where $(Y_i)_{i\in\N}$ is a stationary $\beta$-mixing process with mean zero and variance one. The local means and variances are described by the functions $\mu:[0,1]\rightarrow \R$ and $\sigma^2:[0,1]\rightarrow (0,\infty)$, respectively. 
We will test the null hypothesis $\sigma(x)\equiv \sigma_H$ against the alternative that $\sigma$ is a non-constant c\`adl\`ag function. 
We show that under some mild regularity conditions on $\mu$ and $\sigma$, this test is consistent against all non-constant variance functions. In contrast, tests that are specifically designed for the alternative of a single change-point in the variance, as modelled by the local variance function $\sigma^2(x)=\sigma_1^2 \1_{[0,\tau)}(x) +\sigma_2^2 \1_{[\tau,1]} (x)$, will hardly be consistent against arbitrary alternatives as considered here.

A model of the above type \eqref{Eq: Data model} is quite common, for instance in the context of non-parametric regression (see, for instance Wu and Zhao \cite{Wu.2007}),  where the mean function is usually assumed to be Lipschitz-continuous, possibly with jumps, and where the variance is required to be stationary. In this regard, conducting our test can be understood as a preliminary step in determining whether the latter model assumption is met. We develop our test in the first place under the assumption of a Lipschitz-continuous mean function $\mu$, and discuss an extension to piecewise Lipschitz-continuous mean functions using the differenced time series later in Section 4. While our test for constant variance is novel even under the restrictive assumption of a constant mean, we thus allow for more realistic scenarios where the mean function is nearly constant on a small time scale $o(n)$, but possibly non-constant on larger time scales.

More precisely, our test is based on the statistic $U(n)$, which takes the form of Gini's mean difference $\frac{1}{b_n(b_n-1)}\sum_{1\leq j\neq k \leq b_n} \left| \nu_j-\nu_k\right|$ of the logarithmic local sample variances $\nu_j=\log(\hat{\sigma}_j^2)$. Specifically, we subdivide the time interval $\{1,\ldots,n\}$ into blocks of length $\ell_n$, the $j$-th block being given by $\{(j-1)\ell_n+1,\ldots,j\, \ell_n\}$. Let $b_n$ denote the number of full blocks that fit into $\{1,\ldots,n\}$, i.e., $b_n=\lfloor n/\ell_n\rfloor$, and $\hat{\sigma}_j^2$ be the local sample variance in the $j$-th block,
\[
  \hat{\sigma}^2_{j,n}=\hat{\sigma}^2_j =\frac{1}{\ell_n} \sum_{i=(j-1) \ell_n+1}^{j\ell_n}
\left(X_i - \frac{1}{\ell_n}\sum_{r=(j-1) \ell_n+1}^{j\ell_n} X_r \right)^2.
\]
Using this notation, the statistic $U(n)$ reads
\begin{equation}\label{teststatistic}
  U(n)=\frac{1}{b_n(b_n-1)} \sum_{1\leq j\neq k \leq b_n} \left|\log \hat{\sigma}_j^2 -\log \hat{\sigma}_k^2 \right|.
\end{equation}
The use of the log-transformed local variances makes $U(n)$ scale invariant.
Note that, on a broader level, $U(n)$ constitutes a $U$-statistic with kernel $h(x,y)=|x-y|$, whose entries are given by the triangular array $\log \hat{\sigma}_{j,n}^2$, $1\leq j\leq b_n$. These entries, after proper centering and scaling, converge in distribution to a normal law. We develop U-statistic theory for this type of triangular arrays in order to show that $U(n)$ is asymptotically normal under the null hypothesis. Our results hold for absolutely regular processes, and thus cover a large class of time series models.
Absolute regularity is also known under the term $\beta$-mixing and is a slightly stronger assumption than the well-known ``strong mixing''-condition. Still, it covers a wide range of examples
such as certain classes of Markov chains, stationary non-degenerate Gaussian processes with a particular form of the spectral density  (see, Bradley \citep{Bradley.2005}) as well as ARMA- and GARCH-models (see, Example \ref{Example: ARMA, GARCH} below).

Our test is computationally feasible even in case of huge data sets since we analyse deviations between local statistics and compare them to critical values calculated from the asymptotical distribution derived in this paper. Its low computation time is confirmed in an extensive simulation study. Considering a wide range of data generating processes and alternatives, we find our test to have good finite sample properties and to be especially well-suited in case of multiple structural breaks or non-monotone variation of the variance function.  In particular, we compare our procedure to the approach in Dette, Wu and Zhou \citep{Dette.2015}, \citep{Dette.2019}. Moreover, we illustrate that our test performs well for both types of mean scenarios, Lipschitz-continuous mean functions and mean functions with jumps. As an application, we use our test to detect periods of different volatility in the annual increments of the relative search interest for the topic ``global warming''  retrieved from Google Trends.

The rest of the paper is structured as follows: Section \ref{Sec: Test statistic and main results} covers some central definitions, presents the key asymptotic results for the test statistic $U(n)$ and outlines the estimation of the long run variance in our particular setting. The main ideas of the proof are sketched in Section \ref{Sec:Outline and Main Ideas of Proofs}. Section \ref{Sec: Extensions} treats several extensions of our theory, among which are possible modifications of the test statistic, an application to data with jumps in the mean and the estimation of the change-point locations. The results of the simulation study and the data example are reported in Sections \ref{Sec: Simualtion Study} and \ref{Sec: Data Example}, respectively. All further proofs and some additional simulation results are deferred to an appendix. 

\section{Main Results}
\label{Sec: Test statistic and main results}

\subsection{Basic Definitions and Assumptions}

Throughout the paper, we assume that $n$ observations, $X_1, ..., X_n$, generated by the model \eqref{Eq: Data model} are given, i.e. $X_i=\sigma(i/n)Y_i+\mu(i/n)$, where $(Y_i)_{i\in\N}$ is a short-range dependent or, more precisely, an absolutely regular stationary process. In the following, we assume $$\mu:[0,1]\rightarrow \R$$ to be a Lipschitz-continuous mean function and $$\sigma:[0,1]\rightarrow [\sigma_0,\infty)$$ for some $\sigma_0>0$ to be a c\`adl\`ag-function. We want to test the hypothesis of a constant variance, i.e. $\sigma\equiv \sigma_H$, against the alternative of a non-constant variance function. Note that under the null hypothesis, we have $\E{X_i}=\mu\rbraces{\frac{i}{n}}$ and $\Var{X_i}=\sigma_H^2$. 

Lipschitz-continuity of the mean function $\mu$ implies that the means $\E{X_i}=\mu(i/n)$ are nearly constant on a small time scale $o(n)$, as for $|i-r| =o(n)$, we have $\mu(i/n)-\mu(r/n)=o(1)$. In contrast, the mean function will generally be non-constant on the larger time scale $O(n)$, as 
$\E{X_{ c_1\cdot n}} - \E{X_{ c_2 \cdot n}}  =\mu(c_1)-\mu(c_2)$ for $c_1, c_2\in [0,1]$. In fact, one can even relax the assumption on $\mu$ to Hölder-continuity, see Remark \ref{Remark: Hölder-ctd} below, but we will restrict ourselves to the notationally more feasible case of Lipschitz-continuity.  In addition, we will later also present an approach that can handle piecewise Lipschitz-continuous mean functions by considering first order differences of the data.

The dependence structure of $(X_i)_{i\in\N}$ is determined by the underlying process $(Y_i)_{i\in \N}$, which is assumed to be absolutely regular. Such processes are also known under the term $\beta$-mixing.

\begin{definition}{
A sequence of random variables $\rbraces{Y_i}_{i\in \N}$ is called absolutely regular if
\begin{equation*}
\beta_Y(k):= \sup_{m\in\N} \beta\rbraces{\sigma(Y_i, 1\leq i\leq m), \sigma(Y_i, k+m\leq i\leq \infty)} \rightarrow 0 \text{ as } k \rightarrow \infty,
\end{equation*}
where the $\beta$-mixing coefficient of two $\sigma$-fields $\mathcal{A}$ and $\mathcal{B}$ is given by
\begin{equation*}
\beta(\mathcal{A},\mathcal{B}):= \E{\underset{A\in \mathcal{A}}{\mathrm{ess}\,\sup}\abs{\Pb\rbraces{A|\mathcal{B}}-\Pb(A)}}.
\end{equation*}
}\end{definition}

Regarding the moments and the mixing rate of the process $(Y_i)_{i\in\N}$, we assume that there exists a $\vartheta>0$ such that the following two conditions are satisfied
 \begin{align*}
\text{(A1)} \quad &\E{\abs{Y_1}^{4+2\vartheta}}<\infty  \\
 \text{(A2)} \quad &\sum_{k=1}^{\infty} \beta_Y(k)^{\vartheta/(2+\vartheta)}<\infty.
\end{align*}
Given these assumptions,  the long run variance
\begin{equation*}
{\kappa}^2:=\Var{Y_1^2} +2 \sum_{k=1}^\infty \cov{Y_1^2}{Y_{k+1}^2}
\end{equation*}
is finite since $\beta_{Y^2}(k)\leq \beta_Y(k)$ (see, Theorem 10.7 in Bradley \citep{Bradley.2007}). Throughout, we will assume the long run variance to be strictly positive, ${\kappa}^2>0$.

Most of the classical models in time series analysis satisfy the above assumptions. The following example points out some suitable processes that meet all the conditions required to derive the asymptotic results below.
\begin{example}{
\label{Example: ARMA, GARCH}
\begin{enumerate}
\item Let $(Y_i)_{i\in\N}$ be a strictly stationary, causal ARMA(p,q)-process following the model
$$Y_i= \varepsilon_i + \sum_{j=1}^p \alpha_jY_{i-j} +\sum_{m=1}^q \beta_m \varepsilon_{i-m}$$
with independent innovations $(\varepsilon_i)_{i\in\N}$ and with all roots of  $1-\sum_{j=1}^p \alpha_j z^j$ being larger than one in absolute value. Additionally, assume that
the AR-and MA-polynomials have no common roots, and that the innovations $(\varepsilon_i)_{i\in \N}$ have an absolutely continuous distribution with respect to the Lebesgue measure. Then, $(Y_i)_{i\in\N}$ is absolutely regular with a geometric rate, i.e., $\beta(k)=O(e^{-\xi k})$ for some $\xi>0$ (see, Theorem 1 in Mokkadem \citep{Mokkadem.1988}).
\item Strictly stationary GARCH(p,q)-processes are strictly stationary solutions to the equations
$$Y_i=\sigma_i\varepsilon_i\quad \text{and} \quad\sigma_i^2=\alpha_0+ \sum_{j=1}^p\alpha_j Y_{i-j}^2+\sum_{m=1}^q \beta_m \sigma_{i-m}^2.$$
They are likewise absolutely regular with a geometric rate if the i.i.d. noise sequence $(\varepsilon_i)_{i\in \N}$ has finite absolute $r$th moment for some $r\in(0,\infty)$, i.e., $\E{\abs{\varepsilon_1}^r}<\infty$, and if  $\varepsilon_1$ has an absolutely continuous distribution with a density that is strictly positive in a neighbourhood of zero (see, Lindner \citep{Lindner2009} and the references therein for this result and an overview on the existence of strictly stationary solutions and the existence of moments for GARCH-models).
\end{enumerate}
}\end{example}

 The aim of this paper is to test for changes in the variance of  the time series $(X_i)_{i\in\N}$ by means of the statistic $  U(n)=\frac{1}{b_n(b_n-1)} \sum_{1\leq j\neq k \leq b_n} |\log \hat{\sigma}_j^2 -\log \hat{\sigma}_k^2 |$ from \eqref{teststatistic}, which compares the local estimates $\hat{\sigma}_j^2$, $j=1,\ldots,b_n$, derived from splitting the data into $b_n$ blocks of length $\ell_n$. Both $b_n$ and $\ell_n$ are assumed to grow with the sample size and, for simplicity, to be integers. Moreover, we  have to impose certain growth restrictions on the block length $\ell_n$ and therewith the number of blocks $b_n$ in order to ensure the desired convergence of the test statistic. Throughout the paper, we will need the block length $\ell_n$ to grow faster than the number of blocks. In particular, if we set $\ell_n=n^s$ for some $s\in(0,1)$ and consequently $b_n=n^{1-s}$, this translates to $s>0.5$.
A change in the variance should result in large differences between the block-estimates $\hat{\sigma}_j^2$ and ultimately in a high value of $U(n)$, which will lead to a rejection of the hypothesis.

\subsection{Asymptotical Results}

  First, we show the convergence of the statistic $U(n)$ defined in \eqref{teststatistic} towards a two-dimensional Riemann-integral.
\begin{theorem}{
\label{THM: Behaviour under A for U(n)}
Given the assumptions (A1) and (A2) and if $\ell_n=n^s$ with $s\in (0.5,0.75)$, it holds
\begin{equation*}
U(n) \pConv \int_0^1 \int_0^1 \abs{\log \sigma^2(x)-\log \sigma^2(y)} \mathrm{d}x\mathrm{d}y \qquad \text{as } n \rightarrow \infty.
\end{equation*}
}\end{theorem}
Obviously, $U(n)$ converges to 0 in probability under the null hypothesis, while it converges to a strictly positive value under any alternative for which
$\sigma$ is not almost surely constant with respect to the Lebesgue measure on $[0,1]$.
This will imply that our test is consistent against the very general alternative of a changing variance function.

By standardizing the statistic $U(n)$ via the long run variance ${\kappa}^2$ and by using an appropriate scaling $\sqrt{\ell_n}$, we can now state a law of large numbers for $U(n)$ under the null hypothesis.

\begin{theorem}{
\label{Theorem: LLN Un}
Let the above assumptions (A1) and (A2) be fulfilled and let $\ell_n=n^s$ with $s\in (0.5, 0.75)$. If there exists a sequence $m_n\rightarrow\infty$ as $n\rightarrow \infty$ such that $m_n=o(n^{2s-1})$ and $b_n\beta_Y(m_n)\rightarrow 0$, then it holds under the null hypothesis 
\begin{equation*}
\frac{\sqrt{\ell_n}}{\kappa}U(n)\pConv \frac{2}{\sqrt{\pi}} \quad \text{as}\quad n \rightarrow \infty.
\end{equation*}

}\end{theorem}
Theorem \ref{Theorem: LLN Un} already reveals the double asymptotics governing the statistic $U(n)$. While there holds a central limit theorem for the inner block sums $\hat{\sigma}_j$ for which we need the scaling $\sqrt{\ell_n}$, there holds a law of large numbers for the outer structure of a U-statistic. Under the additional assumption of polynomially decaying mixing coefficients, a central limit theorem for the outer U-statistic holds as well.

\begin{theorem}{
\label{Theorem: CLT Un}
Assume there exist constants $0<\delta\leq 1$ and $\rho>1\vee \frac{9\delta}{(\delta+1)(\delta+2)}$  such that $\E{\abs{Y_1}^{4+2\delta}}<\infty$ and for all $k\in \N$ it holds $\beta_Y(k)\leq C k^{-\rho (2+\delta)(1+\delta)/\delta^2}$. Moreover, choose $\ell_n=n^s$ with
$s\in (0.5, 0.75)$ as well as $s>1/\rbraces{1+\delta\frac{\rho-1}{\rho+1}}\vee (1+\frac{\delta^2}{\rho(2+\delta)(1+\delta)})/(2+\frac{\delta^2}{\rho(2+\delta)(1+\delta)})$. Then it holds under the null hypothesis
\begin{equation*}
\sqrt{b_n}\rbraces{\frac{\sqrt{\ell_n}}{\kappa}U(n)- \frac{2}{\sqrt{\pi}}}\distConv \NoD{0}{\frac{4}{3}+\frac{8}{\pi}\rbraces{\sqrt{3}-2}} \quad \text{as}\quad n \rightarrow \infty.
\end{equation*}
}\end{theorem}

\begin{example}
Coming back to Example \ref{Example: ARMA, GARCH}, consider the ARMA(p,q)- and GARCH(p,q)-models presented there with 
independent standard normally distributed innovations. Their mixing coefficients decay at a geometric rate, $\beta(k)=O(e^{-\xi k})$ for some $\xi>0$, which corresponds to the border case $\rho\rightarrow \infty$ in Theorem \ref{Theorem: CLT Un}. The conditions incorporating $\rho$ thus boil down to $s>1/(1+\delta)$. Put differently, we may choose the tuning parameter $s\in (0.5, 0.75)$ and have to ensure $\E{\abs{Y_1}^{2+2/s+\varepsilon}}<\infty$ for some $\varepsilon>0$ for Theorem \ref{Theorem: CLT Un} to hold. 
\end{example}
Theorem \ref{Theorem: CLT Un} requires an additional outer scaling factor $\sqrt{b_n}$ depending on the number of blocks for the U-statistic central limit theorem to hold. Referring to the theory of U-statistics, one can then derive convergence towards a normal distribution whose variance $\psi^2=4\Var{h_1(Z)}$ with $h_1(x):=\E{\abs{x-Z'}}-2/\sqrt{\pi}$ for two independent standard normally distributed random variables $Z$ and $Z'$ equals the limit $\limn \Var{\sqrt{n}\tilde{U}_n}$, where $\tilde{U}_n$ denotes Gini's mean difference computed from a sample of $n$ iid standard normal observations. The latter limit can be calculated explicitly, see Gerstenberger and Vogel  \citep{Gerstenberger.2015}, such that $\psi^2=\frac{4}{3}+\frac{8}{\pi}\rbraces{\sqrt{3}-2}$.

In particular, Theorem \ref{Theorem: CLT Un}  can be used for structural break testing. Based on the data $x_1, ..., x_n$, one can compute the value of the properly standardized test statistic and compare it to the asymptotic critical values obtained from the limit distribution. We shall reject the hypothesis of a constant variance if the computed value exceeds the $(1-\alpha)$-quantile of the $\NoD{0}{\frac{4}{3}+\frac{8}{\pi}\rbraces{\sqrt{3}-2}}$-distribution. Note that this test is consistent against arbitrary alternatives for which $\sigma(x)$ is not almost surely constant with respect to the Lebesgue measure on $[0,1]$. A key factor is the reliable estimation of the unknown long run variance $\kappa^2 $, which is discussed in Subsection \ref{Subsec: Estimation of LRV}.

\subsection{Estimation of the Long-run Variance}
\label{Subsec: Estimation of LRV}

For a practical implementation of our test, we need to estimate the long-run variance
\begin{equation*}
\kappa^2=\Var{Y_1^2} +2 \sum_{k=1}^\infty \cov{Y_1^2}{Y_{k+1}^2}.
\end{equation*}
In the literature, there are various procedures for estimating such long-run variances. We employ the subsampling approach, introduced by Carlstein \citep{Carlstein.1986}. We will use the relation $\E{\abs{\frac{1}{\sqrt{n}}\sum_{i=1}^n \rbraces{Y_i^2-\E{Y_1^2}}}}\rightarrow \kappa \sqrt{\frac{2}{\pi}}$ for the construction of our estimator. For standard subsampling, we divide the observations into $\bn$ non-overlapping blocks of length $\elln$, and consider the estimator
\[
  \frac{1}{\bn}\sqrt{\frac{\pi}{2}} \sum_{j=1}^{\bn}
\left| \frac{1}{\sqrt{\elln}} \sum_{i=(j-1) \elln+1}^{j \elln} (Y_i^2 -\frac{1}{n}\sum_{r=1}^n Y_r^2 )   \right|.
\]
Various authors have established consistency of this estimator for a wide class of short-range dependent data.

As we do not observe the $Y_i^2$ directly, but only $X_i=\sigma(\frac{i}{n})Y_i +\mu(\frac{i}{n})$, we need to modify the standard subsampling procedure. We will use a subsampling that is consistent under the null hypothesis $\sigma(x)\equiv \sigma_H$, i.e., when the observations are given by $X_i=\sigma_H Y_i +\mu(\frac{i}{n})$. We first center the observations by their local means, defining
\[
  \tilde{X}_i =X_i -\frac{1}{\ell_n} \sum_{r=(j-1) \ell_n+1}^{j\ell_n} X_r, \; \mbox{ for } i\in \{(j-1)\ell_n+1,\ldots, j\ell_n    \}.
\]
Setting $\hat{\sigma}_H^2 =\frac{1}{n}\sum_{i=1}^n \tilde{X}_i^2$, we then define the subsampling long-run variance estimator
\[
  \hat{\kappa} :=\frac{1}{\bn} \sqrt{\frac{\pi}{2}} \frac{1}{\hat{\sigma}_H^2}
 \sum_{j=1}^{\bn} \left| \frac{1}{\sqrt{\elln}}\sum_{i=(j-1)\elln+1}^{j \elln }
(\tilde{X}_i^2-\hat{\sigma}_H^2) \right|.
\]

The next proposition shows that consistency of the above estimator $\hat{\kappa}$ indeed remains valid  given the additional scaling factor $\sqrt{b_n}$.

\begin{proposition}
\label{Prop: Behaviour LRV estimator H}
Assume that there exist constants  $0<\delta\leq 1$ and $\rho>1$ such that $\E{\abs{Y_1}^{4+2\delta}}<\infty$ and for all $k\in \N$ it holds $\beta_Y(k)\leq C k^{-\rho (2+\delta)(1+\delta)/\delta^2}$. Moreover, let $\ell_n=n^s$ and $\elln=n^q$ such that $s>0.5$, $1-s<q\delta (\rho-1)/(\rho+1)$,  $q<s$ and $q<3(1-s)$. 
 Then it holds under the null hypothesis
$$\sqrt{b_n}\abs{\hat{\kappa}-\kappa}\pConv 0 \quad \text{as } n \rightarrow \infty.$$
\end{proposition}

By Proposition \ref{Prop: Behaviour LRV estimator H}, we can replace the long run variance $\kappa^2$ in the central limit theorem \ref{Theorem: CLT Un} by its estimator:

\begin{corollary}{
\label{Corollary: Testing procedure H}
Assume there exist constants  $0<\delta\leq 1$ and $\rho>1\vee \frac{9\delta}{(\delta+1)(\delta+2)}$ such that $\E{\abs{Y_1}^{4+2\delta}}<\infty$ and for all $k\in \N$ it holds $\beta_Y(k)\leq C k^{-\rho (2+\delta)(1+\delta)/\delta^2}$. Moreover, let $\ell_n=n^s$ and $\elln=n^q$ such that $s\in (0.5, 0.75)$, $s>1/\rbraces{1+\delta\frac{\rho-1}{\rho+1}}\vee (1+\frac{\delta^2}{\rho(2+\delta)(1+\delta)})/(2+\frac{\delta^2}{\rho(2+\delta)(1+\delta)})$, $1-s<q\delta (\rho-1)/(\rho+1)$,  $q<s$ and $q<3(1-s)$.
 Then it holds under the null hypothesis
\begin{equation*}
\sqrt{b_n}\rbraces{\frac{\sqrt{\ell_n}}{\hat{\kappa}}U(n)- \frac{2}{\sqrt{\pi}}}\distConv \NoD{0}{\frac{4}{3}+\frac{8}{\pi}\rbraces{\sqrt{3}-2}} \quad \text{as}\quad n \rightarrow \infty.
\end{equation*}
}\end{corollary}

Long run variance estimators often have the drawback of diverging under the alternative. This is unfortunate since the long run variance serves in the denominator as a standardization and its overestimation thus results in a non-negligible loss of power. Still, the following proposition ensures that the  divergence here is not fast enough to cancel out the growth rate associated with the test statistic in the numerator.

\begin{proposition}{
\label{Prop: Behaviour LRV estimaor A}
Let the assumptions (A1) and (A2) be fulfilled and assume $q<s$ as well as $q<4(1-s)$ for $\ell_n=n^s$ and $\elln=n^q$. Then it holds
$$\frac{1}{\sqrt{\elln}} \hat{\kappa} \pConv \frac{\int_0^1 \abs{\sigma^2(x)-\int_0^1 \sigma^2(y)\mathrm{d}y}\mathrm{d}x}{\int_0^1 \sigma^2(z) \mathrm{d}z} \quad \text{as } n \rightarrow \infty.$$
}\end{proposition}

\begin{corollary}{
\label{Corollary: Testing procedure A}
Given the assumptions (A1) and (A2) and if $\ell_n=n^s$ with $s\in (0.5,0.75)$ as well as $\elln=n^q$ with $q<s$, it holds
 $$\sqrt{\elln} \frac{U(n)}{\hat{\kappa}}\pConv\frac{\int_0^1 \int_0^1 \abs{\log \sigma^2(x)-\log \sigma^2(y)} \mathrm{d}x\mathrm{d}y \cdot\int_0^1 \sigma^2(z) \mathrm{d}z}{\int_0^1 \abs{\sigma^2(x)-\int_0^1 \sigma^2(y)\mathrm{d}y}\mathrm{d}x} \quad \text{as } n\rightarrow \infty.$$
}\end{corollary}

Due to $\elln=o(\ell_n)$, the expression within the central limit theorem thus still diverges under the alternative, though at a slower rate.

\section{Outline and Main Ideas of the Proofs}
\label{Sec:Outline and Main Ideas of Proofs}
This section outlines the key ideas to prove our main results, whereas the technical details are given in Appendix \ref{Appendix: Proofs}. 
 The proofs rely on a series of approximations of the statistic
\[
 U(n)=\frac{1}{b_n(b_n-1)}\sum_{1\leq j\neq k \leq b_n} \left|\log \hat{\sigma}_j^2 -\log \hat{\sigma}_k^2 \right|
\]
by simpler statistics that are easier to analyze and have the same large sample behavior as $U(n)$. In a first step, we replace the block means $\frac{1}{\ell_n}\sum_{i=(j-1)\ell_n +1}^{j\ell_n} X_i$ in the definition of the local sample variance $\hat{\sigma}_j^2$ by the expected values $\mu(i/n)$. For the resulting U-statistic $U_1(n)$ one can then derive an analogue of Theorem \ref{THM: Behaviour under A for U(n)}.

To obtain additional limit theory under the null hypothesis in Theorems \ref{Theorem: LLN Un} and \ref{Theorem: CLT Un}, two more approximations are required. In a second approximation step, we linearize the logarithm, i.e. we use the approximation $\log(1+x)\approx x$, which is valid for $x$ close to $0$. Finally, employing a coupling technique for absolutely regular observations (see, Borovkova, Burton and Dehling \citep{Borovkova.2001}), we replace the slightly shortened dependent blocks
$(X_{(j-1)\ell_n+1},\ldots, X_{j\ell_n-m_n})$ by close-by independent blocks $(X_{(j-1)\ell_n+1}^\prime,\ldots, X_{j\ell_n-m_n}^\prime)$  with the same marginal distributions. The resulting U-statistic $U_3(n)$ can be analyzed by an adaptation of U-statistic theory to triangular arrays since its entries stem from a triangular array of row-wise independent random variables.

\subsection{A first Approximation}
Centering the observations $X_i$ in each block $j$ by their expected value $\mu(i/n)$ instead of the block mean $\frac{1}{\ell_n}\sum_{r=(j-1) \ell_n+1}^{j\ell_n} X_r$, we obtain the following approximation to the empirical block variances
\[
s_{j,n}^2:= s_j^2:=\frac{1}{\ell_n} \sum_{i=(j-1) \ell_n +1}^{j\ell_n} \rbraces{X_i-\mu\rbraces{\frac{i}{n}}}^2
=\frac{1}{\ell_n}\sum_{i=(j-1) \ell_n +1}^{j\ell_n} \sigma^2\rbraces{\frac{i}{n}}\, Y_i^2.
\]
Taking Gini's mean difference of the $\log s_j^2$, we get the statistic
\[
 U_1(n)=\frac{1}{b_n(b_n-1)} \sum_{1\leq j\neq k \leq b_n}\left| \log s_j^2-\log s_k^2  \right|.
\]
Our first approximation theorem shows that $U_1(n)$ is sufficiently close to $U(n)$ for all limit theorems to carry over.
\begin{proposition}
\label{Prop: Approximation mit U_1 unter H und A}
Assume that conditions (A1) and (A2) hold and that $\ell_n=n^s$ with $s\in (0.5,0.75)$. Then, we have
\[
  \sqrt{n}\abs{U(n)-U_1(n)} \pConv 0,
\]
as $n\rightarrow \infty$.
\end{proposition}

Proposition \ref{Prop: Approximation mit U_1 unter H und A} holds both under the null hypothesis as well as under the alternative. Thus, we may henceforth restrict our analysis to the U-statistic $U_1(n)$, and to the centered data
$X_i-\mu(i/n) =\sigma(i/n) Y_i$. Without loss of generality, we may assume from now on that $\mu(i/n)=0$, and that the data are given by
\[
   X_i=\sigma\rbraces{ \frac{i}{n}} Y_i.
\]

\subsection{Outline of the Proof of Theorem \ref{THM: Behaviour under A for U(n)}}
Recall that the statistic $U_1(n)$ employs the arguments $\log s_j^2$ for $1\leq j\leq b_n$, where
$s_j^2 =\frac{1}{\ell_n}\sum_{i=(j-1)\ell_n +1}^{j\ell_n} \sigma^2(\frac{i}{n}) Y_i^2$. Their mean can be approximated via

\[
  Es_j^2=\sum_{i=(j-1)\ell_n+1}^{j\ell_n} \sigma^2\rbraces{\frac{i}{n}}\approx \sigma^2\rbraces{\frac{j\ell_n}{n}}=\sigma^2\rbraces{\frac{j}{b_n}}
\]
and one can moreover show that $\Var{s_j^2}\rightarrow 0$ as $n\rightarrow \infty$. Hence, we set $s_j^2\approx \sigma^2\rbraces{j/b_n}$ and this in turn implies
\[
  U_1(n)\approx \frac{1}{b(b_n-1)} \sum_{1\leq j\neq k \leq b_n} \abs{\log \sigma^2\rbraces{ \frac{j}{b_n} }-
\log \sigma^2\rbraces{\frac{k}{b_n} }}.
\]
The latter sum is a Riemann-type approximation of the desired limit integral
$\int_0^1\int_0^1 |\log \sigma^2(x) -\log \sigma^2(y)| dx dy$. A rigorous proof is given in the appendix. %

\subsection{Two further Approximations}

The second and third approximation are essential ingredients in the analysis of the asymptotic behavior of $U(n)$ under the null hypothesis, i.e. when $\sigma(x)\equiv \sigma_H$.  In this case, we have
$s_j^2=\sigma_H^2 \frac{1}{\ell_n}\sum_{i=(j-1)\ell_n+1}^{j\ell_n} Y_i^2$ and hence
\begin{align*}
  \abs{ \log s_j^2 -\log s_k^2 } & =
  \Big| \log \Big( \frac{1}{\ell_n} \sum_{i=(j-1)\ell_n+1}^{j\ell_n} Y_i^2\Big) - \log \Big(\frac{1}{\ell_n} \sum_{i=(k-1)\ell_n+1}^{k \ell_n} Y_i^2 \Big) \Big| \\
&= \big| \log (1+S_j) -\log(1+S_k)  \big|,
\end{align*}
where we have defined
$$S_j:=S_{j,n}:=\frac{1}{\ell_n}\sum_{i=(j-1)\ell_n+1}^{j\ell_n}Y_i^2 -1 =\frac{s_j^2}{\sigma_H^2}-1.$$
Since $E(Y_i)=0$ and $\Var{Y_i}=1$, the law of large  numbers implies that for a sufficiently large sample size, $S_j$ is close to $0$. This motivates a Taylor expansion of $\log(1+x)$ around $x=0$. Replacing $\log(1+S_j)$ by $S_j$ in the definition of $U_1(n)$ yields the statistic
\[
  U_2(n) :=\frac{1}{b_n(b_n-1)} \sum_{1\leq j\neq k\leq b_n} \abs{S_j-S_k}.
\]
\begin{proposition}
\label{Prop: Approx U1 U2}
Assume that conditions (A1) and (A2) hold, and that $\ell_n=n^s$ with $s>0.5$. Then, under the null hypothesis,
\[
  \sqrt{n}(U_1(n)-U_2(n)) \pConv 0,
\]
as $n\rightarrow \infty$.

\end{proposition}

In a third and final approximation step, we replace the arguments $S_j$, $1\leq j \leq b_n$, of $U_2(n)$ by close-by independent random variables $S_j'$, $1\leq j\leq b_n$. In order to achieve this, we use a coupling technique for $\beta$-mixing processes. First, we divide each of the blocks $B_{j,n}=(Y_{(j-1)\ell_n+1},\ldots, Y_{j\ell_n})$ into a long block
\[
\tilde{B}_j:=\tilde{B}_{j,n}:= (Y_{(j-1)\ell_n +1},\ldots, Y_{j\ell_n-m_n})
\]
of length $\ell_n-m_n$ and a short block
\[
R_j:=R_{j,n}:= (Y_{j\ell_n-m_n+1},\ldots,Y_{j\ell_n})
\]
of length $m_n=o(\ell_n)$. The latter blocks function as separation between the main blocks  whose interdependence decreases when $m_n\rightarrow \infty$, as $n\rightarrow \infty$. At the same time, the separating blocks $R_j$ need to be sufficiently short to be asymptotically negligible compared to the longer blocks $\tilde{B}_j$, and thus $m_n$ should grow only slowly.
By Lemma 2.4 in Borovkova, Burton and Dehling \citep{Borovkova.2001}, there exists a sequence of i.i.d.  random vectors
\[
 \tilde{B}_{j,n}':= \tilde{B}_j':=(Y^\prime_{(j-1)\ell_n+1},\ldots,Y_{j\ell_n-m_n}^\prime  )
\]
with the same marginal distribution as $\tilde{B}_j$ such that
\[
  P(\tilde{B}_j^\prime = \tilde{B}_j)=1-\beta_Y(m_n),
\]
for all $1\leq j\leq b_n$. Define the corresponding block sums
\[
 \tilde{S}_{j,n}^\prime:=  \tilde{S}_j^\prime:=\frac{1}{\ell_n} \sum_{i=(j-1) \ell_n+1}^{j\ell_n-m_n} ((Y_i^\prime)^2-1),
\]
and note that for any $n$, the random variables $\tilde{S}_j^\prime$, $1\leq j\leq b_n$, are independent and identically distributed. Finally, we consider the U-statistic
\[
  U_3(n):=\frac{1}{b_n(b_n-1)} \sum_{1\leq j\neq k\leq b_n} |\tilde{S}_j^\prime-\tilde{S}_k^\prime|.
\]

\begin{proposition}
\label{Prop: Approximation U2 U3}
Assume that $\ell_n=n^s$ with  $s>0.5$, and that $m_n$ is chosen in such a way that $m_n=o(n^{2s-1})$ as well as $b_n\beta(m_n) \rightarrow 0$. Then,  under the null hypothesis,
\[
  \sqrt{n} |U_2(n)-U_3(n)| \pConv 0,
\]
as $n\rightarrow \infty$.
\end{proposition}

\begin{remark}{The exact choice of $m_n$ is insignificant for the structural break testing. For polynomially decaying mixing coefficients, i.e. $\beta_Y(k)\leq Ck^{-\rho(2+\delta)(1+\delta)/\delta^2}$, the conditions $m_n=o(n^{2s-1})$ and $b_n\beta_Y(m_n)\rightarrow 0$ are met as long as one chooses
 $s>(1+\frac{\delta^2}{\rho(2+\delta)(1+\delta)})/(2+\frac{\delta^2}{\rho(2+\delta)(1+\delta)})$.
}\end{remark}

\subsection{Outline of the Proof of Theorem \ref{Theorem: LLN Un}}
Given the above approximations, under the null hypothesis, it suffices to analyze the asymptotic behavior of the U-statistic $U_3(n)$, whose entries $\tilde{S}_j' =\frac{1}{\ell_n} \sum_{i=(j-1)\ell_n+1}^{j\ell_n-m_n} ((Y_i')^2 -1)$ form a row-wise independent and identically distributed triangular array. Moreover, $E(\tilde{S}_1')=0$ and given the assumptions (A1) and (A2), we have
\begin{align*}
&\kappa_n^2:=\Var{ \sqrt{\ell_n}\tilde{S}_1'} \\
=&\frac{\ell_n-m_n}{\ell_n}\Var{Y_1^2}\; + \;2 \sum_{k=1}^{\ell_n-m_n-1}\frac{\ell_n-m_n-k}{\ell_n} \cov{Y_1^2}{Y_{k+1}^2}
\longrightarrow \kappa^2,
\end{align*}

by dominated convergence. By the central limit theorem for partial sums of $\beta$-mixing processes,
$  \sqrt{\ell_n} S_j' $ converges in distribution to a normal law with mean $0$ and variance $\kappa^2$. In our further analysis, we will essentially show that we may replace $S_j'$ by $\frac{\kappa}{\sqrt{\ell_n}}Z_j$, where $Z_j$ are independent standard normal random variables.

We first establish a law of large numbers for a rescaled version of $U_3(n)$.
\begin{proposition}
\label{Prop: LLN fuer U2}
Assume that (A1) and (A2) hold and that $m_n=o(\ell_n)$. Then, under the hypothesis,
\[
  \frac{\sqrt{\ell_n}}{\kappa} U_3(n) \pConv \frac{2}{\sqrt{\pi}},
\]
as $n\rightarrow \infty$.
\end{proposition}
\begin{proof}
In order to prove Proposition \ref{Prop: LLN fuer U2}, we first show that $\E{\frac{\sqrt{\ell_n}}{\kappa}U_3(n)}\rightarrow \frac{2}{\sqrt{\pi}}$. Note that
\[
  \E{\frac{\sqrt{\ell_n}}{\kappa} U_3(n)} =\E{\abs{\frac{\sqrt{\ell_n}}{\kappa} \tilde{S}_1'-\frac{\sqrt{\ell_n}}{\kappa}
\tilde{S}_2'}}.
\]
An application of the central limit theorem to the stationary $\beta$-mixing process
$((Y_i^\prime)^2 -1)_{i\in \N}$ yields
\[
  \E{\abs{\frac{\sqrt{\ell_n}}{\kappa} \tilde{S}_1' - \frac{\sqrt{\ell_n}}{\kappa} \tilde{S}_2' }}
\rightarrow \E{\abs{Z-Z^\prime}}=\frac{2}{\sqrt{\pi}}.
\]
and it suffices to additionally verify that $\Var{\sqrt{\ell_n}U_3(n)}\rightarrow 0$. By the definition of $U_3(n)$ and by the independence of the $S_j'$, we obtain
\begin{align*}
\Var{\sqrt{\ell_n}U_3(n)}
=& \frac{1}{b_n(b_n-1)} \Var{\abs{ \sqrt{\ell_n} \tilde{S}_1'-\sqrt{\ell_n} \tilde{S}_2'}}\\
& + \frac{2}{b_n} \cov{\abs{ \sqrt{\ell_n} \tilde{S}_1'-\sqrt{\ell_n} \tilde{S}_2'}}{
 \abs{ \sqrt{\ell_n} \tilde{S}_1' -\sqrt{\ell_n} \tilde{S}_3'}}.
\end{align*}
Now, one can check that the right-hand side converges to zero. Details of the proof are given in the appendix.
\end{proof}

Theorem \ref{Theorem: LLN Un} is an immediate corollary of Proposition \ref{Prop: LLN fuer U2} and the above approximation steps.

\subsection{Outline of the Proof of Theorem \ref{Theorem: CLT Un}}
A standard tool from U-statistics theory is the Hoeffding decomposition of the kernel $h(x,y)$, given by
\[
  h(x,y)=\theta^{(n)} +h_1^{(n)}(x) +h_1^{(n)}(y) +h_2^{(n)}(x,y),
\]
where
\begin{align*}
\theta^{(n)}=& \E{ h\rbraces{\sqrt{\ell_n}\tilde{S}_1'/\kappa, \sqrt{\ell_n}\tilde{S}_2')/\kappa}} \\
h_1^{(n)}(x) =& \E{h\rbraces{x, \sqrt{\ell_n}\tilde{S}_1'/\kappa}}-\theta^{(n)} \\
h_2^{(n)}(x,y)=& h(x,y)-\theta^{(n)}-h_1^{(n)}(x) -h_1^{(n)}(y).
\end{align*}
Note that since we are dealing with a U-statistic of a triangular array, the decomposition depends upon the sample size $n$ .
Using the independence of $\tilde{S}_j'$ and Fubini's theorem, we obtain
\begin{align*}
 \E{h_1^{(n)}\rbraces{\sqrt{\ell_n}\tilde{S}_1'/\kappa}}&= 0 \\
  \E{h_2^{(n)}\rbraces{x, \sqrt{\ell_n}\tilde{S}_1'/\kappa}} &=  \E{h_2^{(n)}\rbraces{\sqrt{\ell_n}\tilde{S}_1'/\kappa,y}} =0
\end{align*}
for all $x,y\in \R$. Thus, the random variables $h_1^{(n)}(\sqrt{\ell_n}\tilde{S}_1'/\kappa)$ are independent and have mean $0$. In addition, the kernels $h_2^{(n)}(x,y)$ are degenerate, and thus the random variables $h_2^{(n)}(\sqrt{\ell_n}\tilde{S}_j')/\kappa, \sqrt{\ell_n}\tilde{S}_k'/\kappa)$ are pairwise uncorrelated.
The Hoeffding decomposition of the kernel gives rise to the Hoeffding decomposition of $U_3(n)$,
\begin{align*}
\sqrt{b_n} \rbraces{ \frac{\sqrt{\ell_n}}{\kappa}U_3(n) -\theta^{(n)} }
=& \frac{2}{\sqrt{b_n}} \sum_{j=1}^{b_n} h_1^{(n)} \rbraces{\sqrt{\ell_n} \frac{\tilde{S}_j'}{\kappa}} \\
&+ \frac{\sqrt{b_n}}{b_n(b_n-1)} \sum_{1\leq j\neq k \leq b_n}
h_2^{(n)}\rbraces{\sqrt{\ell_n} \frac{\tilde{S}_j'}{\kappa}, \sqrt{\ell_n} \frac{\tilde{S}_k'}{\kappa}}.
\end{align*}
By the degeneracy of $h_2^{(n)}(x,y)$, we obtain
\begin{align*}
 & \Var{ \frac{\sqrt{b_n}}{b_n(b_n-1)} \sum_{1\leq j\neq k \leq b_n} h_2^{(n)}\rbraces{\sqrt{\ell_n} \frac{\tilde{S}_j'}{\kappa}, \sqrt{\ell_n} \frac{\tilde{S}_k'}{\kappa}} }\\
=& \frac{1}{b_n- 1} \Var{h_2^{(n)}\rbraces{\sqrt{\ell_n} \frac{\tilde{S}_1'}{\kappa}, \sqrt{\ell_n} \frac{\tilde{S}_2'}{\kappa}} },
\end{align*}
and some further calculations show that the right-hand side converges to $0$. To handle the linear term in the Hoeffding decomposition, we can apply Lyapunov's central limit theorem for row-wise independent triangular arrays, and obtain convergence towards a normal law with mean $0$ and variance
\[
  \psi^2=4   \lim_{n\rightarrow \infty} \Var{h_1^{(n)}\rbraces{\sqrt{\ell_n}\frac{\tilde{S}_1'}{\kappa}}}.
\]
Since $  \sqrt{\ell_n}\frac{\tilde{S}_1' }{\kappa} \rightarrow N(0,1)$,
one can prove $h_1^{(n)}(x) =Eh(x,\sqrt{\ell_n}\frac{ \tilde{S}_1' }{\kappa})-\theta^{(n)} \rightarrow h_1(x)=Eh(x,Z)-\theta$ with $\theta=\E{h(Z, Z^\prime)}$ and the above variance equals $4\Var{h_1(Z)}$.
In the end, we obtain the following central limit theorem for $U_3(n)$.

\begin{proposition}{
\label{Prop: CLT U3}
Assume that $\ell_n=n^s$ with $s>0.5$,  $m_n=o(n^s\wedge n^{2s-1})$ and that there exist constants $\rho>1\vee \frac{9\delta}{(\delta+1)(\delta+2)}$ and $0<\delta\leq 1$ such that $\E{\abs{Y_1}^{4+2\delta}}<\infty$ and for all $k\in \N$, it holds $\beta_Y(k)\leq C k^{-\rho (2+\delta)(1+\delta)/\delta^2}$. Then, we obtain under the null hypothesis
\begin{equation*}
\sqrt{n} \frac{U_3(n)}{\kappa}-\sqrt{b_n}\theta^{(n)} \distConv \NoD{0}{\psi^2},
\end{equation*}
where $\psi^2=4\Var{h_1(Z)}$, $h_1(x)=\E{\abs{x-Z'}}-2/\sqrt{\pi}$, where $Z$ and $Z'$ are two independent standard normal random variables.
}\end{proposition}

Note that the centering $\theta^{(n)}$ likewise depends on the sample size, and is generally unknown. Thus, it is important to replace $\theta^{(n)}$ by its limit $\theta=2/\sqrt{\pi}$.

\begin{lemma}
 \label{Lemma: Asymptotik theta thetan}
Assume that there exist constants $\rho>1$ and $0<\delta\leq 1$ such that $\E{\abs{Y_1}^{4+2\delta}}<\infty$ and for all $k\in \N$ it holds $\beta_Y(k)\leq C k^{-\rho (2+\delta)(1+\delta)/\delta^2}$. Moreover, let $\ell_n=n^s$ with $s>\rbraces{1+\delta\frac{\rho-1}{\rho+1}}^{-1}$ and let $m_n=o(n^s)$. Then, under the null hypothesis, we have
\[
  \sqrt{b_n}(\theta^{(n)}-\frac{2}{\sqrt{\pi}}) \rightarrow 0,
\]
as $n\rightarrow \infty$.
\end{lemma}

The proof of Theorem \ref{Theorem: CLT Un} now follows from a combination of the former results.
Proposition \ref{Prop: Approximation mit U_1 unter H und A}, Proposition \ref{Prop: Approx U1 U2}, and Proposition \ref{Prop: Approximation U2 U3} together imply
\[
  \sqrt{n}(U_3(n)-U(n)) \rightarrow 0.
\]
Proposition \ref{Prop: CLT U3} and Lemma \ref{Lemma: Asymptotik theta thetan}  yield
$\frac{\sqrt{n}}{\kappa} U_3(n)-\sqrt{b_n}\frac{2}{\sqrt{\pi}} \stackrel{\mathcal{D}}{\longrightarrow} N(0,\psi^2)$. Thus, Theorem \ref{Theorem: CLT Un} is a consequence of Slutzky's lemma.

\section{Extensions}
\label{Sec: Extensions}

\subsection{Modifications of our Test Statistic}

Our test statistic corresponds to Gini's mean difference with entries given by the logarithms of local empirical variances. It explores the variability of estimated local variabilities. There are several possible extensions of this idea which might be useful not only for testing the constancy of the variance.

One possibility is the choice of other kernel functions $h:\R\times\R\rightarrow \R$ with suitable characteristics. In particular, we obtained similar limit results for the more general statistic
\begin{equation*}
U(n)=\frac{1}{b_n(b_n-1)}\sum_{1\leq j \neq k \leq b_n} \abs{\log\rbraces{\hat{\sigma}_j^2}-\log\rbraces{\hat{\sigma}_k^2}}^\alpha \quad \text{for } \alpha\in (0,1]
\end{equation*}
but they come at the price of additional requirements regarding the growth conditions of the blocks, for instance $b_n/\ell_n^\alpha \rightarrow 0$. Moreover, additional simulations not reported here suggest that not much is gained by employing values other than  $\alpha=1$. The good performance obtained for $\alpha=1$ can partly be explained by the rather large efficiency and better robustness of Gini's mean difference as compared to standard measures of variability like the standard deviation, see Gerstenberger, Vogel and Wendler \citep{Gerstenberger.2019} and the references cited therein. 

Another modification is the evaluation of the variability of other local statistics, which shall be addressed in future work. More robust measures of scale as considered by Gerstenberger, Vogel and Wendler \citep{Gerstenberger.2019} in a setting with a single change point might replace the empirical variances. Using estimates of central location, kurtosis or tail behavior allows testing the stability of the corresponding characteristic. A hypothesis of particular interest is the assumption of the stability of second order characteristics.

\begin{remark}{Davis, Huang and Yao \citep{Davis.1995} investigate likelihood ratio statistics for testing whether the time series can be described by a stable autoregressive process against the alternative of a single change. A change of the variance is not of direct interest but implicitly considered there since a change of the autoregressive parameters usually implies a change of the marginal variance if the variance of the innovations is constant. Within the somewhat restrictive parametric framework of Davis, Huang and Yao \citep{Davis.1995}, one can apply a version of their test which allows for changing variances of the innovation process to check the assumption of a stable dependence structure underlying our test. If we are willing to accept the hypothesis of a stable autoregressive model after application of this preliminary test, testing the stability of the marginal variance using our approach is equivalent to testing the stability of the innovational variance.

Dette, Wu and Zhou \citep{Dette.2015}, \citep{Dette.2019} investigate CUSUM-statistics for testing the constancy of second order characteristics, i.e., of the variance or the correlation structure. They work on detrended data, allowing for a possibly time-varying mean under the null hypothesis, like we do.
       Given the complicated structure of their limiting process, they need bootstrap procedures for the calculation of critical values. Their test for the constancy of the correlation structure could be combined with our test, as outlined above for the test in Davis, Huang and Yao \citep{Davis.1995}.

Tests for the null hypothesis of stable second order characteristics which are consistent not only against alternatives with a single change might be constructed by analyzing the variability of local variance and local correlation estimates jointly. Such multivariate extensions of the theory developed here are left to future work.
}\end{remark}

\subsection{Application to data with a discontinuous or a  Hölder-continuous mean}
So far, we have assumed a time series setting with a Lipschitz-continuous mean function. This assumption might be too restrictive, e.g., if structural breaks may occur not only in the variance but also in the mean. In such contexts, the hypothesis that the variance of the observed data $(X_i)_{i\in \N}$ is constant can be tested using the time series $(Z_i)_{i\in \N}$ of differences $$Z_i=X_i-X_{i-1}=\sigma_H(Y_i-Y_{i-1})+(\mu_i-\mu_{i-1}).$$
It inherits many properties from $(X_i)_{i\in \N}$ itself, like the existence of moments and the absolute regularity. 
Under the null hypothesis, the variance of the observations  becomes
\begin{align*}
\Var{Z_i} =& \sigma_H^2\cdot \rbraces{\Var{Y_i}+\Var{Y_{i-1}}+2\cov{Y_i}{Y_{i-1}}}\\
= & 2\sigma_H^2 \cdot \rbraces{\Var{Y_1}+\cov{Y_1}{Y_2}}.
\end{align*}
If the covariances are stationary, which comes along with 
the strict stationarity of $(Y_i)_{i\in \N}$ assumed in this entire work, changes in the variance of $(X_i)_{i\in\N}$ can thus be detected from the sequence of differences $(Z_i)_{i\in\N}$. Moreover, the mean function of the differences not only preserves the Lipschitz-continuity and gets arbitrarily small, $\abs{\mu_i-\mu_{i-1}}=O(1/n)$, but a close look at the proofs reveals that the results still hold in case of finitely many outliers in the mean (i.e., if $\mu$ is only piecewise Lipschitz-continuous on the intervals between finitely many jumps and we consider the time series $(Z_i)_{i\in\N}$). Hence, resorting to the differenced time series is advisable if abrupt changes in the mean are suspected.

\begin{example}{
Assume that we observe an increasing number of $\ell_n$ observations in each of an increasing number of $b_n$ groups.
One-way analysis of variance allows testing whether the group centers are identical, assuming all observations to be independent and identically distributed within each of the groups, with the same variance $\sigma_j^2=\sigma^2$ for all groups $j=1,\ldots,b_n$.
The results derived in this paper allow testing the basic assumption of constant variance even if the observations within the different groups are dependent, as long as $\ell_n$ and $b_n$ increase at appropriate rates. If the mean of the observations is not constant but only Lipschitz-continuous within each group, we can take differences of the observations within each group and proceed as described above. The condition of identical group sizes $\ell_n$ is restrictive but can presumably be relaxed as long as the group sizes $\ell_n^{(j)}$ increase at identical rates $\ell_n^{(j)}=O(\ell_n)$.
}\end{example}

\begin{remark}{
\label{Remark: Hölder-ctd} A careful inspection of our proofs shows that one can moreover weaken the assumption of a Lipschitz-continuous mean function of $(X_i)_{i\in\mathbb{N}}$ to include a Hölder-continuous one with $|\mu_i-\mu_r|\leq L|i-r|/n^{s+1/4}$.  This signifies that the overall variation of $\mu$ on the interval $[0,1]$ is even allowed to grow at rate $n^{3/4-s}$. By reverting to the differenced time series $(Z_i)_{i\in\mathbb{N}}$, one is in addition able to deal with an increasing number of jumps $J(n)$ or a growing jump height $\Delta(n)$, as long as the condition $J(n)\Delta(n)^2=o(n^{1/2})$ holds, and with a continuous part of the mean function fulfilling $|\mu_i-\mu_r|\leq L|i-r|/n^{1/4}$. 
}\end{remark}

\subsection{Estimation of the Change-points and the Variance Function}
\label{Subsec: Estimation of Cp location}

Another interesting issue is the estimation of the scale function $\sigma(\cdot)$ in case our test rejects the null hypothesis of $\sigma$ being constant. If we assume the scale function to be piecewise constant, it is reasonable to estimate the change-point locations in advance and to calculate the empirical standard deviations of the observations in between subsequent change-points. Keeping in mind that our test is consistent also against smoothly varying alternatives, one could alternatively calculate local estimates using moving window techniques, possibly with an adaptive choice of the window width.

Under the assumption of a piecewise constant scale, we can adopt the approach from Wornowizki, Fried and Meintanis \citep{Wornowizki.2017} as a first means to determine the number and location of the change-points. This approach is based on a recursive procedure which conducts in every step the change-point test described above. If the hypothesis is rejected, the dominant change point of the (sub-)sample is located by narrowing down the set of candidates to two blocks $B_{j^*}=\{(j^*-1)\ell_n+1, ..., j^*\ell_n\}$ and  $B_{j^*+1}=\{j^*\ell_n+1, ..., (j^*+1)\ell_n\}$ via
\begin{equation*}
j^*=\underset{j\in \{1, ..., b_n-1\}}{\mathrm{argmax}}\abs{\log(\hat{\sigma}^2_j)-\log(\hat{\sigma}^2_{j+1})}.
\end{equation*}
Afterwards, the location of the change-point $t^*$ is determined by
\begin{equation*}
t^*=\underset{t\in B_{j^*}\cup B_{j^*+1}}{\mathrm{argmax}}\abs{\hat{\sigma}^2\rbraces{X_{(j^*-1)\ell_n+1}, ..., X_t}- \hat{\sigma}^2\rbraces{X_{t+1}, ..., X_{(j^*+1)\ell_n}}},
\end{equation*}
where $\hat{\sigma}^2\rbraces{X_i, ..., X_k}$ for $i\leq k$ is the empirical variance based on the observations $X_i, ..., X_k$  and where we exclude values of t that are very close to the boundaries  in order to ensure a reliable estimation on both subsets $\{X_{(j^*-1)\ell_n+1}, ..., X_t \}$ and  $\{X_{t+1}, ..., X_{(j^*+1)\ell_n}\}$.
Once the change-point $t^*$ is detected, we split the sample at $t^*$ into two parts and the procedure is repeated on both subsamples. If the hypothesis can no longer be rejected, the procedure stops and it is assumed that all change points have been located.
This procedure looks somewhat ad hoc and can probably be improved. Nevertheless, in our simulations it yielded similarly good results  as the MOSUM approach described by
Eichinger and Kirch \citep{Eichinger.2018}, which in turn performed well as compared to several other state-of-the-art procedures in the simulations reported there.

\section{Simulation Study}
\label{Sec: Simualtion Study}
We investigate the empirical size and power of the change-point test introduced in Section \ref{Sec: Test statistic and main results} and compare its performance to the procedure proposed in an earlier version \citep{Dette.2015} of Dette, Wu and Zhou \citep{Dette.2019}. Moreover, we discuss the performance of the long run variance estimator $\hat{\kappa}$ introduced in Section \ref{Subsec: Estimation of LRV}.

As data-generating processes, we consider two examples of independent observations, namely standard normal, $X_i \sim \NoD{0}{1}$, and exponential ones, $X_i\sim Exp(1)$, and four examples under dependence. The latter consist of two AR(1)-processes with $\alpha_1=0.4$ and $\alpha_1=0.7$, respectively, an ARMA(2,2)-process
$$X_i= 0.8X_{i-1}-0.4X_{i-2}+\varepsilon_i+0.5 \varepsilon_{i-1}+0.34 \varepsilon_{i-2} $$
and a GARCH(1,1)-process
$$X_i=\sigma_i\varepsilon_i\quad \text{with} \quad \sigma_i^2= 0.1+0.1X_{i-1}^2+0.8 \sigma_{i-1}^2,$$
each with independent standard normal innovations $(\varepsilon_i)_{i\in\N}$. The parameter choice for the GARCH-process is as in Andreou and Ghysels \citep{AndreouGhysels.2002} and describes a high volatility persistence. ARMA- and GARCH- processes are, as mentioned in Example \ref{Example: ARMA, GARCH}, mixing at an exponential rate. 
Moreover, all models considered possess finite sixth moments (see, Theorem 5 in Lindner \citep{Lindner2009} for the GARCH-case). Based on extensive simulations not reported here, we recommend setting the tuning parameters to $s=0.7$ and $q=0.5$ with $\ell_n=n^s$  and $\tilde{\ell}_n=n^q$. Note that this choice is in line with all restrictions on the block lengths imposed by our asymptotic theory.

We investigate the empirical power of the tests under various (local) alternatives listed below. These include scenarios with one, two and four structural breaks, as well as a smoothly changing variance function, each with variance changes of magnitude $n^{-1/2}$.
  \begin{align*}
\mathbb{A}1: \sigma(x)=& 1\cdot\1_{\{0\leq x< 1/2\}}+ (1+0.2 \sqrt{2000/n})\cdot\1_{\{1/2\leq x\leq 1\}}\\
\mathbb{A}2: \sigma(x)= & 1\cdot\1_{\{0\leq x<1/3\}}+ (1+0.2 \sqrt{2000/n})\cdot\1_{\{1/3\leq x<2/3\}}+ 1 \cdot\1_{\{2/3\leq x\leq 1\}}\\
\mathbb{A}3: \sigma(x)=& 1\cdot\1_{\{0\leq x< 1/5\}}+ (1+0.2 \sqrt{2000/n})\cdot\1_{\{1/5\leq x< 2/5\}}
 + 1 \cdot\1_{\{2/5 \leq x< 3/5\}}\\
 &+ (1+0.2 \sqrt{2000/n})\cdot\1_{\{3/5 \leq x< 4/5\}} +  1 \cdot\1_{\{4/5\leq x\leq 1\}}\\
\mathbb{A}4:  \sigma(x)=&1+ 0.1\cdot \sin(4\pi x)\cdot\sqrt{2000/n}
\end{align*}

All simulations are conducted in R \citep{RCoreTeam.2019}. The long run variance $\hat{\kappa}$ is estimated as described in Section \ref{Subsec: Estimation of LRV}, whereas the centering term $\E{\abs{Z-Z'}}=2/\sqrt{\pi}$ and the variance of the limit distribution  $\psi^2= 4/3+ 8/\sqrt{\pi}(\sqrt{3}-2)$ can be calculated explicitly. All results are obtained for a nominal significance level of $\alpha=5\%$. The analysis in Subsections \ref{Subsec: Simulations Comparison procedures}-\ref{Subsec: Simulations Empirical Power} is conducted for centered observations. We investigate the influence of non-constant mean functions on our test results in Subsection \ref{Subsec: Simulations non-centred Data} and on the performance of the long run variance estimator $\hat{\kappa}$ in Subsection \ref{Subsec: Performance of LRV estimator}.

\subsection{Comparison to the Procedure of Dette, Wu and Zhou \citep{Dette.2015}, \citep{Dette.2019}}
\label{Subsec: Simulations Comparison procedures}
This section compares our test (henceforth denoted by SWFD) to the results of Dette, Wu and Zhou \citep{Dette.2015}, \citep{Dette.2019} (denoted by DWZ), who employ a CUSUM-approach to detect change points in the variance. Since the asymptotic distribution of their test statistic crucially depends upon the dependence structure of the underlying time series, they use a bootstrap approach to obtain critical values.

The comparison is limited to the sample lengths $n=500, 2000$ and the simulated rejection probabilities are based on only 4000 replications each, due to the computational cost of the DWZ procedure. As a rough comparison of the computation times of the respective procedures, we measured the overall computation time (on a 3.4 GHz Intel Core i7) required to obtain the results in Table \ref{Table: Comparison 500,2000, A size-corrected}, i.e., the time required for $4000\cdot 30$ executions of the respective test
(including the time for the simulation of the data sets).
For $n=500$ ($n=2000$), the SWFD procedure needed on average 0.00064 (0.00138) seconds per execution, 
while the DWZ procedure took 0.88813 (7.73488) seconds. 

 Table \ref{Table: Comparison 500,2000, A size-corrected}  shows the simulated rejection probabilities of the SWFD and the DWZ procedure under the hypothesis and the local alternatives $\mathbb{A}1$-$\mathbb{A}4$. 
Our SWFD test performs anti-conservative for $n=500$ as we observe empirical sizes of about 10\% or even higher instead of the nominal 5\% level. For $n=2000$, the empirical sizes fall below 10\%, except for the GARCH(1,1). A closer analysis of the SWFD procedure will be presented in the next subsection. For the DWZ test, we concentrate on the two sample sizes because of its much larger computation times. The size of the DWZ procedure is adequate for independent normal data but conservative for the exponential data and anti-conservative in the presence of higher positive dependencies. In case of the GARCH(1,1), it performs worse than the SWFD test.  

 To achieve a fair comparison, we report size-corrected empirical powers under the four alternatives, using the empirical 95\% percentile of the test results for the same distribution and the same sample size as critical values (see, Table \ref{Table: Comparison 500, 1000, 2000 nominal} in Appendix \ref{Appendix: Additional Simulations} for the rejection rates at the nominal (asymptotical) 5\%-level). While the DWZ procedure obtains higher empirical power in case of one or few structural breaks ($\mathbb{A}1$ and $\mathbb{A}2$), the SWFD does so in case of multiple breaks and for the sine function ($\mathbb{A}3$ and $\mathbb{A}4$).
 Moreover, both tests show difficulties in coping with models of strong dependence, especially the GARCH(1,1)-model, as well as with exponentially distributed observations.

\begin{table}[H]
    \centering
    \caption{Simulated rejection probabilities of the SWFD and DWZ test at the significance level $\alpha=0.05$ for the sample sizes $n=500, 2000$ under the null hypothesis $\mathbb{H}$ and various local alternatives $\mathbb{A}1$ to $\mathbb{A}4$ with effect sizes of magnitude $n^{-1/2}$ and for different data-generating processes. For the results under the alternatives, a size-correction has been conducted.}
    \label{Table: Comparison 500,2000, A size-corrected}
    \begin{tabular}{|cc|cccccc|}
     \hline
         & & N(0,1) & Exp(1) & AR(1), 0.4 & AR(1), 0.7 & ARMA(2,2) & GARCH(1,1) \\ \hline         				\multicolumn{8}{|c|}{$n=500$} \\\hline
         \multirow{2}{*}{$\mathbb{H}$}  & SWFD & 0.085 & 0.112 & 0.098 & 0.134 & 0.106 & 0.180 \\
         & DWZ & 0.052 & 0.029 & 0.064 & 0.088 & 0.073 & 0.290 \\\hline
           \multirow{2}{*}{$\mathbb{A}1$}  & SWFD & 0.734 & 0.318 & 0.630 & 0.356 & 0.400 & 0.336 \\
        & DWZ & 0.999 & 0.649 & 0.980 & 0.824 & 0.836 & 0.680 \\\hline
        \multirow{2}{*}{$\mathbb{A}2$}  & SWFD & 0.457 & 0.186 & 0.376 & 0.200 & 0.253 & 0.178 \\
        & DWZ & 0.618 & 0.159 & 0.492 & 0.258 & 0.252 & 0.134 \\\hline
        \multirow{2}{*}{$\mathbb{A}3$} & SWFD & 0.221 & 0.126 & 0.202 & 0.128 & 0.154 & 0.114 \\
        & DWZ & 0.088 & 0.047 & 0.069 & 0.044 & 0.061 & 0.042 \\\hline
        \multirow{2}{*}{$\mathbb{A}4$}  & SWFD & 0.481 & 0.194 & 0.399 & 0.222 & 0.255 & 0.199 \\
        & DWZ & 0.372 & 0.118 & 0.274 & 0.144 & 0.160 & 0.096 \\\hline
       \multicolumn{8}{|c|}{$n=2000$} \\\hline
      \multirow{2}{*}{$\mathbb{H}$}   & SWFD & 0.073 & 0.091 & 0.074 & 0.096 & 0.084 & 0.148 \\
         & DWZ & 0.052 & 0.039 & 0.058 & 0.108 & 0.074 & 0.394 \\ \hline
        \multirow{2}{*}{$\mathbb{A}1$}  & SWFD & 0.891 & 0.344 & 0.784 & 0.439 & 0.505 & 0.375 \\
             & DWZ & 1 & 0.760 & 0.995 & 0.875 & 0.920 & 0.727 \\ \hline
        \multirow{2}{*}{$\mathbb{A}2$} & SWFD & 0.734 & 0.246 & 0.610 & 0.310 & 0.343 & 0.251 \\
         & DWZ & 0.869 & 0.239 & 0.748 & 0.382 & 0.446 & 0.228 \\ \hline
        \multirow{2}{*}{$\mathbb{A}3$}  & SWFD & 0.805 & 0.269 & 0.674 & 0.368 & 0.397 & 0.294 \\
         & DWZ & 0.318 & 0.066 & 0.261 & 0.127 & 0.140 & 0.070 \\ \hline
        \multirow{2}{*}{$\mathbb{A}4$} & SWFD & 0.644 & 0.185 & 0.504 & 0.256 & 0.279 & 0.218 \\
        & DWZ & 0.516 & 0.136 & 0.392 & 0.192 & 0.248 & 0.138 \\ \hline

    \end{tabular}
\end{table}

\subsection{Analysis of the Empirical Size}
\label{Subsec: Simulations Empirical Size}
In the next subsections, we will analyse the performance of the SWFD-test in more detail. Throughout, we will base our results on 6000 replications. Figure \ref{Fig:Empirical Size} depicts the empirical size as a function of the sample length, where we used $ n=500,1000,2000,3000,4000, 5000,8000,12000$ and $16000$.

The empirical size is about 10\% or even larger in case of moderately large samples ($n=500$). It approaches the nominal significance level alpha=5\%
as the sample size increases, though the test stays liberal in case of the scenarios considered here.
Moreover, the dependence as well as the distribution of the data seem to be important. The empirical sizes are smaller if the dependences among the observations are low (N(0,1), AR(1) with parameter $\alpha_1=0.4$), while stronger dependences as in the GARCH(1,1) process lead to much higher rejection rates. The rejection rates are also rather large in case of the exponentially distributed observations, despite their independence.

\begin{figure}[H]
 \includegraphics[width=\textwidth, trim={0 0 0 1.6cm}, clip]{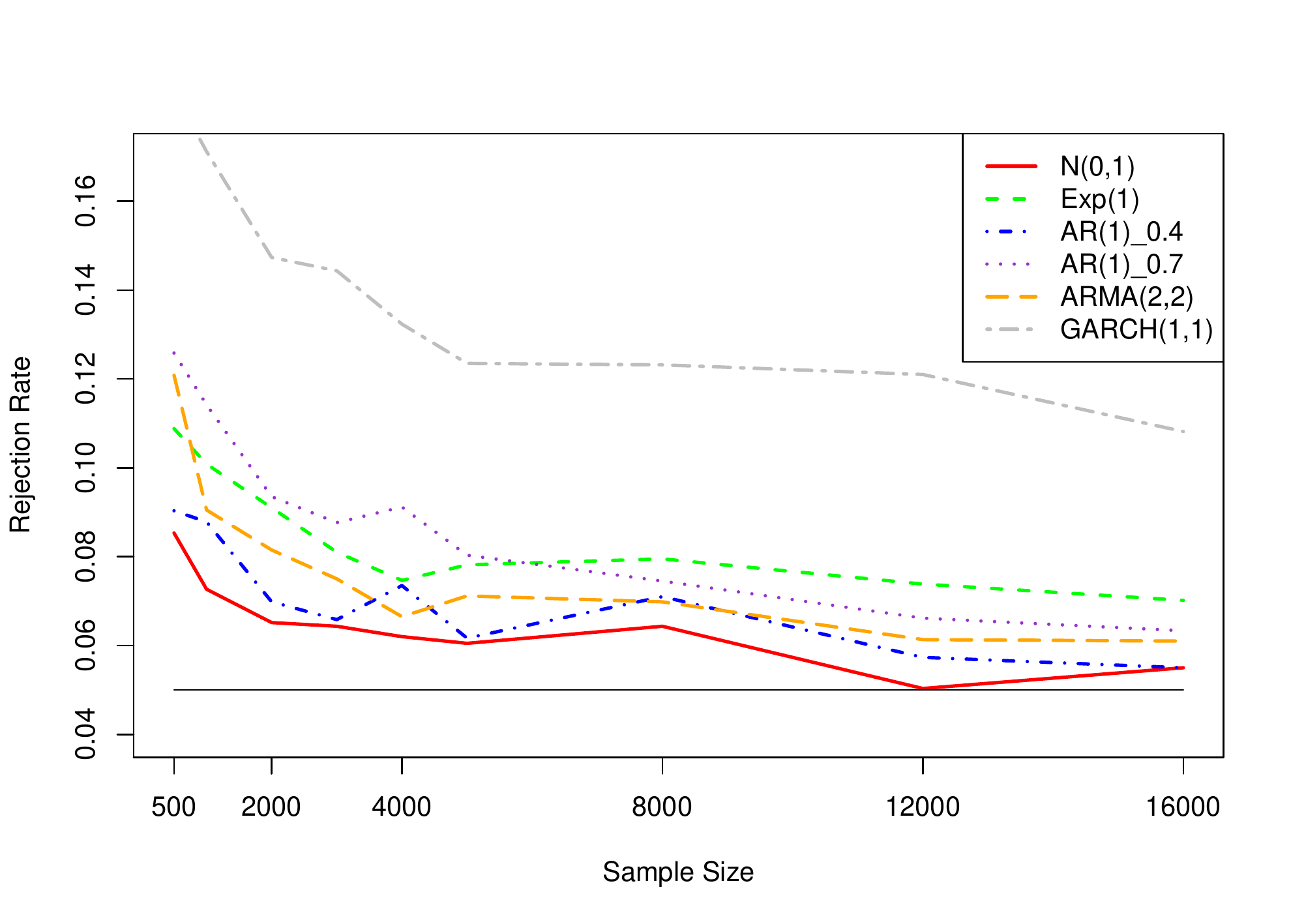}
\vspace*{-1cm}\caption{Empirical rejection rates of the SWFD test under the null hypothesis as a function of the sample size for different distributions of the data-generating process. 
 The nominal significance level $\alpha=0.05$ is represented by the solid black line.}
\label{Fig:Empirical Size}
\end{figure}

\subsection{Analysis of the Empirical Power}
\label{Subsec: Simulations Empirical Power}
Figure \ref{Fig: Empirical power} shows the empirical power at the nominal significance level $\alpha=0.05$ as a function of the sample size separately for each of the local alternatives $\mathbb{A}1$-$\mathbb{A}4$ and for the same choices of $n$ as considered before (see, Figure \ref{Fig: Empirical power, size-corrected} in Appendix \ref{Appendix: Additional Simulations} for the size-corrected version). All simulated rejection probabilities are based on 6000 replications.
Across all alternatives, the graphs for the different underlying distributions show the same order, which is almost a reversed image of the order in Figure \ref{Fig:Empirical Size}.
The highest empirical power is obtained for the distributions with the lowest dependence (N(0,1), AR(1) with parameter $\alpha_1=0.4$), while all other processes lead to somewhat smaller rejection rates. The performance for small sample sizes is significantly worse compared to the empirical power for larger $n\geq 2000$, but once an appropriate sample length is reached, the  rejection rates stabilize at a certain level.
Besides, the partitioning into blocks of length $\ell_n=n^{0.7}$ might have a certain influence, being less ideal for some $n$ with respect to the location of the break points. As an example, observe the peak at $n=4000$ under alternative $\mathbb{A}2$ (top right).
The partition there is close to ideal due to $(4000^{0.7})\cdot 4\approx 1328$ while the break points are located at $4000/3\approx 1333$ and at $2666$.

\begin{figure}[H]
 \begin{subfigure}[c]{0.49\textwidth}
  \includegraphics[width=\textwidth,  trim={0cm 0cm 0.8cm 2cm},clip]{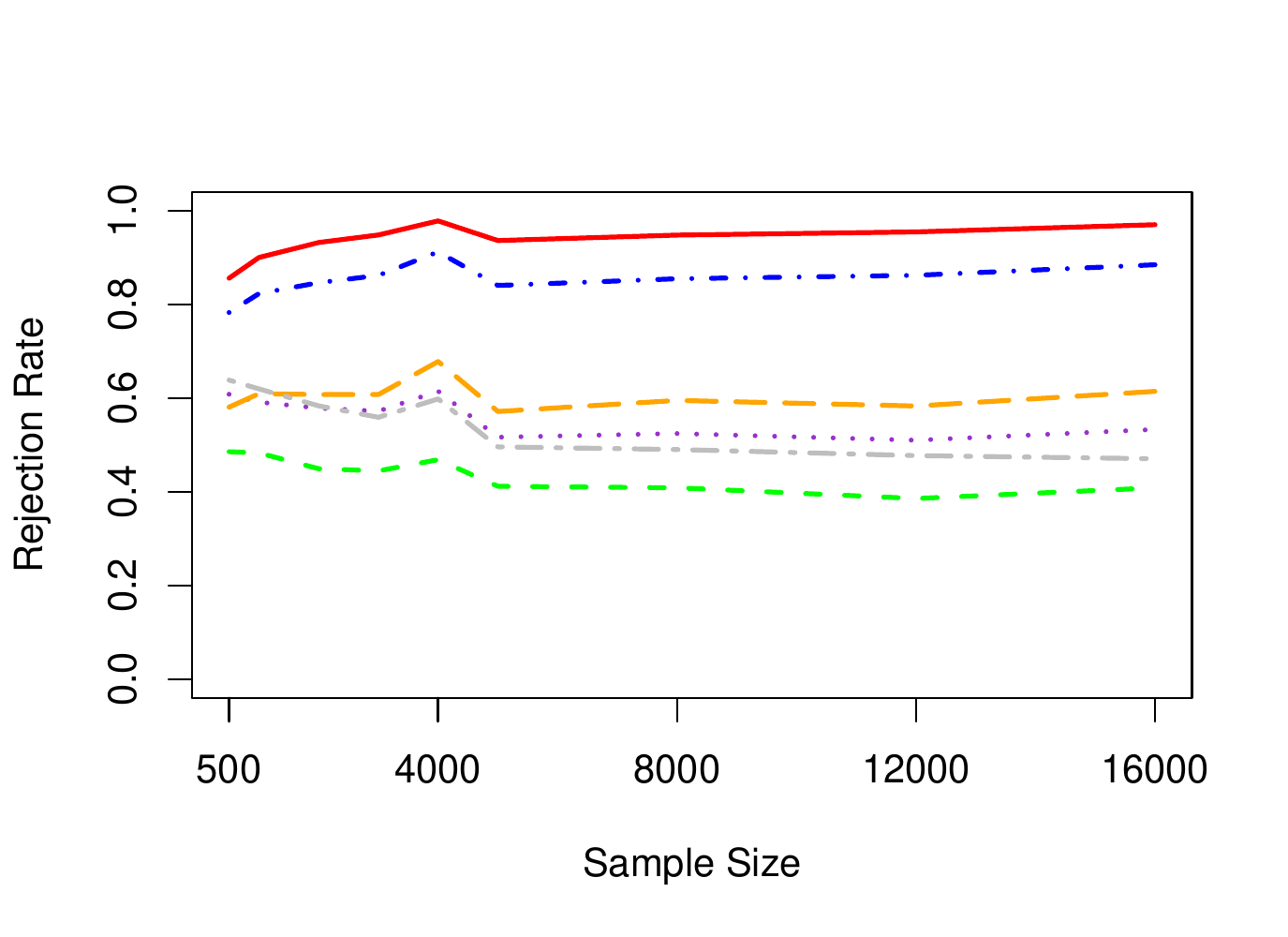}
 \end{subfigure}
   \begin{subfigure}[c]{0.49\textwidth}
  \includegraphics[width=\textwidth,  trim={0cm 0cm 0.8cm 2cm},clip]{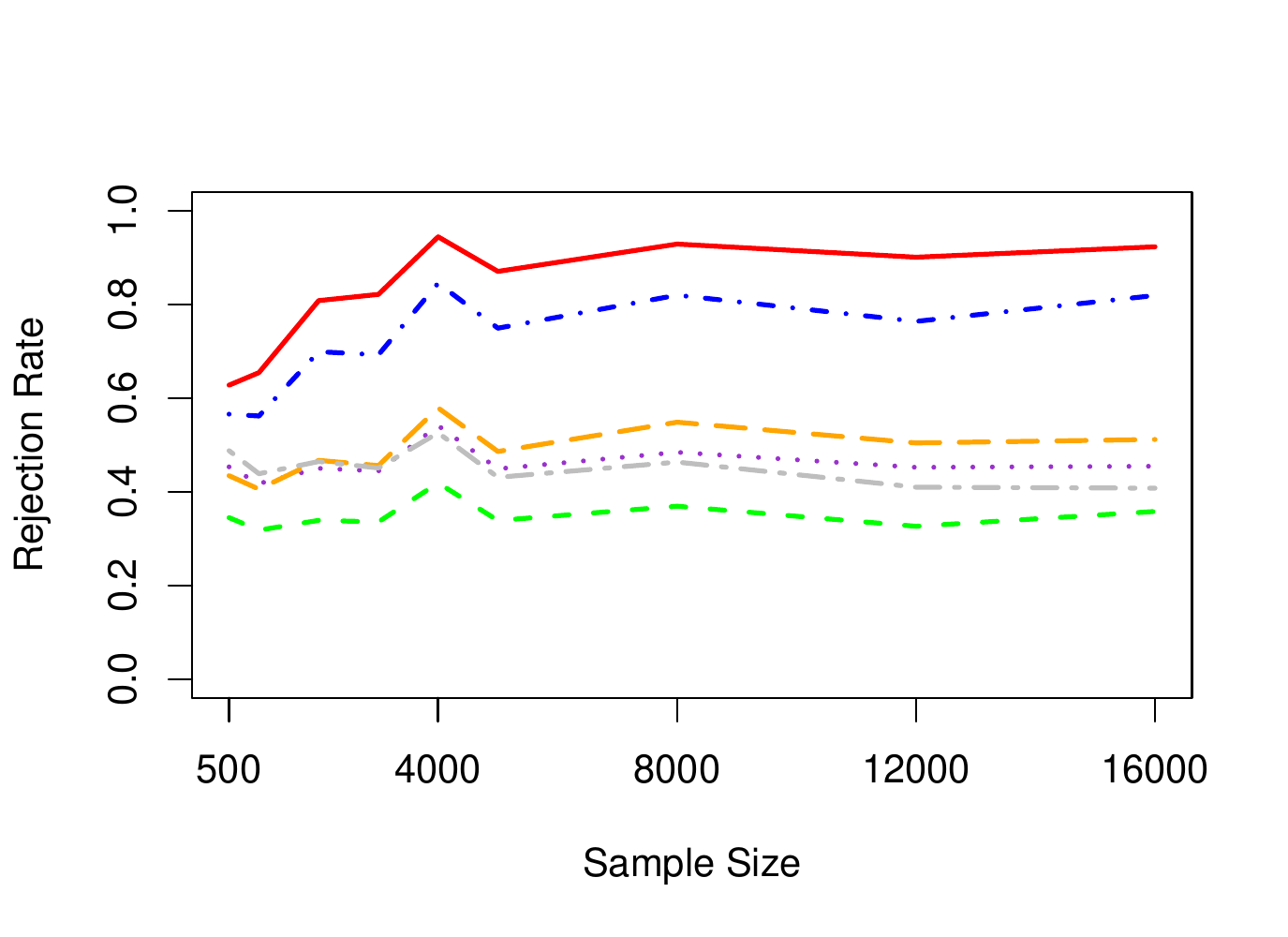}
  \end{subfigure}
   \begin{subfigure}[c]{0.49\textwidth}
  \includegraphics[width=\textwidth,  trim={0cm 0cm 0.8cm 2cm},clip]{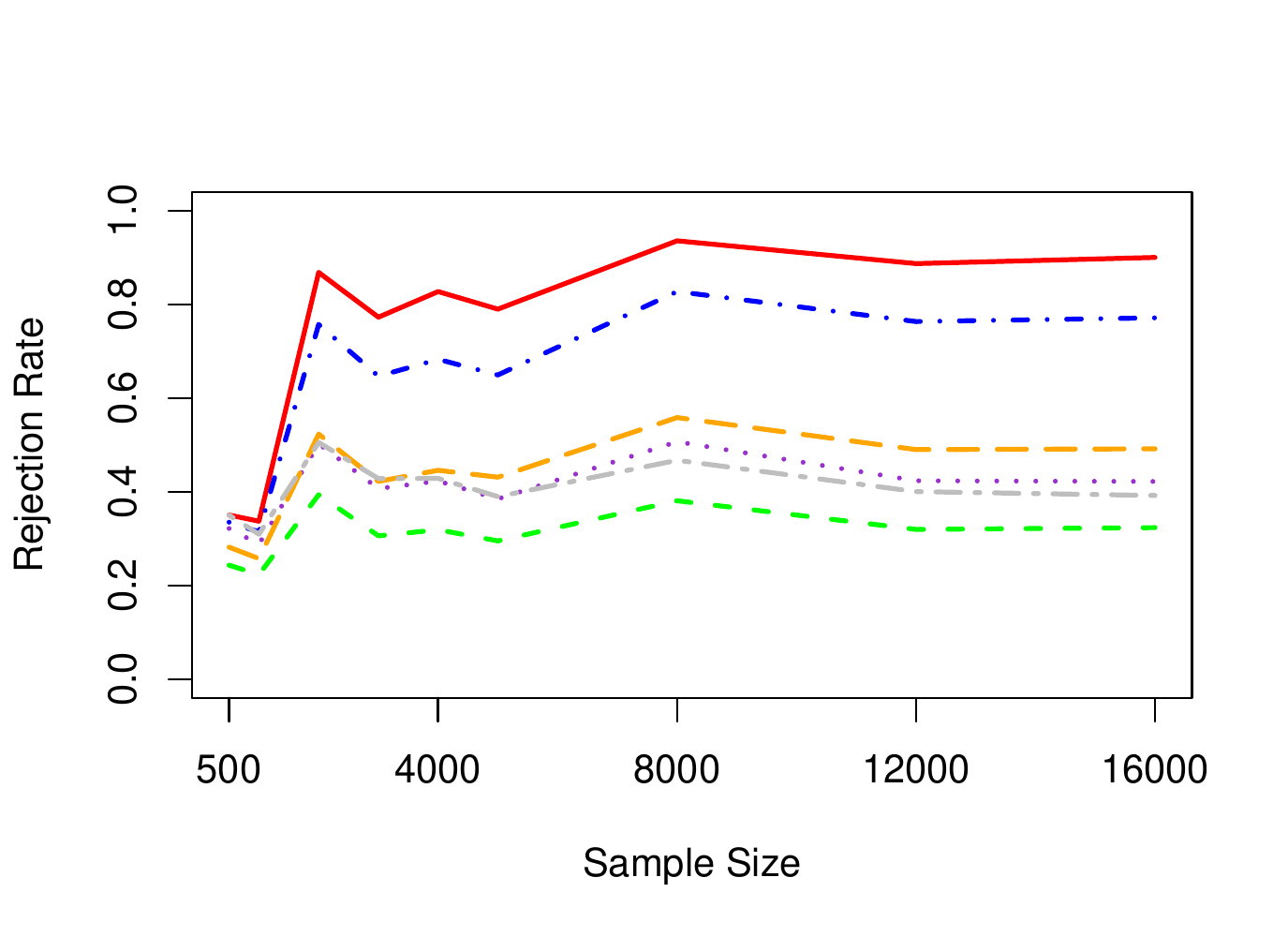}
  \end{subfigure}
   \begin{subfigure}[c]{0.49\textwidth}
  \includegraphics[width=\textwidth,  trim={0cm 0cm 0.8cm 2cm},clip]{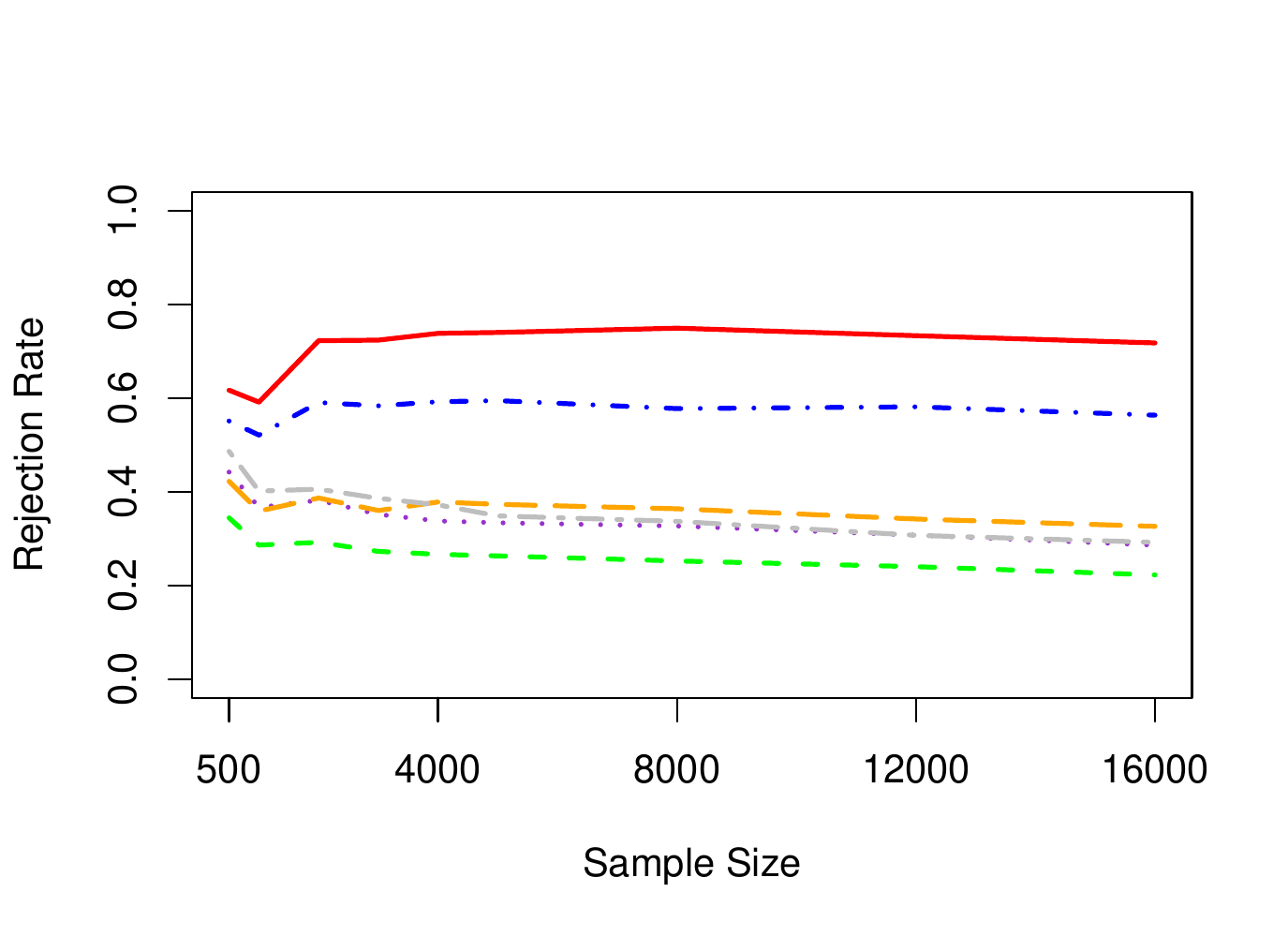}
  \end{subfigure}
   \begin{subfigure}[c]{\textwidth}

  \includegraphics[width=\textwidth, trim={0cm 7cm 0cm 6.5cm},clip]{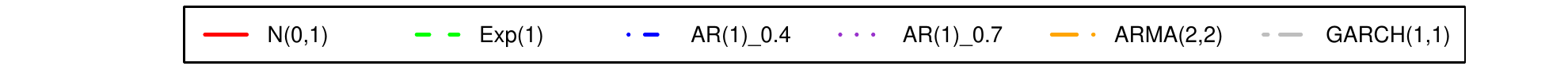}
  \end{subfigure}

\caption{Empirical power of the SWFD test for the local alternatives $\mathbb{A}1$(top left)-$\mathbb{A}4$(bottom right) with effect sizes of magnitude $n^{-1/2}$ as a function of the sample size $n$ at the nominal significance level $\alpha=0.05$. }
   \label{Fig: Empirical power}
    \end{figure}

\subsection{Performance for non-centred Data}
\label{Subsec: Simulations non-centred Data}
Our limit theory likewise holds for time series that are not stationary in the mean. Table \ref{Table: Rej prob different mu} shows the simulated rejection rates at the nominal significance level $\alpha=0.05$ for samples of length $n=3000$ and for the mean functions $\mu(x)=x$, $\mu(x)=\sin(2\pi x)$ and $\mu(x)=0\cdot\1_{\{0\leq x<1/2\}}+ 1\cdot\1_{\{1/2\leq x\leq 1\}}$  (see, Table \ref{Table: Rej prob different mu size-corr} in Appendix \ref{Appendix: Additional Simulations} for the size-corrected version). The theory developed here requires that our test is applied to data with a Lipschitz-continuous mean function, implying that jumps in the mean cause problems. As pointed out in Section \ref{Sec: Extensions}, this issue can be   resolved by analysing the differenced time series $(Z_i)_{i\in\N}$, $Z_i=X_i-X_{i-1}$. The results for the last mean function demonstrate that this approach is indeed also effective in practice.

\begin{table}
    \centering
    \caption{Simulated rejection probabilities of the test SWFD at the nominal significance level $\alpha=0.05$ for the sample size $n=3000$ under the hypothesis and various local alternatives with effect sizes of magnitude $n^{-1/2}$. Results for different data-generating processes and different mean functions $\mu$ are given. }
    \label{Table: Rej prob different mu}
   
    \begin{tabular}{|ccccccc|}
         \hline
         & N(0,1) & Exp(1) & AR(1), 0.4 & AR(1), 0.7 & ARMA(2,2) & GARCH(1,1) \\ \hline
        \multicolumn{7}{|c|}{$\mu(x)=x$} \\ \hline
         $\mathbb{H}$ & 0.062 & 0.081 & 0.066 & 0.085 & 0.076 & 0.134 \\
       $\mathbb{A}1$ & 0.954 & 0.441 & 0.863 & 0.578 & 0.619 & 0.561 \\
        $\mathbb{A}2$ & 0.822 & 0.330 & 0.703 & 0.426 & 0.463 & 0.464 \\
       $\mathbb{A}3$ & 0.783 & 0.312 & 0.639 & 0.400 & 0.420 & 0.418 \\
        $\mathbb{A}4$ & 0.723 & 0.267 & 0.582 & 0.354 & 0.368 & 0.387 \\
        \hline
        \multicolumn{7}{|c|}{$\mu(x)=\sin(2\pi x)$} \\ \hline
          $\mathbb{H}$ & 0.062 & 0.086 & 0.077 & 0.094 & 0.070 & 0.148 \\
        $\mathbb{A}1$ & 0.940 & 0.451 & 0.854 & 0.579 & 0.607 & 0.569 \\
        $\mathbb{A}2$ & 0.835 & 0.360 & 0.697 & 0.436 & 0.466 & 0.457 \\
        $\mathbb{A}3$ & 0.696 & 0.278 & 0.576 & 0.382 & 0.409 & 0.383 \\
        $\mathbb{A}4$ & 0.722 & 0.267 & 0.574 & 0.360 & 0.370 & 0.367 \\
        \hline
        \multicolumn{7}{|c|}{$\mu(x)=0\cdot\1_{\{0\leq x<1/2\}}+ 1\cdot\1_{\{1/2\leq x\leq 1\}}$ for time series $(Z_i)_{i\in\N}$}  \\ \hline
          $\mathbb{H}$ & 0.067 & 0.091 & 0.064 & 0.062 & 0.079 & 0.124 \\
        $\mathbb{A}1$ & 0.814 & 0.390 & 0.889 & 0.924 & 0.664 & 0.491 \\
        $\mathbb{A}2$ & 0.648 & 0.306 & 0.737 & 0.783 & 0.512 & 0.391 \\
        $\mathbb{A}3$ & 0.587 & 0.279 & 0.692 & 0.733 & 0.454 & 0.369 \\
        $\mathbb{A}4$ & 0.538 & 0.234 & 0.628 & 0.687 & 0.405 & 0.340 \\
        \hline
    \end{tabular}
\end{table}

\subsection{Performance of the long run variance estimator $\hat{\kappa}$}
\label{Subsec: Performance of LRV estimator}

We briefly discuss the performance of the estimator $\hat{\kappa}$ introduced in Section \ref{Subsec: Estimation of LRV}
on the basis of its empirical bias and root mean square error (RMSE). To facilitate a comparison, we standardize the data-generating processes introduced above to yield a theoretical long run variance of 1. 
 All results are based on 6000 replications.
Table \ref{Table: Performance of the LRV estimator} shows the results for the sample lengths $n=500, 1000, 3000$ and for the mean functions  $\mu(x)=0$, $\mu(x)=\sin(2\pi x)$ and $\mu(x)=\1_{\{1/2<x\leq1\}}$. Note that for the latter mean function, we consider the differenced time series $(Z_i)_{i\in\N}$.
The dependence structure and underlying distribution seem to influence the quality of the estimate, with high dependence leading to a larger bias and a higher RMSE. The empirical bias is usually negative, thereby providing an explanation for the liberal behaviour of our test in Figure \ref{Fig:Empirical Size}. 
In contrast, the mean function seems to have little influence.
These results suggest that an improved estimator $\hat{\kappa}$ might yield even better results for our test procedure. However, the task of finding a suitable estimator is rather intricate since we only observe non-centred data $(X_i)_{i\in\N}$ and additionally require a convergence rate of $\sqrt{b_n}$.

\begin{table}
  \caption{Simulated Bias and RMSE of $\hat{\kappa}$ for the mean functions $\mu(x)=0$ (top row for each $n$), 
$\mu(x)=\sin(2\pi x)$ and $\mu(x)=\1_{\{1/2<x\leq1\}}$ (bottom row for each $n$) for different sample lengths $n$, $q=0.5$ and $s=0.7$. For the last mean function, the time series $(Z_i)_{i\in\N}$ is used. The observations are standardized to yield a theoretical long run variance $\kappa_{Y^2}=1$. }
    \begin{tabular}{|rc|rc|rc|rc|rc|rc|}
    \hline
\multicolumn{2}{|c|}{\multirow{2}{*}{N(0,1)}} & \multicolumn{2}{c|}{\multirow{2}{*}{Exp(1)}} &  \multicolumn{2}{c|}{\multirow{2}{*}{AR(1), 0.4}} &  \multicolumn{2}{c|}{\multirow{2}{*}{AR(1), 0.7}} &  \multicolumn{2}{c|}{\multirow{2}{*}{ARMA(2,2)}} &  \multicolumn{2}{c|}{\multirow{2}{*}{GARCH(1,1)}} \\
  \multicolumn{2}{|c|}{} & \multicolumn{2}{c|}{}  & \multicolumn{2}{c|}{} & \multicolumn{2}{c|}{} & \multicolumn{2}{c|}{} & \multicolumn{2}{c|}{} \\
  \multicolumn{1}{|c}{Bias} & RMSE &  \multicolumn{1}{c}{Bias} & RMSE &  \multicolumn{1}{c}{Bias} & RMSE & \multicolumn{1}{c}{Bias} & RMSE & \multicolumn{1}{c}{Bias} & RMSE & \multicolumn{1}{c}{Bias} & RMSE 
  \\ \hline
 \multicolumn{12}{|c|}{$n=500$}\\\hline
             \num{ -0.0445843072647825 } & \num{ 0.171809485426514 } & \num{ -0.130349317421188 } & \num{ 0.261134714558069 } & \num{ -0.089420712390942 } & \num{ 0.192161718670788 } & \num{ -0.19869728402943 } & \num{ 0.266636014981698 } & \num{ -0.0962693044113337 } & \num{ 0.200406243065377 } & \num{ -0.28508247356466 } & \num{ 0.349705454738594 } \\
        \num{ -0.00551859040899858 } & \num{ 0.173528784907639 } & \num{ -0.12689097803988 } & \num{ 0.2548018802145 } & \num{ -0.0543493618885507 } & \num{ 0.184793396890238 } & \num{ -0.18278342317257 } & \num{ 0.257038301942507 } & \num{ -0.0889315283857526 } & \num{ 0.197927485014663 } & \num{ -0.27526880688379 } & \num{ 0.343579269306306 } \\
                      \num{ -0.0172527264528704 } & \num{ 0.178940728890507 } & \num{ -0.120072288745491 } & \num{ 0.25632999390585 } & \num{ -0.0297240215878947 } & \num{ 0.174654894740891 } & \num{ -0.032622823189028 } & \num{ 0.173413461780066 } & \num{ -0.0868763383819581 } & \num{ 0.199615562952757 } & \num{ -0.273387240330414 } & \num{ 0.344258109940323 } \\
        \hline
        \multicolumn{12}{|c|}{$n=1000$}\\\hline
        \num{ -0.0292109947460065 } & \num{ 0.147129588320456 } & \num{ -0.100956764384342 } & \num{ 0.206107979273544 } & \num{ -0.0646675012844003 } & \num{ 0.159710240943472 } & \num{ -0.136769064156063 } & \num{ 0.208521234250035 } & \num{ -0.0627238682434367 } & \num{ 0.162899668314353 } & \num{ -0.21935125570182 } & \num{ 0.286549443984359 } \\
        \num{ -0.00396562697247614 } & \num{ 0.147088167866383 } & \num{ -0.101251857214619 } & \num{ 0.208608460336517 } & \num{ -0.035094407484249 } & \num{ 0.155114850245553 } & \num{ -0.128855889391169 } & \num{ 0.201080896765942 } & \num{ -0.0634811338209933 } & \num{ 0.164505227864862 } & \num{ -0.218915668809588 } & \num{ 0.282869814066289 } \\
                \num{ -0.000926651992208415 } & \num{ 0.15435850692379 } & \num{ -0.0908529202216113 } & \num{ 0.204282278977438 } & \num{ -0.0235065168898116 } & \num{ 0.150471790266004 } & \num{ -0.023938568982439 } & \num{ 0.149154064822289 } & \num{ -0.0594396779288074 } & \num{ 0.167620736254187 } & \num{ -0.222687711031767 } & \num{ 0.288952032494635 } \\      
        \hline    
        \multicolumn{12}{|c|}{$n=3000$}\\\hline
        \num{ -0.0136148133229635 } & \num{ 0.105688512067622 } & \num{ -0.0635544942594582 } & \num{ 0.14118045112482 } & \num{ -0.0316381687252705 } & \num{ 0.108383756984066 } & \num{ -0.0745549281990314 } & \num{ 0.13166058084684 } & \num{ -0.0354984207162781 } & \num{ 0.110863391842735 } & \num{ -0.145135167050922 } & \num{ 0.194757799290149 } \\
         \num{ -0.001669970423094 } & \num{ 0.106298271295523 } & \num{ -0.059608883788893 } & \num{ 0.137228847105866 } & \num{ -0.0160307145424692 } & \num{ 0.107468978489744 } & \num{ -0.0661334366914267 } & \num{ 0.130620641221908 } & \num{ -0.0323138326300929 } & \num{ 0.113605104954864 } & \num{ -0.140053199466725 } & \num{ 0.191994531971215 } \\
                \num{ 0.0119804099723556 } & \num{ 0.109948192590694 } & \num{ -0.051389989598946 } & \num{ 0.134794801254871 } & \num{ -0.0120166335934188 } & \num{ 0.105026912428303 } & \num{ -0.0115131606110773 } & \num{ 0.105787585408145 } & \num{ -0.0328282493445137 } & \num{ 0.115270410073187 } & \num{ -0.152352499421852 } & \num{ 0.197998612139115 } \\   
        \hline
    \end{tabular}
\label{Table: Performance of the LRV estimator}
\end{table}

\section{Data Example}
\label{Sec: Data Example}

As a data example, we consider the worldwide relative search interest for the topic ``global warming''  retrieved from Google Trends (\url{www.google.com/trends}). Data from Google Trends has frequently been used in environmental research as a measure for public interest, see, for instance, Anderegg and Goldsmith \citep{Anderegg.2014} and Burivalova, Butler and Wilcove \citep{Burivalova.2018}.
Google Trends does not provide the absolute search volume but adjusts the data in two steps: Each time period, the number of searches for the topic ``global warming'' as a proportion of the total searches within that time period is calculated. Afterwards, these proportions of the total searches are scaled, with the time period of the highest proportion being assigned the value 100 (for details, see the Google Trends FAQs (\url{https://support.google.com/trends/answer/4365533?hl=en&ref_topic=6248052}).

 \begin{figure}
 \includegraphics[width=\textwidth, trim={0 0 0 1.3cm}, clip]{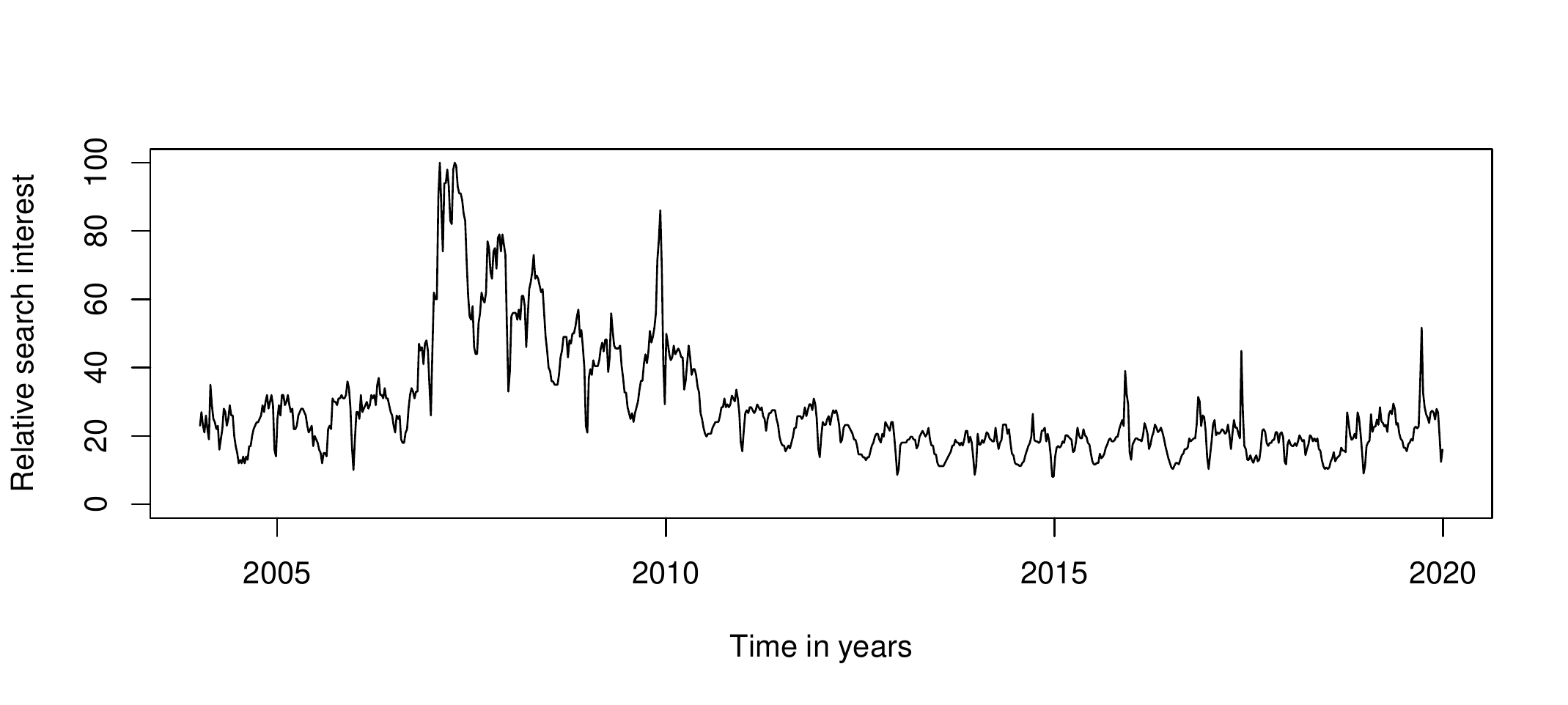}
\vspace*{-1cm}\caption{Weekly worldwide relative search interest for the topic ``global warming'' from January 2004 to December 2019 obtained from Google Trends.}
\label{Fig: Relative search interest}
\end{figure}
Figure \ref{Fig: Relative search interest} shows the weakly worldwide relative search interest from January 2004 to December 2019.  The increase in relative search interest between 2006/2007 and 2010/2011 is often related to the release of the documentary film ``An Inconvenient Truth''  in 2006 as well as the two media events colloquially often referred to as ``climategate'' and ``glaciergate'' in November 2009 and January 2010, respectively (see, \citep{Anderegg.2014} and  \citep{Burivalova.2018}). While a rise in the mean can be interpreted as a general rise of the public interest in the topic, an increase in the variance might hint at frequent news publication and an increased media coverage causing spikes and overall more fluctuation in the search interest.

\begin{figure}[H]
 \includegraphics[width=\textwidth, trim={0 0 0 1.3cm}, clip]{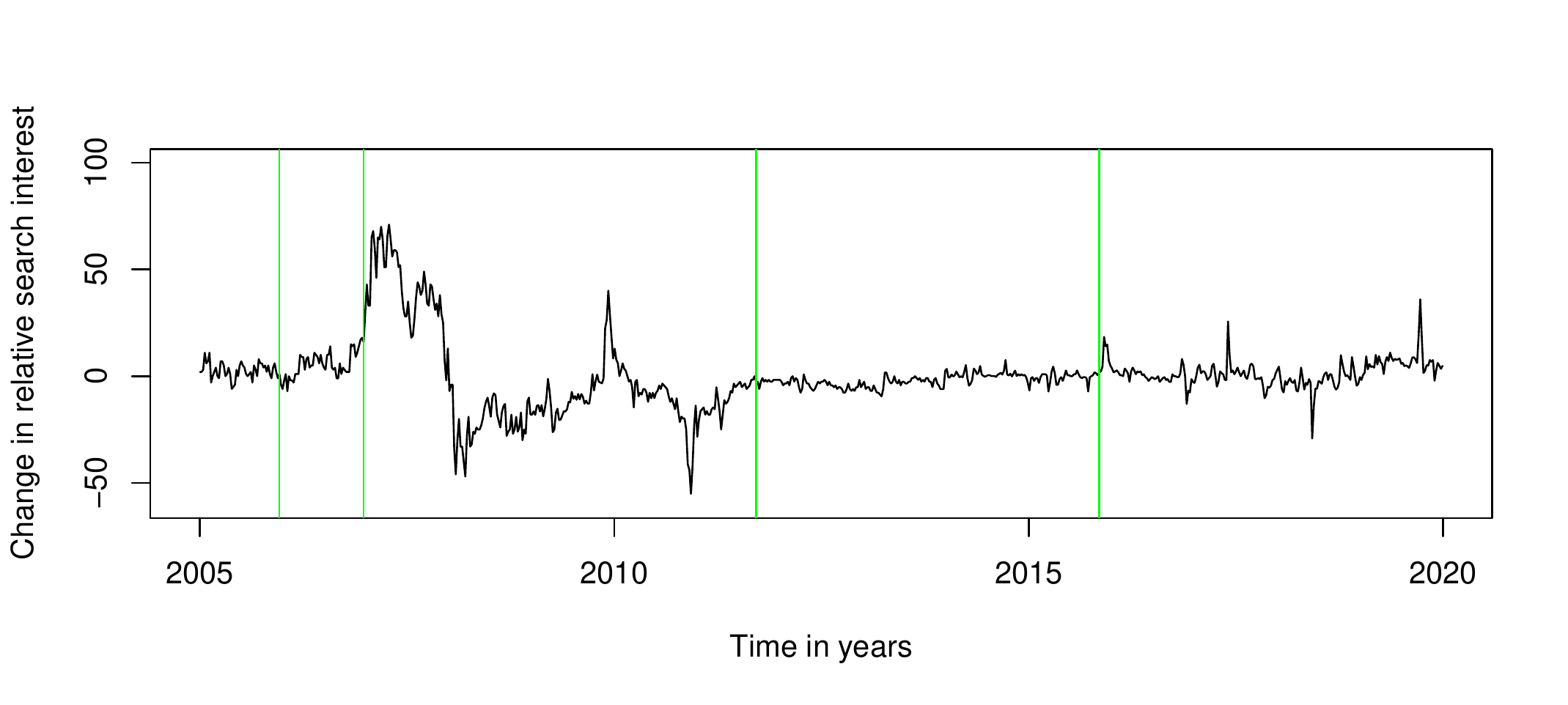}
\vspace*{-1cm}\caption{Change in the weekly worldwide relative search interest for the topic ``global warming'' from January 2005 to December 2019 obtained from Google Trends. Detected variance change points for $\alpha=5\%$ are marked by the vertical green lines.}
\label{Fig:Change in relative search interest}
\end{figure}

To eliminate seasonal effects, we work with the time series of annual differences $(Z_i)_{i\in\N}$, i.e.  $Z_i=X_i-X_{i-52}$, starting in January 2005, and thus consider changes in the relative search interest. After eliminating the last week in all years with 53 weeks (2006, 2012 and 2017) in the original time series $X$, we are left with $n=780$ differences. In the following, we test for a stationary variance in the series of annual differences. 

As we find our test  to behave rather liberal for small sample sizes in the simulation study in Section \ref{Subsec: Simulations Empirical Size}, we consider not only the usual significance level $\alpha=5\%$ but also $\alpha=1\%$. Figure \ref{Fig:Change in relative search interest} shows the seasonally  differenced observations $Z_1, \dots, Z_{780}$ as well as the detected changes in the variance for $\alpha=5\%$, which are estimated as described in Section \ref{Subsec: Estimation of Cp location} and located at the weeks of the 18.12.2005, 24.12.2006, 18.09.2011 and 08.11.2015. The estimated change points seem to capture the periods of increased variance of $(Z_i)_{i\in\N}$ quite well, thereby not only indicating changes during the time period of increased interest mentioned above, but additionally detecting a change in late 2015. For $\alpha=1\%$, we detect the same except for the last change point. 

Moreover, a close look at the data hints at a possibly non-stationary mean, with which our procedure seems to cope quite well due to the centering via the block means. A cumulative sum test for a change in mean does not reject the null hypothesis at $\alpha=1\%$, while at $\alpha=5\%$, it  sequentially detects a very large number of 11 changes when being combined with binary segmentation. There is hence considerable uncertainty concerning mean stationarity when applying such a standard test. In contrast, our test consistently indicates a non-stationary variance, irrespective of the chosen significance level.

 \section*{Acknowledgements}
The authors would like to thank the associate editor and the two anonymous referees for their careful reading of the manuscript and their insightful comments which  helped to improve the presentation of the paper significantly. The research was supported by the DFG Collaborative Research Center 823 ``Statistical Modelling of Nonlinear Dynamic Processes.'' The first author was additionally supported by the Friedrich-Ebert-Stiftung.

%
%

\appendix

\section{Proofs}
\label{Appendix: Proofs}
This appendix contains the proof details for our main theorems as well as the proofs of the auxiliary results. Its structure mainly follows that of Section \ref{Sec:Outline and Main Ideas of Proofs}. Some key results from the literature we frequently make use of are collected in Section \ref{Appendix: Subsec: Results from Literature}.

In our proofs, we will make frequent use of the shortening notation $\mu_i:= \mu\rbraces{\frac{i}{n}}$ and $\sigma_i:= \sigma\rbraces{\frac{i}{n}}$. Since $\sigma$ is a c\`adl\`ag-function, it is bounded on $[0,1]$ (see Chapter 14 in Billingsley \citep{billingsley1968}) and one can define $\sigma_{\sup}:= \sup_{t\in [0,1]} \sigma(t) <\infty$. Throughout, we will denote by $C$ a positive constant whose value is of no significance and may even change from line to line.
\subsection{A first Approximation}
\label{Appendix: Subsec: non-zero mean}
Note that we require the mean function $\mu: [0,1]\rightarrow \R$ to be Lipschitz-continuous on $[0,1]$ and therefore, we will use the property $\abs{\mu_i-\mu_r}\leq L \frac{\abs{i-r}}{n}$  without always specifically mentioning it. Without loss of generality, we set the Lipschitz-constant $L$ to be equal to one. Note that this implies
\begin{equation*}
\abs{\mu_i-\frac{1}{\ell_n}\sum_{r=(j-1)\ell_n+1}^{j\ell_n} \mu_r }\leq   \frac{1}{\ell_n}\sum_{r=(j-1)\ell_n+1}^{j\ell_n} \abs{\mu_i-\mu_r} \leq \frac{\ell_n}{n}=\frac{1}{b_n}
\end{equation*}
since $\abs{i-r}\leq \ell_n$ whenever $i,r \in  \{(j-1)\ell_n+1, ..., j\ell_n\}$.

\begin{proof}[Proof of Proposition 3.1]

\begin{align*}
&\sqrt{b_n \ell_n} \abs{U(n)-U_1(n)}\\
= & \sqrt{\ell_n b_n} \abs{\frac{1}{b_n(b_n-1)}\sum_{1\leq j\neq k\leq b_n} \rbraces{\abs{\log \hat{\sigma}_j^2 - \log\hat{\sigma}_k^2} - \abs{\log s_j^2 - \log s_k^2}} } \\
\leq  &\sqrt{\ell_n b_n} \frac{1}{b_n(b_n-1)}\sum_{1\leq j\neq k\leq b_n} \rbraces{\abs{\log \hat{\sigma}_j^2 - \log s_j^2} + \abs{\log\hat{\sigma}_k^2-\log s_k^2 } } \\
= &  2\frac{\sqrt{\ell_n}}{\sqrt{b_n}}\sum_{1\leq j\leq b_n} \abs{\log \hat{\sigma}_j^2 - \log s_j^2} \leq 2\frac{\sqrt{\ell_n}}{\sqrt{b_n}}\sum_{1\leq j\leq b_n}\frac{1}{\hat{\sigma}_j^2 \wedge s_j^2} \abs{ \hat{\sigma}_j^2 -s_j^2}
\end{align*}
As a first step towards showing convergence in probability of the entire expression, we will prove $\frac{\sqrt{\ell_n}}{\sqrt{b_n}}\sum_{1\leq j\leq b_n} \abs{ \hat{\sigma}_j^2 -s_j^2}\pConv 0 $. It holds
{\allowdisplaybreaks
\begin{align*}
&\frac{\sqrt{\ell_n}}{\sqrt{b_n}}\sum_{1\leq j\leq b_n} \abs{ \hat{\sigma}_j^2 -s_j^2}\\
= & \frac{\sqrt{\ell_n}}{\sqrt{b_n}}\sum_{1\leq j\leq b_n} \abs{\frac{1}{\ell_n}\sum_{i=(j-1)\ell_n+1}^{j\ell_n}\rbraces{X_i-\frac{1}{\ell_n}\sum_{r=(j-1)\ell_n+1}^{j\ell_n}X_r}^2- \frac{1}{\ell_n}\sum_{i=(j-1)\ell_n+1}^{j\ell_n}\rbraces{X_i-\mu_i}^2}\\
= & \frac{\sqrt{\ell_n}}{\sqrt{b_n}}\sum_{1\leq j\leq b_n} \abs{ \frac{1}{\ell_n}\sum_{i=(j-1)\ell_n+1}^{j\ell_n}\rbraces{\sigma_iY_i+\mu_i-\frac{1}{\ell_n}\sum_{r=(j-1)\ell_n+1}^{j\ell_n}\rbraces{\sigma_r Y_r+\mu_r}}^2- \frac{1}{\ell_n}\sum_{i=(j-1)\ell_n+1}^{j\ell_n}\rbraces{\sigma_i Y_i}^2}\\
\leq & \frac{\sqrt{\ell_n}}{\sqrt{b_n}} \sum_{1\leq j\leq b_n} \abs{\frac{1}{\ell_n}\sum_{i=(j-1)\ell_n+1}^{j\ell_n} \sigma_iY_i}^2 + 2 \frac{\sqrt{\ell_n}}{\sqrt{b_n}}\sum_{1\leq j\leq b_n} \abs{\frac{1}{\ell_n}\sum_{i=(j-1)\ell_n+1}^{j\ell_n} \sigma_iY_i \rbraces{\mu_i-\frac{1}{\ell_n}\sum_{r=(j-1)\ell_n+1}^{j\ell_n} \mu_r}} \\
& + \frac{\sqrt{\ell_n}}{\sqrt{b_n}}  \sum_{1\leq j\leq b_n} \abs{\frac{1}{\ell_n}\sum_{i=(j-1)\ell_n+1}^{j\ell_n} \rbraces{\mu_i-\frac{1}{\ell_n}\sum_{r=(j-1)\ell_n+1}^{j\ell_n} \mu_r }^2}\\
& + 2 \frac{\sqrt{\ell_n}}{\sqrt{b_n}} \sum_{1\leq j\leq b_n} \abs{\rbraces{\frac{1}{\ell_n}\sum_{i=(j-1)\ell_n+1}^{j\ell_n} \rbraces{\mu_i-\frac{1}{\ell_n}\sum_{r=(j-1)\ell_n+1}^{j\ell_n} \mu_r }} \rbraces{\frac{1}{\ell_n}\sum_{i=(j-1)\ell_n+1}^{j\ell_n} \sigma_i Y_i } }.
\end{align*}}

A sharp look reveals that the fourth of the above terms equals zero. We will verify that the remaining three expressions likewise converge towards zero in probability. For the first one, we obtain by an application of Markov's inequality for every $\varepsilon^*>0$ that
{\allowdisplaybreaks
\begin{align*}
&\Pb\rbraces{ \frac{\sqrt{\ell_n}}{\sqrt{b_n}} \sum_{1\leq j\leq b_n} \abs{\frac{1}{\ell_n}\sum_{i=(j-1)\ell_n+1}^{j\ell_n} \sigma_i Y_i}^2>\varepsilon^*} \\
 \leq & \frac{\sqrt{\ell_n b_n}}{\varepsilon^*} \max_{1\leq j \leq b_n} \E{ \abs{\frac{1}{\ell_n}\sum_{i=(j-1)\ell_n+1}^{j\ell_n} \sigma_i Y_i}^2}\\
= & \frac{\sqrt{\ell_n b_n}}{\varepsilon^*} \max_{1\leq j \leq b_n} \frac{1}{\ell_n^2 }\sum_{i=(j-1)\ell_n+1}^{j\ell_n} \sum_{r=(j-1)\ell_n+1}^{j\ell_n}\sigma_i \sigma_r\cov{Y_i}{Y_r}\\
\leq &\frac{\sqrt{\ell_n b_n}}{\varepsilon^*} \sigma_{\sup}^2\frac{1}{\ell_n^2 }\sum_{i=1}^{\ell_n}\sum_{r=1}^{\ell_n}\abs{\cov{Y_i}{Y_r}}\\
=& \frac{\sigma_{\sup}^2}{\varepsilon^*}\frac{\sqrt{b_n}}{\sqrt{\ell_n}} \rbraces{\Var{Y_1}+2\sum_{i=1}^{\ell_n-1}\frac{\ell_n-i}{\ell_n}\abs{\cov{Y_1}{Y_{i+1}}}}\\
\leq &\frac{\sigma_{\sup}^2}{\varepsilon^*}\frac{\sqrt{b_n}}{\sqrt{\ell_n}} \rbraces{\Var{Y_1}+2\sum_{i=1}^{\infty}\abs{\cov{Y_1}{Y_{i+1}}}}.
\end{align*}}
Due to an application of Davydov's covariance inequality for strongly mixing sequences 
together with the summability condition on the mixing coefficients $\beta_Y$ and the finiteness of the $(2+\vartheta)$-th moments of $Y_i$, the sum of absolute covariances is finite,
\begin{align*}
\sum_{i=1}^{\infty}\abs{\cov{Y_1}{Y_{i+1}}} \leq & 8 \sum_{i=1}^{\infty} \alpha\rbraces{\sigma(Y_1),\sigma(Y_{i+1})}^{\frac{\vartheta}{2+\vartheta}} \| Y_1 \|_{2+\vartheta} \| Y_{i+1} \|_{2+\vartheta} \\
 \leq & 8 \| Y_1 \|_{2+\vartheta}^2 \sum_{i=1}^{\infty} \beta_Y(i)^{\frac{\vartheta}{2+\vartheta}}<\infty.
\end{align*}
Since $s>0.5$, we have thus shown the convergence of the first term. For the second term, we have for every $\varepsilon^*>0$

{\allowdisplaybreaks
\begin{align*}
&\Pb\rbraces{\frac{\sqrt{\ell_n}}{\sqrt{b_n}}\sum_{1\leq j\leq b_n} \abs{\frac{1}{\ell_n}\sum_{i=(j-1)\ell_n+1}^{j\ell_n} \sigma_i Y_i \rbraces{\mu_i-\frac{1}{\ell_n}\sum_{r=(j-1)\ell_n+1}^{j\ell_n} \mu_r }} >\varepsilon^*}\\
\leq & \frac{1}{\varepsilon^*} \frac{\sqrt{\ell_n}}{\sqrt{b_n}} \sum_{1\leq j\leq b_n} \E{\abs{\frac{1}{\ell_n}\sum_{i=(j-1)\ell_n+1}^{j\ell_n} \sigma_i Y_i \rbraces{\mu_i-\frac{1}{\ell_n}\sum_{r=(j-1)\ell_n+1}^{j\ell_n} \mu_r }}}\\
\leq &\frac{\sqrt{\ell_nb_n}}{\varepsilon^*}  \max_{1\leq j\leq b_n} \rbraces{\E{\abs{\frac{1}{\ell_n}\sum_{i=(j-1)\ell_n+1}^{j\ell_n} \sigma_i Y_i \rbraces{\mu_i-\frac{1}{\ell_n}\sum_{r=(j-1)\ell_n+1}^{j\ell_n} \mu_r }}^2}}^{1/2}\\
\leq  & \frac{\sqrt{\ell_nb_n}}{\varepsilon^*}  \max_{1\leq j\leq b_n} \left( \frac{1}{\ell_n^2}\sum_{i=(j-1)\ell_n+1}^{j\ell_n} \sum_{k=(j-1)\ell_n+1}^{j\ell_n} \sigma_i\sigma_k \abs{\cov{Y_i}{Y_k}}\right. \\
& \quad \cdot \left.\abs{\mu_i-\frac{1}{\ell_n}\sum_{r=(j-1)\ell_n+1}^{j\ell_n} \mu_r }\abs{\mu_k-\frac{1}{\ell_n}\sum_{r=(j-1)\ell_n+1}^{j\ell_n} \mu_r }\right)^{1/2}\\
\leq & \frac{\sqrt{\ell_nb_n}}{b_n \sqrt{\ell_n}\varepsilon^*}\sigma_{\sup} \rbraces{\frac{1}{\ell_n}\sum_{i=1}^{\ell_n} \sum_{k=1}^{\ell_n} \abs{\cov{Y_i}{Y_k}} }^{1/2} \\
\leq & \frac{1}{\varepsilon^* \sqrt{b_n}}\rbraces{\Var{Y_1}+2\sum_{i=1}^{\infty}\abs{\cov{Y_1}{Y_{i+1}}}}^{1/2}
\end{align*}}
which tends towards zero as argued above. The third term is purely deterministic and can be controlled by utilising the Lipschitz-property of $\mu$ via
\begin{align*}
&\frac{\sqrt{\ell_n}}{\sqrt{b_n}}  \sum_{1\leq j\leq b_n} \abs{\frac{1}{\ell_n}\sum_{i=(j-1)\ell_n+1}^{j\ell_n} \rbraces{\mu_i-\frac{1}{\ell_n}\sum_{r=(j-1)\ell_n+1}^{j\ell_n} \mu_r }^2}\leq \frac{\sqrt{\ell_n}}{\sqrt{b_n^3}}
\end{align*}
which converges towards zero as long as we choose $s<3/4$. 
Finally, we return to our original claim and define $\tilde{s}_j^2:= \frac{1}{\ell_n}\sum_{i=(j-1)\ell_n+1}^{j\ell_n}\E{(X_i-\mu_i)^2}$. Note that it holds
$$\tilde{s}_j^2=\E{s_j^2}=\frac{1}{\ell_n}\sum_{i=(j-1)\ell_n+1}^{j\ell_n}\sigma_i^2\E{Y_i^2}\in [\sigma_0^2, \sigma_{\sup}^2]$$
 and hence, for every $\varepsilon^*>0$,
 {\allowdisplaybreaks
\begin{align*}
&\Pb\rbraces{\frac{\sqrt{\ell_n}}{\sqrt{b_n}}\sum_{1\leq j\leq b_n} \frac{1}{s_j^2\wedge \hat{\sigma}_{j}^2} \abs{s_j^2-\hat{\sigma}_{j}^2}\geq \varepsilon^*}\\
\leq & \Pb\left(\left\{\frac{\sqrt{\ell_n}}{\sqrt{b_n}}\sum_{1\leq j\leq b_n} \frac{1}{s_j^2\wedge \hat{\sigma}_{j}^2} \abs{s_j^2-\hat{\sigma}_{j}^2}\geq \varepsilon^*\right\}\cap \left\{\max_{1\leq j \leq b_n} \abs{s_j^2-\hat{\sigma}_{j}^2} \leq \frac{\sigma_0^2}{4} \right\}\right.\\
 & \left.\cap \left\{\max_{1\leq j \leq b_n} \abs{s_j^2-\tilde{s}_{j}^2} \leq \frac{\sigma_0^2}{4} \right\}\right)
+ \Pb \rbraces{ \max_{1\leq j \leq b_n} \abs{s_j^2-\hat{\sigma}_{j}^2} > \frac{\sigma_0^2}{4} }\\
&+ \Pb\rbraces{\max_{1\leq j \leq b_n} \abs{s_j^2-\tilde{s}_{j}^2} > \frac{\sigma_0^2}{4}}\\
\leq & \Pb\rbraces{\frac{\sqrt{\ell_n}}{\sqrt{b_n}}\sum_{1\leq j\leq b_n} \frac{2}{\sigma_0^2} \abs{s_j^2-\hat{\sigma}_{j}^2}\geq \varepsilon^*}+ \Pb \rbraces{ \max_{1\leq j \leq b_n} \abs{s_j^2-\hat{\sigma}_{j}^2} > \frac{\sigma_0^2}{4} }\\
& + \Pb\rbraces{\max_{1\leq j \leq b_n} \abs{s_j^2-\tilde{s}_{j}^2} > \frac{\sigma_0^2}{4}}.
\end{align*}}

The convergence of the first probability towards zero has been shown in the first part of this proof. Moreover, due to this convergence, the second probability converges towards zero as well. Lastly, we turn towards the third term,
\begin{align*}
 & \Pb\rbraces{\max_{1\leq j\leq b_n} \abs{s^2_j-  \tilde{s}^2_j} \geq \varepsilon^*}
  \leq \sum_{1\leq j \leq b_n} \Pb\rbraces{ \abs{ s^2_j-  \tilde{s}^2_j}\geq \varepsilon^*} \\
\leq &  \frac{1}{{\varepsilon^*}^2} \sum_{1\leq j \leq b_n} \E{\abs{ s^2_j- \tilde{s}^2_j}^{2}}
=  \frac{1}{{\varepsilon^*}^2 \ell_n^{2}} \sum_{1\leq j \leq b_n} \E{\abs{\sum_{i=(j-1)\ell_n+1}^{j\ell_n}\rbraces{\sigma^2_i Y_i^2-\sigma^2_i\E{Y_i^2}} }^{2}} \\
\leq & \frac{1}{{\varepsilon^*}^2 \ell_n^{2}}  \sum_{1\leq j \leq b_n} \sum_{i=(j-1)\ell_n+1}^{j\ell_n}\sum_{k=(j-1)\ell_n+1}^{j\ell_n} \sigma_i^2 \sigma_k^2 \abs{\cov{Y_i^2}{Y_k^2}}\\
\leq & \frac{b_n}{{\varepsilon^*}^2 \ell_n} \sigma_{\sup}^4 \rbraces{\Var{Y_1^2}+2\sum_{i=1}^\infty  \abs{\cov{Y_1^2}{Y_{i+1}^2}}}.
\end{align*}
Since $(Y_i^2)_{i\in \N}$ is absolutely regular with $\beta_{Y^2}(k)\leq \beta_Y(k)$ for all $k\in \N$, assumptions (A1) and (A2) in combination with Davydov's covariance inequality yield that the sum of covariances is finite and, consequently, the third probability converges towards zero, which finishes the proof.
\end{proof}

\subsection{Proof of Theorem 2.3}
\label{Appendix: subsec: Proof of THM: Behaviour under A for U(n)}
\begin{proof}
Recall the definitions
\begin{equation*}
s_j^2:= \frac{1}{\ell_n} \sum_{i=(j-1)\ell_n+1}^{j\ell_n}(X_i-\mu_i)^2 \quad \text{and} \quad \tilde{s}^2_j:= \frac{1}{\ell_n} \sum_{i=(j-1)\ell_n+1}^{j\ell_n}\E{(X_i-\mu_i)^2}.
\end{equation*}
Our intention is to approximate the test statistic $U_1(n)$ by a deterministic term via
{\allowdisplaybreaks
\begin{align*}
& \abs{U_1(n)- \frac{1}{b_n(b_n-1)} \sum_{1\leq j\neq k\leq b_n} \abs{\log \sigma^2\rbraces{\frac{j\ell_n}{n}}- \log \sigma^2\rbraces{\frac{k\ell_n}{n}}}}\\
=& \abs{\frac{1}{b_n(b_n-1)} \sum_{1\leq j\neq k\leq b_n} \rbraces{ \abs{\log s^2_j- \log s^2_k} -\abs{\log \sigma^2\rbraces{\frac{j\ell_n}{n}}- \log \sigma^2\rbraces{\frac{k\ell_n}{n}}}} }\\
\leq &  \frac{1}{b_n(b_n-1)} \sum_{1\leq j\neq k\leq b_n} \abs{ \rbraces{\log s^2_j- \log \sigma^2\rbraces{\frac{j\ell_n}{n}}} -\rbraces{\log s^2_k- \log \sigma^2\rbraces{\frac{k\ell_n}{n}}}} \\
\leq &  \frac{1}{b_n(b_n-1)} \sum_{1\leq j\neq k\leq b_n} \rbraces{\abs{ \log s^2_j- \log \sigma^2\rbraces{\frac{j\ell_n}{n}}} +\abs{\log s^2_k- \log \sigma^2\rbraces{\frac{k\ell_n}{n}}} }\\
= & \frac{2}{b_n} \sum_{1\leq j \leq b_n} \abs{\log s^2_j- \log \sigma^2\rbraces{\frac{j\ell_n}{n}}}\\
\leq & \frac{2}{b_n} \sum_{1\leq j \leq b_n} \abs{\log s^2_j- \log \tilde{s}^2_j}
+ \frac{2}{b_n} \sum_{1\leq j \leq b_n} \abs{\log \tilde{s}^2_j- \log \sigma^2\rbraces{\frac{j\ell_n}{n}}}.
\end{align*}}
We will show that the last two sums converge towards zero. For the first part, we obtain
\begin{align*}
\frac{1}{b_n} \sum_{1\leq j \leq b_n} \abs{\log s^2_j- \log \tilde{s}^2_j} \leq  \max_{1\leq j\leq b_n} \abs{\log s^2_j- \log \tilde{s}^2_j} \leq  \max_{1\leq j\leq b_n} \frac{1}{s^{2}_j \wedge \tilde{s}^{2}_j}\abs{s^2_j- \tilde{s}^2_j}.
\end{align*}
In the proof of Proposition 3.1,  
we have already shown that for every $\varepsilon^*>0$, it holds $\Pb\rbraces{\max_{1\leq j\leq b_n} \abs{s^2_j-  \tilde{s}^2_j} \geq \varepsilon^*} \rightarrow 0$. Consequently, the entire first sum converges towards zero in probability due to
\begin{align*}
& \Pb\rbraces{\max_{1\leq j\leq b_n} \frac{1}{s^{2}_j \wedge \tilde{s}^{2}_j}\abs{s^2_j- \tilde{s}^2_j}\geq \varepsilon^* }\\
\leq & \Pb\rbraces{\left\{\max_{1\leq j\leq b_n} \frac{1}{s^{2}_j \wedge \tilde{s}^{2}_j}\abs{s^2_j- \tilde{s}^2_j}\geq \varepsilon^*\right\} \cap \left\{\max_{1\leq j\leq b_n} \abs{s^2_j- \tilde{s}^2_j}>\frac{\sigma_0}{4} \right\} } \\
& +  \Pb\rbraces{\left\{\max_{1\leq j\leq b_n} \frac{1}{s^{2}_j \wedge \tilde{s}^{2}_j}\abs{s^2_j- \tilde{s}^2_j}\geq \varepsilon^*\right\} \cap \left\{\max_{1\leq j\leq b_n} \abs{s^2_j- \tilde{s}^2_j}\leq \frac{\sigma_0}{4} \right\} }\\
 \leq & \Pb\rbraces{\max_{1\leq j\leq b_n} \abs{s^2_j- \tilde{s}^2_j}>\frac{\sigma_0}{4} } +  \Pb\rbraces{\max_{1\leq j\leq b_n} \frac{1}{(\tilde{s}^{2}_j-\sigma_0/4) \wedge \tilde{s}^{2}_j}\abs{s^2_j- \tilde{s}^2_j}\geq \varepsilon^* }
\end{align*}
which tends towards zero since $\tilde{s}_j^2$ is deterministic and bounded from below by $\sigma_0^2$. We now turn towards the second part, which is purely deterministic. As a first step to proving
\begin{equation*}
 \frac{1}{b_n} \sum_{1\leq j \leq b_n} \abs{\log \tilde{s}^2_j- \log \sigma^2\rbraces{\frac{j\ell_n}{n}}} \rightarrow 0 \quad \text{as } n \rightarrow \infty,
\end{equation*}
we bound each summand by
\begin{align*}
& \abs{\log \tilde{s}^2_j- \log \sigma^2\rbraces{\frac{j\ell_n}{n}}}
\leq \frac{1}{\sigma_0^{2}}\abs{ \tilde{s}^2_j-\sigma^2\rbraces{\frac{j\ell_n}{n}}}\\
\leq  &  \frac{1}{\sigma_0^{2}} \rbraces{\frac{1}{\ell_n}\sum_{i=(j-1)\ell_n+1}^{j\ell_n}\abs{  \sigma^2\rbraces{\frac{i}{n}} - \sigma^2\rbraces{\frac{j\ell_n}{n}}}}.
\end{align*}

Note that for every c\`adl\`ag-function $\sigma$ and every $\varepsilon>0$, there exist an $r\in \N$ and $0=t_0<t_1< ... <t_r=1$  such that
\begin{equation*}
\max_{1\leq i\leq r} \sup_{s,t \in [t_{i-1},t_i)} \abs{\sigma^2(s)-\sigma^2(t)} <\varepsilon
\end{equation*}
(see, Lemma 1, Section 14, in \citep{billingsley1968}) due to $\sigma^2$ being c\`adl\`ag as well. In particular, there are at most $r$ jumps larger than $\varepsilon$ and at most $r$ of the $b_n$ intervals $\left( \frac{(j-1)\ell_n}{n} , \frac{j\ell_n}{n} \right]$  for $1\leq j \leq b_n$ lie in more than one of the intervals $[t_{i-1},t_i)$. In these at most $r$ cases, it holds
\begin{equation*}
\frac{1}{\sigma_0^{2}} \rbraces{\frac{1}{\ell_n}\sum_{i=(j-1)\ell_n+1}^{j\ell_n}\abs{  \sigma^2\rbraces{\frac{i}{n}} - \sigma^2\rbraces{\frac{j\ell_n}{n}}}}\leq 2 \frac{\sigma_{\sup}^2}{\sigma_0^{2}}
\end{equation*}
due to $\sigma$ being bounded.
In the other $b_n-r$ cases, one obtains
\begin{align*}
&\frac{1}{\sigma_0^{2}} \rbraces{\frac{1}{\ell_n}\sum_{i=(j-1)\ell_n+1}^{j\ell_n}\abs{  \sigma^2\rbraces{\frac{i}{n}} - \sigma^2\rbraces{\frac{j\ell_n}{n}}}}\\
\leq & \frac{1}{\sigma_0^{2}} \rbraces{\max_{1\leq i\leq r} \sup_{s,t \in [t_{i-1},t_i)} \abs{\sigma^2(s)-\sigma^2(t)}}\leq  \frac{1}{\sigma_0^{2}} \varepsilon.
\end{align*}
Together, it holds
\begin{align*}
& \frac{1}{b_n} \sum_{1\leq j \leq b_n} \abs{\log \tilde{s}^2_j- \log \sigma^2\rbraces{\frac{j\ell_n}{n}}}
\leq  2 \frac{\sigma_{\sup}^2}{\sigma_0^{2}} \cdot \frac{r}{b_n}+  \frac{1}{\sigma_0^{2}} \varepsilon \cdot \frac{b_n-r}{b_n}\leq  \varepsilon
\end{align*}
 for $b_n$ chosen large enough. Hence, the absolute difference between $U_1(n)$ and its deterministic approximation converges towards zero in probability. To moreover prove the convergence towards the desired integral, it remains to verify
 \begin{align*}
& \frac{1}{b_n(b_n-1)} \sum_{1\leq j\neq k\leq b_n} \abs{\log \sigma^2\rbraces{\frac{j\ell_n}{n}}- \log \sigma^2\rbraces{\frac{k\ell_n}{n}}}\\
 \longrightarrow & \int_0^1 \int_0^1 \abs{\log \sigma^2(x)-\log \sigma^2(y)} \mathrm{d}x\mathrm{d}y.
 \end{align*}
Being a c\`adl\`ag-function, $\sigma^2$ has at most countably many points of discontinuity on the interval $[0,1]$ (see, Section 14 in \citep{billingsley1968}). Consequently, the points of discontinuity on $[0,1]\times[0,1]$ of the two-dimensional function $g:[0,1]\times[0,1]\rightarrow \R_+$ defined by
\begin{equation*}
g(x,y):=\abs{\log\sigma^2(x)-\log \sigma^2(y)}
\end{equation*}
constitute a Lebesgue null set. Additionally, $g$ is bounded, due to $\sigma^2$ being bounded from above by $\sigma_{\sup}^2$ and from below by $\sigma_0^2>0$.  By Lebesgue's integrability criterion, the function $g$ is hence Riemann-integrable and the sequence of Riemann sums 
\begin{align*}
& \frac{1}{b_n(b_n-1)} \sum_{1\leq j\neq k\leq b_n} \abs{\log \sigma^2\rbraces{\frac{j\ell_n}{n}}- \log \sigma^2\rbraces{\frac{k\ell_n}{n}}}\\
= & \frac{b_n}{b_n-1} \cdot \sum_{1\leq j\leq b_n} \sum_{1\leq k\leq b_n} \abs{\log \sigma^2\rbraces{\frac{j}{b_n}}- \log \sigma^2\rbraces{\frac{k}{b_n}}} \cdot \frac{1}{b_n^2}
\end{align*}
 converges towards the desired integral as $n\rightarrow \infty$.
\end{proof}

\subsection{Two further Approximations}

\begin{proof}[Proof of Proposition 3.2]
Note that under the hypothesis, it holds $s_j^2=\sigma_H^2\rbraces{S_j+1}$ and we can thus bound the expression of interest from above via
\begin{align*}
&\sqrt{b_n}\sqrt{\ell_n} \abs{U_1(n)-U_2(n)}\\
= & \frac{\sqrt{b_n}\sqrt{\ell_n}}{b_n(b_n-1)}\abs{\sum_{1\leq j \neq k \leq b_n} \abs{ \log \rbraces{S_j+1}-\log \rbraces{S_k+1}}- \sum_{1\leq j \neq k \leq b_n} \abs{S_j-S_k} }\\
\leq & \frac{\sqrt{\ell_n}}{\sqrt{b_n}(b_n-1)}\sum_{1\leq j \neq k \leq b_n} \rbraces{ \abs{\log \rbraces{ S_j+1} - S_j} +\abs{\log \rbraces{S_k+1}- S_k} }\\
=& 2\frac{\sqrt{\ell_n}}{\sqrt{b_n}} \sum_{1\leq j\leq b_n} \abs{\log \rbraces{ S_j+1} - S_j}.
\end{align*}

We will further bound the last term by means of a Taylor-expansion for the logarithm. 
Define for a fixed $\varepsilon\in (0,1)$ the set $M_\varepsilon:= \{ \max_{1\leq j\leq b_n} \abs{S_j}\leq \varepsilon  \}$, on which it holds

\begin{equation*}
\1_{M_\varepsilon} \cdot \frac{\sqrt{\ell_n}}{\sqrt{b_n}} \sum_{1\leq j\leq b_n} \abs{\log \rbraces{ S_j+1} - S_j} \leq \1_{M_\varepsilon} \cdot  C \frac{\sqrt{\ell_n}}{\sqrt{b_n}} \sum_{1\leq j\leq b_n}  \abs{S_j}^{2}.
\end{equation*}

By Markov's inequality, we obtain for  $\varepsilon^*>0$
\begin{align*}
&\Pb\rbraces{1_{M_\varepsilon} \cdot    \frac{\sqrt{\ell_n}}{\sqrt{b_n}}\sum_{1\leq j  \leq b_n} \abs{S_j}^{2}>\varepsilon^*  }
\leq \Pb\rbraces{  \frac{\sqrt{\ell_n}}{\sqrt{b_n}}\sum_{1\leq j  \leq b_n} \abs{S_j}^{2}>\varepsilon^* }\\
\leq & \frac{1}{\varepsilon^*}  \frac{\sqrt{\ell_n}}{\sqrt{b_n}} \sum_{1\leq j  \leq b_n} \E{\abs{S_j}^{2}}
=  \frac{1}{\varepsilon^*}  \frac{\sqrt{b_n}}{\sqrt{\ell_n}}\Var{\sqrt{\ell_n}S_1},
\end{align*}
which tends towards zero due to the variance converging towards $\kappa^2$ and $s>0.5$. On $M_\varepsilon^C$, we also obtain convergence in probability due to the set itself being asymptotically negligible,
\begin{equation*}
\Pb\rbraces{M_\varepsilon^C}= \Pb\rbraces{\max_{1\leq j\leq b_n} \abs{S_j}>\varepsilon}\leq b_n \cdot \Pb\rbraces{\abs{S_1}>\varepsilon}\leq \frac{b_n}{\ell_n} \cdot\frac{\Var{\sqrt{\ell_n}S_1}}{\varepsilon^2} \longrightarrow 0.
\end{equation*}
\end{proof}

\begin{proof}[Proof of Proposition 3.3]
{
\allowdisplaybreaks{
Define
\begin{equation*}
\tilde{S}_{j}:=\tilde{S}_{j,n}:=\frac{1}{\ell_n} \sum_{i=(j-1)\ell_n+1}^{j\ell_n-m_n}\rbraces{Y_i^2-\E{Y_i^2}}
\end{equation*}
and note that
\begin{align*}
&\sqrt{b_n} \sqrt{\ell_n}\abs{U_2(n)-U_3(n)} \\
&= \abs{\frac{\sqrt{\ell_n}}{\sqrt{b_n}(b_n-1)}\sum_{1\leq j\neq k\leq b_n} \rbraces{\abs{S_j-S_k}-\abs{\tilde{S}'_j-\tilde{S}'_k}}} \\
&\leq \frac{\sqrt{\ell_n}}{\sqrt{b_n}(b_n-1)}\sum_{1\leq j\neq k\leq b_n} \abs{\rbraces{S_j-S_k}-\rbraces{\tilde{S}'_j-\tilde{S}'_k}} \\
&\leq \frac{\sqrt{\ell_n}}{\sqrt{b_n}(b_n-1)}\sum_{1\leq j\neq k\leq b_n} \rbraces{\abs{S_j- \tilde{S}_j}+ \abs{\tilde{S}_j-\tilde{S}'_j}+\abs{S_k-\tilde{S}_k}+\abs{\tilde{S}_k-\tilde{S}'_k}} \\
&=\frac{2 \sqrt{\ell_n}}{\sqrt{b_n}} \sum_{1\leq j\leq b_n} \rbraces{\abs{S_j- \tilde{S}_j}+ \abs{\tilde{S}_j-\tilde{S}'_j}}.
\end{align*}}

We will show that both $\frac{\sqrt{\ell_n}}{\sqrt{b_n}} \sum_{1\leq j\leq b_n} \abs{S_j- \tilde{S}_j}$ and $\frac{ \sqrt{\ell_n}}{\sqrt{b_n}} \sum_{1\leq j\leq b_n} \abs{\tilde{S}_j-\tilde{S}'_j}$ converge in probability towards zero. The first sum can be rewritten as
\begin{equation*}
\frac{\sqrt{\ell_n}}{\sqrt{b_n}} \sum_{1\leq j\leq b_n} \abs{S_j- \tilde{S}_j}=\frac{1}{\sqrt{b_n}} \sum_{1\leq j\leq b_n} \abs{ \frac{1}{\sqrt{\ell_n}}\sum_{i=j\ell_n-m_n+1}^{j\ell_n}\rbraces{Y_i^2-\E{Y_i^2}}}
\end{equation*}
and by stationarity, it holds
\begin{align*}
&\Pb\rbraces{\frac{ \sqrt{\ell_n}}{\sqrt{b_n}} \sum_{1\leq j\leq b_n} \abs{S_j- \tilde{S}_j}>\varepsilon^*}
\leq  \frac{1}{\varepsilon^*} \frac{ \sqrt{\ell_n}}{\sqrt{b_n}} \sum_{1\leq j\leq b_n} \E{\abs{S_j- \tilde{S}_j}}\\
=&  \frac{1}{\varepsilon^*} \frac{\sqrt{b_n}}{\sqrt{\ell_n}}\E{\abs{\sum_{i=\ell_n-m_n+1}^{\ell_n}\rbraces{Y_i^2-\E{Y_i^2}}}}\\
\leq &\frac{1}{\varepsilon^*}\frac{\sqrt{b_n m_n}}{\sqrt{\ell_n}} \rbraces{\E{\abs{\frac{1}{\sqrt{m_n}}\sum_{i=1}^{m_n}\rbraces{Y_i^2-\E{Y_i^2}}}^2}}^{1/2}
\leq   C \frac{\sqrt{b_n m_n}}{\sqrt{\ell_n}}.
\end{align*}
We required $s>0.5$ such that we must additionally choose $m_n=o\rbraces{\frac{\ell_n}{b_n}}=o\rbraces{n^{2s-1}}$ for the above expression to converge towards zero. For the second sum, the coupling argument yields
\begin{align*}
&\Pb\rbraces{\frac{ \sqrt{\ell_n}}{\sqrt{b_n}} \sum_{1\leq j\leq b_n} \abs{\tilde{S}_j-\tilde{S}'_j}>\varepsilon^*}\\
\leq &\Pb\rbraces{\tilde{B}_j\neq \tilde{B}_j' \text{ for at least one } j\in \{1, ..., b_n \}}
\leq  b_n \beta_Y(m_n)
\end{align*}
which tends towards zero as $n\rightarrow \infty$ by assumption.
}\end{proof}

\subsection{Ingredients for the Proofs of Theorems 2.4 and 2.5}
\label{Appendix: Subsec: Ingredients for the Proofs of Theorem: LLN Un and Theorem: CLT Un}
\begin{proof}[Proof of Proposition 3.5]
According to the central limit theorem for partial sums of $\beta$-mixing processes (see, e.g., Theorem 10.7 in \citep{Bradley.2007}), the random variable $\sqrt{\ell_n}\tilde{S}_j'/\kappa$ converges in distribution towards a standard normally distributed random variable,
\begin{equation*}
\frac{\sqrt{\ell_n}}{\kappa}\tilde{S}_j'=\frac{\sqrt{\ell_n-m_n}}{\sqrt{\ell_n}}\cdot \frac{1}{\kappa\sqrt{\ell_n-m_n}} \sum_{i=(j-1)\ell_n+1}^{j\ell_n-m_n}\rbraces{{Y_i'}^2-\E{{Y_i'}^2}}\distConv \NoD{0}{1}.
\end{equation*}
Moreover, the $\tilde{S}_j'(n)$'s are independent and identically distributed such that
\begin{equation*}
\sqrt{\ell_n} \cdot\frac{\tilde{S}_j'-\tilde{S}_k'}{\kappa} \distConv\NoD{0}{2}
\end{equation*}
for $j \neq k$. According to Theorem 5.4 in \citep{billingsley1968}, one additionally obtains convergence of the first absolute moments if the $\sqrt{\ell_n} \abs{\tilde{S}_j'-\tilde{S}_k'}$'s are uniformly integrable, but this holds due to
\begin{align*}
&\sup_{n\in \N} \E{\abs{\sqrt{\ell_n}\rbraces{\tilde{S}_j'-\tilde{S}_k'}}^{1+\varepsilon}}  \leq \max\rbraces{2,2^{1+\varepsilon}} \sup_{n\in \N} \E{\abs{\sqrt{\ell_n}\tilde{S}_1'}^{1+\varepsilon}}\\
\leq & 2^{1+\varepsilon} \sup_{n\in \N} \rbraces{\E{\abs{\sqrt{\ell_n}\tilde{S}_1'}^2}}^{\frac{1+\varepsilon}{2}} \leq C
\end{align*}
for $\varepsilon>0$ small enough, by the $c_r$- inequality, Jensen's inequality and since $\Var{\sqrt{\ell_n}\tilde{S}_1'}\rightarrow\kappa^2$. Hence,
\begin{align*}
&\E{\frac{\sqrt{\ell_n}}{\kappa}U_3(n)}=\frac{1}{b_n(b_n-1)}\sum_{1\leq j\neq k\leq b_n} \E{\abs{\sqrt{\ell_n}\cdot\frac{\tilde{S}'_j-\tilde{S}'_k}{\kappa}}}\\
=& \E{\abs{\sqrt{\ell_n}\cdot\frac{\tilde{S}'_1-\tilde{S}'_2}{\kappa}}}\longrightarrow\E{\abs{Z-Z'}}
\end{align*}
for two independent and standard normally distributed random variables $Z$ and $Z'$. To finish the proof, it remains to check that the variance of $\sqrt{\ell_n}U_3(n)$ converges towards zero. By the Cauchy-Schwarz inequality, it holds

\begin{align*}
&\Var{\sqrt{\ell_n}U_3(n)} \\
= &\frac{1}{b_n^2(b_n-1)^2}\sum_{1\leq j_1\neq k_1\leq b_n} \sum_{1\leq j_2\neq k_2\leq b_n} \cov{\abs{\sqrt{\ell_n}\rbraces{\tilde{S}'_{j_1}-\tilde{S}'_{k_1}}}}{\abs{\sqrt{\ell_n}\rbraces{\tilde{S}'_{j_2}-\tilde{S}'_{k_2}}}}\\
=& \frac{2}{b_n(b_n-1)}\Var{\abs{\sqrt{\ell_n}\rbraces{\tilde{S}'_1- \tilde{S}'_{2}}}}\\
  & + \frac{2(b_n-2)}{b_n(b_n-1)}\cov{\abs{\sqrt{\ell_n}\rbraces{\tilde{S}'_1-\tilde{S}'_{2}}}}{\abs{\sqrt{\ell_n}\rbraces{\tilde{S}'_1-\tilde{S}'_{3}}}}\\
\leq & \frac{2b_n-2}{b_n(b_n-1)} \Var{\abs{\sqrt{\ell_n}\rbraces{\tilde{S}'_1-\tilde{S}'_{2}}}}
= \frac{2b_n-2}{b_n(b_n-1)}2 \Var{\abs{\sqrt{\ell_n}\tilde{S}'_1}}
\leq C \frac{2b_n-2}{b_n(b_n-1)}
\end{align*}
which tends towards zero as $n\rightarrow\infty$.
\end{proof}

\begin{proposition}{
\label{Prop: Replacing kappan by kappa}

Assume that there exist constants $\rho>1$ and $0<\delta\leq 1$ such that $\E{\abs{Y_1}^{4+2\delta}}<\infty$ and for all $k\in \N$ it holds $\beta_Y(k)\leq C k^{-\rho (2+\delta)(1+\delta)/\delta^2}$. Moreover, assume $\ell_n=n^s$ with $s>0.5$ and $m_n=o(n^{2s-1})$. Then, under the null hypothesis,
\begin{equation*}
\sqrt{b_n}\abs{\kappa-\kappa_n} \rightarrow 0
\end{equation*}
as $n \rightarrow \infty$, where  $\kappa_n =\Var{\sqrt{\ell_n}\tilde{S}_1'}$.
}\end{proposition}

\begin{proof}[Proof of Proposition \ref{Prop: Replacing kappan by kappa}]
\allowdisplaybreaks{
It holds
\begin{align*}
& \sqrt{b_n} \abs{\kappa-\kappa_n}\\
 = & \sqrt{b_n}\left|\abs{\Var{Y_1^2}+2\sum_{k=1}^\infty \cov{Y_1^2}{Y_{k+1}^2}}^{1/2}\right. \\
& \left. -\abs{ \frac{\ell_n-m_n}{\ell_n} \Var{Y_1^2} +2\sum_{k=1}^{\ell_n-m_n-1} \frac{\ell_n-m_n-k}{\ell_n}\cov{Y_1^2}{Y_{k+1}^2} }^{1/2} \right|\\
 \leq & \sqrt{b_n}\left|\Var{Y_1^2}+2\sum_{k=1}^\infty \cov{Y_1^2}{Y_{k+1}^2} - \frac{\ell_n-m_n}{\ell_n} \Var{Y_1^2}\right. \\
 & \left.-2\sum_{k=1}^{\ell_n-m_n-1} \frac{\ell_n-m_n-k}{\ell_n}\cov{Y_1^2}{Y_{k+1}^2}\right|^{1/2}\\
\leq & \abs{ \frac{b_n m_n}{\ell_n}\Var{Y_1^2} }^{1/2}+ \abs{2 b_n\sum_{k=1}^{\ell_n-m_n-1} \frac{m_n+k}{\ell_n}\cov{Y_1^2}{Y_{k+1}^2} }^{1/2}\\
& + \abs{2 b_n \sum_{k=\ell_n-m_n}^\infty \cov{Y_1^2}{Y_{k+1}^2} }^{1/2}.
\end{align*}
We will treat each of these terms separately. The first expression obviously converges towards zero due to $m_n=o\rbraces{n^{2s-1}}$. For the second term, we apply Davydov's covariance inequality together with the polynomial decay of the mixing coefficients and obtain
\begin{align*}
&b_n\sum_{k=1}^{\ell_n-m_n-1} \frac{m_n+k}{\ell_n}\cov{Y_1^2}{Y_{k+1}^2}\leq 8 b_n \lVert Y_1^2\rVert_{2+\delta}^2 \sum_{k=1}^{\ell_n-m_n-1} \frac{m_n+k}{\ell_n} \beta_{Y^2}(k)^{\delta/(2+\delta)}\\
\leq & C \frac{b_n m_n}{\ell_n}\sum_{k=1}	^\infty k^{-\rho(1+\delta)/\delta} + C\frac{b_n}{\ell_n}\sum_{k=1}^\infty k^{1-\rho(1+\delta)/\delta}.
\end{align*}
Since both of the above sums are finite as long as $\rho>2\delta/(1+\delta)$ and we assumed $\rho>1$ together with $\delta\in (0,1]$, the second term converges towards zero as well. Similarly, the third term can be treated via
\begin{align*}
& b_n \sum_{k=\ell_n-m_n}^\infty \cov{Y_1^2}{Y_{k+1}^2} \leq b_n \sum_{k=\ell_n-m_n}^\infty \frac{k}{\ell_n-m_n} \cov{Y_1^2}{Y_{k+1}^2}\\
 \leq & C \frac{b_n}{\ell_n-m_n}\sum_{k=1}^\infty k^{1-\rho(1+\delta)/\delta}
\end{align*}
which tends likewise towards zero.
}
\end{proof}

\begin{proof}[Proof of Lemma 3.7]
{Define  the cumulative distribution functions
\begin{equation*}
F_n(x):=\Pb\rbraces{\frac{\sqrt{\ell_n}}{\sqrt{2}\kappa_n}\rbraces{\tilde{S}'_1-\tilde{S}'_{2}}\leq x}
\end{equation*}
and
\begin{equation*}
\Phi(x):=\Pb\rbraces{\frac{1}{\sqrt{2}}\rbraces{Z-Z'}\leq x},
\end{equation*}
where $Z$ and $Z'$ are independent standard normally distributed random variables. Denote their maximal difference by
\begin{equation*}
\Delta_n:=\underset{x}{\sup}\abs{F_n(x)-\Phi(x)}.
\end{equation*}
We will bound this difference from above with the help of Theorem 1 in Tikhomirov \citep{Tikhomirov.1980}.
To do so, note that
\begin{equation*}
\frac{\sqrt{\ell_n}}{\sqrt{2}\kappa_n}\rbraces{\tilde{S}'_1-\tilde{S}'_{2}}= \frac{1}{\sqrt{2}\kappa_n} \sum_{i=1}^{\ell_n-m_n}\frac{1}{\sqrt{\ell_n}}\rbraces{{Y_i'}^2-{Y'}^2_{\ell_n+i}}
\end{equation*}
 has mean zero and variance one. Moreover, the sequence $\rbraces{{Y_i'}^2-{Y'}^2_{\ell_n+i}}/\sqrt{\ell_n}$   for $i=1,..., \ell_n-m_n$ is strictly stationary, absolutely regular with mixing coefficients smaller than or equal to $2\beta_{Y^2}(k)\leq 2\beta_Y(k)\leq 2C k^{-\rho (2+\delta)(1+\delta)/\delta^2}$ (see, Theorem 1, Chapter 1.1, in Doukhan \citep{Doukhan.1994}), and has finite $(2+\delta)$-moments by assumption and by an application of the $c_r$-inequality.
Hence, there exists a constant $C$ such that
\begin{equation*}
\Delta_n\leq C \rbraces{\ell_n-m_n}^{-(\delta/2)(\rho-1)/(\rho+1)}.
\end{equation*}
In particular, $\Delta_n\rightarrow 0$ as $n \rightarrow \infty$ such that $ 0<\Delta_n<\frac{1}{\sqrt{e}}$ for $n$ large enough. Additionally, note that $ \int_{-\infty}^{\infty}\abs{x}^2 \mathrm{d} F_n=\Var{\sqrt{\ell_n}\rbraces{\tilde{S}'_1-\tilde{S}'_{2}}/\sqrt{2}\kappa_n}=1$. According to Theorem 9, Chapter V, in Petrov \citep{Petrov.1975}, there hence exists a constant $C>0$ such that
\begin{equation*}
\abs{F_n(x)-\Phi(x)}\leq \frac{C\Delta_n\log\rbraces{\frac{1}{\Delta_n}}}{1+\abs{x}^2}.
\end{equation*}
Combining both results yields
\begin{equation*}
\abs{F_n(x)-\Phi(x)}\leq \frac{C \rbraces{\ell_n-m_n}^{-(\delta/2)(\rho-1)/(\rho+1)}\log\rbraces{C(\ell_n-m_n)}}{1+\abs{x}^2}.
\end{equation*}
Next, we define the cumulative distribution functions of the absolute values as
\begin{equation*}
\tilde{F}_n(x):=\Pb\rbraces{\frac{\sqrt{\ell_n}}{\sqrt{2}\kappa_n}\abs{\tilde{S}'_1-\tilde{S}'_{2}}\leq x} \quad
\text{and} \quad \tilde{\Phi}(x):=\Pb\rbraces{\frac{1}{\sqrt{2}}\abs{Z-Z'}\leq x}.
\end{equation*}
Due to their symmetry, we obtain
\begin{equation*}
\abs{\tilde{F}_n(x)-\tilde{\Phi}(x)}\leq 2\frac{C \rbraces{\ell_n-m_n}^{-(\delta/2)(\rho-1)/(\rho+1)}\log\rbraces{C(\ell_n-m_n)}}{1+\abs{x}^2}.
\end{equation*}
This yields
\begin{align*}
&\sqrt{b_n}\abs{\theta^{(n)}-\theta} = \sqrt{b_n} \sqrt{2} \abs{\E{\abs{\sqrt{\ell_n}\cdot\frac{\tilde{S}_1'-\tilde{S}_2'}{\sqrt{2} \kappa_n}}}-\E{\abs{\frac{Z-Z'}{\sqrt{2}}}}}\\
=&\sqrt{b_n} \sqrt{2}\abs{\int_0^\infty \rbraces{\tilde{F}_n(x)-\tilde{\Phi}(x)}\mathrm{d}x}
 \leq \sqrt{b_n} \sqrt{2} \int_0^\infty \abs{\tilde{F}_n(x)-\tilde{\Phi}(x)}\mathrm{d}x\\
\leq & \sqrt{b_n} 2\sqrt{2}   C \rbraces{\ell_n-m_n}^{-(\delta/2)(\rho-1)/(\rho+1)}\log\rbraces{C(\ell_n-m_n)} \int_0^\infty  \frac{1}{1+\abs{x}^2}\mathrm{d}x\\
\leq & C \sqrt{b_n} \rbraces{\ell_n-m_n}^{-(\delta/2)(\rho-1)/(\rho+1)}\log\rbraces{C(\ell_n-m_n)},
\end{align*}
which converges towards zero due to  $s>\rbraces{1+\delta\frac{\rho-1}{\rho+1}}^{-1}$ and $m_n=o(\ell_n)$. 
}\end{proof}

\begin{proof}[Proof of Proposition 3.6] 
By Proposition \ref{Prop: Replacing kappan by kappa}, we can replace the long run variance $\kappa$ in the formulation of Proposition 3.6 
by the sample size dependent normalization $\kappa_n$.
Using the Hoeffding-decomposition, we can then rewrite the U-statistic as
\begin{align*}
&\sqrt{b_n}\rbraces{\frac{\sqrt{\ell_n}}{\kappa_n}U_3(n)-\theta^{(n)}}\\
=&\frac{2}{\sqrt{b_n}}\sum_{1\leq j \leq b_n} h_1^{(n)}\rbraces{\sqrt{\ell_n}\frac{\tilde{S}'_j}{\kappa_n}}+ \frac{1}{\sqrt{b_n}(b_n-1)}\sum_{1\leq j \neq k \leq b_n} h_2^{(n)} \rbraces{\sqrt{\ell_n}\frac{\tilde{S}'_j}{\kappa_n},\sqrt{\ell_n}\frac{\tilde{S}'_k}{\kappa_n}}.
\end{align*}
We will prove that the first part of this decomposition converges in distribution towards a normal random variable and that the latter converges towards zero in probability.

To show the convergence of the linear part, we will apply Lyapunov's central limit theorem. By construction, $\E{h_1^{(n)}\rbraces{\sqrt{\ell_n}\tilde{S}'_j/\kappa_n}}=0$. Moreover, define for $j=1, ..., b_n$

\begin{equation*}
Y_{n,j}:= \frac{ h_1^{(n)}\rbraces{\sqrt{\ell_n}\tilde{S}'_j/\kappa_n}}{\sqrt{b_n \Var{h_1^{(n)}\rbraces{\sqrt{\ell_n}\tilde{S}'_j/\kappa_n}}}}
\end{equation*}
such that $\E{Y_{n,j}}=0$ and $\sum_{j=1}^{b_n}\Var{Y_{n,j}}=1$. To apply Lyapunov's central limit theorem, we need to verify that for some $\eta>0$
\begin{equation*}
\limn \sum_{j=1}^{b_n} \E{ \abs{Y_{n,j}}^{2+\eta}}=0.
\end{equation*}
By stationarity, we have
\begin{align*}
&\limn \sum_{j=1}^{b_n} \E{ \abs{Y_{n,j}}^{2+\eta}}= \limn b_n \cdot \E{ \abs{Y_{n,1}}^{2+\eta}}\\
=& \limn b_n^{-\eta/2} \Var{h_1^{(n)}\rbraces{\sqrt{\ell_n}\tilde{S}'_1/\kappa_n}}^{-(1+\eta/2)} \E{\abs{h_1^{(n)}\rbraces{\sqrt{\ell_n}\tilde{S}'_1/\kappa_n}}^{2+\eta}},
\end{align*}
which converges towards zero if the last two terms are bounded. For the latter of the two, it holds
\begin{align*}
&\E{\abs{h_1^{(n)}\rbraces{\sqrt{\ell_n}\tilde{S}'_1/\kappa_n}}^{2+\eta}}= \E{\abs{\EQ{h\rbraces{\sqrt{\ell_n}\tilde{S}'_1/\kappa_n,\sqrt{\ell_n}\tilde{S}'_{2}/\kappa_n}}{2}-\theta^{(n)}}^{2+\eta}}\\
\leq & 2^{1+\eta} \rbraces{\E{\abs{\EQ{h\rbraces{\sqrt{\ell_n}\tilde{S}'_1/\kappa_n,\sqrt{\ell_n}\tilde{S}'_{2}/\kappa_n}}{2} }^{2+\eta}}+ \rbraces{\theta^{(n)}}^{2+\eta}}\\
\leq  & 2^{1+\eta} \rbraces{\E{\abs{h\rbraces{\sqrt{\ell_n}\tilde{S}'_1/\kappa_n,\sqrt{\ell_n}\tilde{S}'_{2}/\kappa_n}}^{2+\eta} }+ \rbraces{\theta^{(n)}}^{2+\eta}}
\end{align*}
by the $c_r$- and by Jensen's inequality. Since $\theta^{(n)} \rightarrow \theta$ by Lemma 3.7, 
 it remains to show that the expectation
\begin{align*}
&\E{\abs{h\rbraces{\sqrt{\ell_n}\tilde{S}'_1/\kappa_n,\sqrt{\ell_n}\tilde{S}'_{2}/\kappa_n}}^{2+\eta} } = \E{\abs{\sqrt{\ell_n}\rbraces{\tilde{S}'_1-\tilde{S}'_{2}}/\kappa_n}^{2+\eta} }\\
\leq & 2^{2+\eta}  \E{\abs{\sqrt{\ell_n}\tilde{S}'_1/\kappa_n}^{2+\eta}}
\end{align*}
is bounded.
To apply a Rosenthal-type inequality (see, Theorem 2, Section 1.4, in \citep{Doukhan.1994}) to the latter expectation, we need to verify that for some $\varepsilon>0$, it holds
$ \E{\abs{Y_1^2}^{2+\eta+\varepsilon}}<\infty$ as well as
$$ \sum_{k=1}^\infty (k+1)^2 \beta(k)^{\frac{\varepsilon}{4+\varepsilon}}\leq C \sum_{k=1}^\infty (k+1)^2 k^{-\rho \frac{(2+\delta)(1+\delta)}{\delta^2}\frac{\varepsilon}{4+\varepsilon}}<\infty.$$
The moment condition is easily met since we may choose $\eta$ and $\varepsilon$ via $\eta+\varepsilon=\delta$. The second condition translates to
$$\rho \frac{(2+\delta)(1+\delta)}{\delta^2}\frac{\varepsilon}{4+\varepsilon}>3, \quad \text{or, equivalently,} \quad \varepsilon>12\rbraces{\rho \frac{(2+\delta)(1+\delta)}{\delta^2}-3}^{-1}$$
To find an admissible choice for $\varepsilon$ and $\eta$, we must choose it in such a way that $\varepsilon<\delta$. The existence of such a choice is ensured if we have $\delta>12\rbraces{\rho \frac{(2+\delta)(1+\delta)}{\delta^2}-3}^{-1}$ which is fulfilled for $\rho>\frac{9\delta}{(2+\delta)(1+\delta)}.$
An application of the Rosenthal-type inequality now yields
\begin{align*}
&\E{\abs{\sqrt{\ell_n}\frac{\tilde{S}_1'}{\kappa_n}}^{(2+\eta)}}
=  \E{\abs{\frac{1}{\kappa_n\sqrt{\ell_n} } \sum_{i=1}^{\ell_n-m_n}\rbraces{{Y_i'}^2-\E{{Y_i'}^2}}}^{(2+\eta)}}\\
\leq & C \kappa_n^{-(2+\eta)} \ell_n^{-(2+\eta)/2} \max\left\lbrace \sum_{i=1}^{\ell_n-m_n}\rbraces{\E{\abs{{Y_i'}^2-\E{{Y_i'}^2}}^{(2+\eta)+\varepsilon}}}^{\frac{(2+\eta)}{(2+\eta)+\varepsilon}}   , \right.\\
& \left. \rbraces{ \sum_{i=1}^{\ell_n-m_n}\E{\abs{{Y_i'}^2-\E{{Y_i'}^2}}^{2+\varepsilon}}^{\frac{2}{2+\varepsilon}}}^{(2+\eta)/2}   \right\rbrace\\
\leq & C \kappa_n^{- (2+\eta)} \ell_n^{-(2+\eta)/2} \max\left\lbrace \ell_n-m_n, \rbraces{\ell_n-m_n}^{(2+\eta)/2} \right\rbrace \leq C.
\end{align*}
We now turn towards the variance $\Var{h_1^{(n)}\rbraces{\sqrt{\ell_n}\tilde{S}'_1/\kappa_n}}$, where
$$h_1^{(n)}\rbraces{\sqrt{\ell_n}\tilde{S}_1'/\kappa_n}= \EQ{h\rbraces{\sqrt{\ell_n}\tilde{S}_1'/\kappa_n, \sqrt{\ell_n}\tilde{S}_2'/\kappa_n}}{2}- \theta^{(n)}.$$
According to the central limit theorem under $\beta$-mixing, it holds
\begin{equation*}
\sqrt{\ell_n}\cdot\frac{\tilde{S}_1'}{\kappa_n} \distConv Z \quad \text{and} \quad \sqrt{\ell_n}\cdot \frac{\tilde{S}_2'}{\kappa_n} \distConv Z',
\end{equation*}
where $Z$ and $Z'$ are two independent standard normally distributed random variables. We will apply the  continuous mapping theorem for sequences of functions (see, Theorem 5.5 in \citep{billingsley1968}) to show that $h_1^{(n)}\rbraces{\sqrt{\ell_n}\tilde{S}_1'/\kappa_n}$ converges towards $h_1(Z)$ in distribution, where $h_1(x)= \E{h(x,Z')}-\theta$. First, we have to check that $\Pb(Z\in E)=0$ for the set
\begin{equation*}
 E:=\{ x\in \R: \exists (x_n)_{n\in \N}\subset \R \text{ with } x_n \rightarrow x   \text{ but } h_1^{(n)}(x_n) \not\rightarrow h_1(x)\}.
\end{equation*}
Actually, $E$ is even empty, which can be seen as follows: Take an arbitrary $x\in \R$ and a sequence $(x_n)_{n\in \N}$ such that $x_n \rightarrow x$. Lemma 3.7 
yields $\theta^{(n)}\rightarrow \theta$ as $n \rightarrow \infty$ and we additionally need to verify that
$\E{h\rbraces{x_n, \sqrt{\ell_n}\tilde{S}_2'/\kappa_n}}$ converges towards $\E{h(x,Z')}$. By the continuous mapping theorem,
$$h\rbraces{x_n,\sqrt{\ell_n}\tilde{S}_2'/\kappa_n}= \abs{x_n- \sqrt{\ell_n}\tilde{S}_2'/\kappa_n} \distConv \abs{x-Z'}=h(x,Z').$$

Moreover, the $\abs{x_n- \sqrt{\ell_n}\tilde{S}_2'/\kappa_n}$'s are uniformly integrable  due to
\begin{align*}
&\sup_n \E{\abs{x_n- \sqrt{\ell_n}\frac{\tilde{S}_2'}{\kappa_n}}^{1+\varepsilon}} \leq  2^\varepsilon
\sup_n \rbraces{\abs{x_n}^{1+\varepsilon}+\E{\abs{\sqrt{\ell_n}\frac{\tilde{S}_2'}{\kappa_n}}^{1+\varepsilon}}}\\
& \leq 2^\varepsilon \sup_n \abs{x_n}^{1+\varepsilon}+ 2^\varepsilon\sup_n\E{\abs{\sqrt{\ell_n}\frac{\tilde{S}_2'}{\kappa_n}}^2}^{(1+\varepsilon)/2} \leq C
\end{align*}
for some $\varepsilon>0$ small enough. Theorem 5.4 in \citep{billingsley1968} thus implies
$$\E{h\rbraces{x_n,\sqrt{\ell_n}\tilde{S}_2'/\kappa_n}}\longrightarrow\E{h(x,Z')}$$ and consequently, $h_1^{(n)}(x_n) \rightarrow h_1(x)$. Hence, $x\notin E$ and
$h_1^{(n)}\rbraces{\sqrt{\ell_n}\tilde{S}_1'/\kappa_n} \distConv h_1(Z)$
follows by Theorem 5.5 in \citep{billingsley1968}.

 To show the convergence of the variances
\begin{equation*}
\Var{h_1^{(n)}\rbraces{\sqrt{\ell_n}\frac{\tilde{S}_1'}{\kappa_n}}}= \E{\rbraces{h_1^{(n)}\rbraces{\sqrt{\ell_n}\frac{\tilde{S}_1'}{\kappa_n}}}^2} \rightarrow \E{h_1(Z)^2}= \Var{h_1(Z)},
\end{equation*}
we  once more apply Theorem 5.4 in \citep{billingsley1968}. To do so, it only remains to ensure that the $h_1^{(n)}\rbraces{\sqrt{\ell_n}\tilde{S}_1'/\kappa_n}$'s have uniformly integrable second moments but
\begin{equation*}
\sup_n \E{\abs{h_1^{(n)}\rbraces{\sqrt{\ell_n}\frac{\tilde{S}_1'}{\kappa_n}}}^{2(1+\varepsilon)}} \leq C
\end{equation*}
for some $\varepsilon>0$ has already been shown at the beginning of the proof. Consequently, the variance converges and
\begin{equation*}
\limn \sum_{j=1}^{b_n} \E{ \abs{Y_{n,j}}^{2+\eta}}= 0.
\end{equation*}
By Lyapunov's central limit theorem, it therefore holds
\begin{equation*}
\sum_{j=1}^{b_n} \frac{h_1^{(n)}\rbraces{\sqrt{\ell_n}\tilde{S}_j'/\kappa_n}}{\sqrt{b_n\Var{h_1^{(n)}\rbraces{\sqrt{\ell_n}\tilde{S}_j'/\kappa_n}}}} \distConv \NoD{0}{1}.
\end{equation*}
Taking into account the convergence of the variances shown above,  this leads to
 \begin{align*}
 \frac{2}{\sqrt{b_n}} \sum_{1\leq j \leq b_n} h_1^{(n)}\rbraces{\sqrt{\ell_n}\tilde{S}_j'/\kappa_n} \distConv \NoD{0}{\psi^2}
 \end{align*}
with $\psi^2=4\Var{h_1(Z))}$.

We now check that the degenerate part of the Hoeffding-decomposition converges towards zero in probability. Note that, due to the independence of the $\tilde{S}_j'$'s, the expectation of the degenerate part is zero and its variance is given by
\begin{align*}
&\frac{1}{b_n (b_n-1)^2}\sum_{1\leq j \neq k \leq b_n} \Var{h_2^{(n)} \rbraces{\sqrt{\ell_n}\tilde{S}_j'/\kappa_n, \sqrt{\ell_n} \tilde{S}_k'/\kappa_n}}\\
= & \frac{1}{(b_n-1) }\Var{h_2^{(n)} \rbraces{\sqrt{\ell_n}\tilde{S}_1'/\kappa_n, \sqrt{\ell_n}\tilde{S}_2'/\kappa_n}}.
\end{align*}
since the terms are mutually uncorrelated by construction. To prove the convergence of the latter expression towards zero, it suffices that the variance is bounded, which holds due to
{\allowdisplaybreaks
\begin{align*}
&\Var{ h_2^{(n)} \rbraces{\sqrt{\ell_n}\tilde{S}_1'/\kappa_n, \sqrt{\ell_n}\tilde{S}_2'/\kappa_n}}
=  \E{h_2^{(n)} \rbraces{\sqrt{\ell_n}\tilde{S}_1'/\kappa_n, \sqrt{\ell_n}\tilde{S}_2'/\kappa_n}^2}\\
= & \E{ \rbraces{h(\sqrt{\ell_n}\tilde{S}_1'/\kappa_n,\sqrt{\ell_n}\tilde{S}_2'/\kappa_n)-\theta^{(n)}-h_1^{(n)}(\sqrt{\ell_n}\tilde{S}_1'/\kappa_n)-h_1^{(n)}(\sqrt{\ell_n}\tilde{S}_2'/\kappa_n) }^2}\\
= & \E{h\rbraces{\sqrt{\ell_n}\tilde{S}_1'/\kappa_n,\sqrt{\ell_n} \tilde{S}_2'/\kappa_n}^2}- \rbraces{\theta^{(n)}}^2\\
& - 4\E{h\rbraces{\sqrt{\ell_n}\tilde{S}_1'/\kappa_n, \sqrt{\ell_n}\tilde{S}_2'/\kappa_n}h_1^{(n)}\rbraces{\sqrt{\ell_n}\tilde{S}_1'/\kappa_n}}+ 2 \E{h_1^{(n)}\rbraces{\sqrt{\ell_n}\tilde{S}_1'/\kappa_n}^2}\\
= &  \E{h\rbraces{\sqrt{\ell_n}\tilde{S}_1'/\kappa_n, \sqrt{\ell_n}\tilde{S}_2'/\kappa_n}^2}- \rbraces{\theta^{(n)}}^2- 2\E{h_1^{(n)}\rbraces{\sqrt{\ell_n}\tilde{S}_1'/\kappa_n}^2}\\
\leq & \E{h\rbraces{\sqrt{\ell_n}\tilde{S}_1'/\kappa_n, \sqrt{\ell_n}\tilde{S}_2'/\kappa_n}^2}
=  \E{\abs{\sqrt{\ell_n}\rbraces{\tilde{S}_1' -\tilde{S}_2'}/\kappa_n}^{2}}\\
\leq & 4 \E{\rbraces{\sqrt{\ell_n}\tilde{S}_2'/\kappa_n}^2}=4
\end{align*}}
by an application of the $c_r$-inequality. Thus, the degenerate part converges towards zero in probability and together with the asymptotic normality of the linear part,  Slutzky's Lemma proves the Proposition.
\end{proof}

\subsection{Estimation of the Long Run Variance}
\subsubsection{An Approximation of the Estimator $\hat{\kappa}^2$}
We will first prove that instead of the estimator $\hat{\kappa}^2$, based on the block-mean centred observations $\tilde{X}_i$, we can employ
$$\hat{B}:= \frac{1}{\tilde{b}_n}\sqrt{\frac{\pi}{2}} \sum_{j=1}^{\tilde{b}_n} \frac{ \abs{ \sum_{i=(j-1)\tilde{\ell}_n+1}^{j\tilde{\ell}_n} \rbraces{\sigma_i^2Y_i^2 - \frac{1}{n}\sum_{k=1}^n \sigma_k^2Y_k^2}}} {\sqrt{\tilde{\ell}_n}},$$
which is based on the centred observations $X_i-\mu_i=\sigma_iY_i$, since both are asymptotically equivalent.

\begin{proposition}{
\label{Appendix:Prop: LRV Mean approx H and A}
Let the assumptions (A1) and (A2) hold and assume $\ell_n=n^s$ and $\elln=n^q$ with $q<s$, $q<3(1-s)$ and $s>0.5$.  Then it holds
 \begin{equation*}
\sqrt{b_n}\abs{\hat{\kappa}\cdot\hat{\sigma}_H^2- \hat{B}}\pConv 0 \quad \text{as } n \rightarrow \infty.
 \end{equation*}
}\end{proposition}

\begin{proof}
It holds
\allowdisplaybreaks{
\begin{align*}
 &\sqrt{\frac{2}{\pi}} \sqrt{b_n}\abs{\hat{\kappa}\cdot\hat{\sigma}_H^2- \hat{B}}\\
 =& \sqrt{b_n} \abs{ \frac{1}{\tilde{b}_n}  \sum_{j=1}^{\tilde{b}_n} \rbraces{\abs{\sum_{i=(j-1)\tilde{\ell}_n+1}^{j\tilde{\ell}_n} \sigma_i^2Y_i^2 - \frac{\tilde{\ell}_n}{n}\sum_{i=1}^n \sigma_i^2Y_i^2 }-\abs{\sum_{i=(j-1)\tilde{\ell}_n+1}^{j\tilde{\ell}_n} \tilde{X}_i^2 - \frac{\tilde{\ell}_n}{n}\sum_{i=1}^n \tilde{X}_i^2 } } \frac{1}{\sqrt{\tilde{\ell}_n}} }\\
 \leq &   \frac{\sqrt{b_n}}{\tilde{b}_n} \sum_{j=1}^{\tilde{b}_n} \abs{\sum_{i=(j-1)\tilde{\ell}_n+1}^{j\tilde{\ell}_n} \sigma_i^2Y_i^2 - \frac{\tilde{\ell}_n}{n}\sum_{i=1}^n \sigma_i^2Y_i^2 -\rbraces{\sum_{i=(j-1)\tilde{\ell}_n+1}^{j\tilde{\ell}_n} \tilde{X}_i^2 - \frac{\tilde{\ell}_n}{n}\sum_{i=1}^n \tilde{X}_i^2 } } \frac{1}{\sqrt{\tilde{\ell}_n}} \\
 \leq &   \frac{\sqrt{b_n}}{\tilde{b}_n} \sum_{j=1}^{\tilde{b}_n} \abs{\sum_{i=(j-1)\tilde{\ell}_n+1}^{j\tilde{\ell}_n} \rbraces{\sigma_i^2Y_i^2 -\tilde{X}_i^2}} \frac{1}{\sqrt{\tilde{\ell}_n}}+ \sqrt{b_n}  \abs{ \frac{\tilde{\ell}_n}{n}\sum_{i=1}^n \rbraces{\sigma_i^2Y_i^2 -\tilde{X}_i^2}} \frac{1}{\sqrt{\tilde{\ell}_n}} \\
 \leq & 2  \frac{\sqrt{b_n}}{\tilde{b}_n} \sum_{j=1}^{\tilde{b}_n} \abs{\sum_{i=(j-1)\tilde{\ell}_n+1}^{j\tilde{\ell}_n} \rbraces{\sigma_i^2Y_i^2 -\tilde{X}_i^2}} \frac{1}{\sqrt{\tilde{\ell}_n}}.
\end{align*}
Inserting the definition of $\tilde{X}_i^2$, we can bound the above expression by
\begin{align*}
&2  \frac{\sqrt{b_n}}{\tilde{b}_n} \sum_{j=1}^{\tilde{b}_n} \abs{\sum_{i=(j-1)\tilde{\ell}_n+1}^{j\tilde{\ell}_n} \rbraces{\mu_i-\frac{1}{\ell_n}\sum_{r(i)}\mu_r}^2} \frac{1}{\sqrt{\tilde{\ell}_n}}\\
+ & 2  \frac{\sqrt{b_n}}{\tilde{b}_n} \sum_{j=1}^{\tilde{b}_n} \abs{\sum_{i=(j-1)\tilde{\ell}_n+1}^{j\tilde{\ell}_n} \rbraces{2\sigma_iY_i\rbraces{\frac{1}{\ell_n}\sum_{r(i)}\mu_r-\mu_i}}} \frac{1}{\sqrt{\tilde{\ell}_n}}\\
+ & 2  \frac{\sqrt{b_n}}{\tilde{b}_n} \sum_{j=1}^{\tilde{b}_n} \abs{\sum_{i=(j-1)\tilde{\ell}_n+1}^{j\tilde{\ell}_n} 2\rbraces{\frac{1}{\ell_n}\sum_{r(i)}\sigma_rY_r}\rbraces{\mu_i-\frac{1}{\ell_n}\sum_{r(i)}\mu_r}} \frac{1}{\sqrt{\tilde{\ell}_n}}\\
+& 2  \frac{\sqrt{b_n}}{\tilde{b}_n} \sum_{j=1}^{\tilde{b}_n} \abs{\sum_{i=(j-1)\tilde{\ell}_n+1}^{j\tilde{\ell}_n} \rbraces{2\sigma_iY_i \frac{1}{\ell_n}\sum_{r(i)}\sigma_rY_r}} \frac{1}{\sqrt{\tilde{\ell}_n}}\\
+ & 2  \frac{\sqrt{b_n}}{\tilde{b}_n} \sum_{j=1}^{\tilde{b}_n} \abs{\sum_{i=(j-1)\tilde{\ell}_n+1}^{j\tilde{\ell}_n} \rbraces{\frac{1}{\ell_n}\sum_{r(i)}\sigma_rY_r}^2} \frac{1}{\sqrt{\tilde{\ell}_n}},
\end{align*}
where $r(i)$ denotes the summation over all $r\in \{(k-1)\ell_n+1, ..., k \ell_n\}$  in case  $i$ itself lies in $\{(k-1)\ell_n+1, ..., k \ell_n\}$. Once more, we will show that all five of these terms converge towards zero. The first of these terms is deterministic and can be bounded by
\begin{align*}
2  \frac{\sqrt{b_n}}{\tilde{b}_n} \sum_{j=1}^{\tilde{b}_n} \abs{\sum_{i=(j-1)\tilde{\ell}_n+1}^{j\tilde{\ell}_n} \rbraces{\frac{1}{\ell_n}\sum_{r(i)}\abs{\mu_i-\mu_r}}^2} \frac{1}{\sqrt{\tilde{\ell}_n}}\leq  2 \sqrt{\frac{\elln}{b_n^3}} \longrightarrow 0.
\end{align*}
To show convergence of the remaining terms, it suffices that their expectations converge towards zero. For the second term, it holds by stationarity
\begin{align*}
& 2  \frac{\sqrt{b_n}}{\tilde{b}_n} \sum_{j=1}^{\tilde{b}_n} \E{ \abs{\sum_{i=(j-1)\tilde{\ell}_n+1}^{j\tilde{\ell}_n} 2\sigma_iY_i\rbraces{\frac{1}{\ell_n}\sum_{r(i)}\mu_r-\mu_i}}} \frac{1}{\sqrt{\tilde{\ell}_n}}\\
\leq & 2  \frac{\sqrt{b_n}}{\tilde{b}_n} \sum_{j=1}^{\tilde{b}_n} \rbraces{ \sum_{i=(j-1)\tilde{\ell}_n+1}^{j\tilde{\ell}_n} \sum_{k=(j-1)\tilde{\ell}_n+1}^{j\tilde{\ell}_n}4\sigma_i\sigma_k\E{Y_iY_k}\rbraces{\frac{1}{\ell_n}\sum_{r(i)}\mu_r-\mu_i}\rbraces{\frac{1}{\ell_n}\sum_{r(k)}\mu_r-\mu_k}}^{1/2} \frac{1}{\sqrt{\tilde{\ell}_n}}\\
\leq & 4 \sigma_{\sup} \frac{1}{\sqrt{b_n}} \rbraces{\frac{1}{\elln} \sum_{i=1}^{\elln}\sum_{k=1}^{\elln} \abs{\E{Y_iY_k}} }^{1/2} \longrightarrow 0.
\end{align*}
Similarly, the expected value of the third term can be bounded by
\begin{align*}
& 4 \frac{\sqrt{b_n}}{\tilde{b}_n} \sum_{j=1}^{\tilde{b}_n} \sum_{i=(j-1)\tilde{\ell}_n+1}^{j\tilde{\ell}_n} \E{\abs{\frac{1}{\ell_n}\sum_{r(i)}\sigma_rY_r}}\cdot\abs{\mu_i-\frac{1}{\ell_n}\sum_{r(i)}\mu_r} \frac{1}{\sqrt{\tilde{\ell}_n}}\\
\leq & 4\sigma_{\sup} \frac{1}{\sqrt{b_n}}\sqrt{\frac{\elln}{\ell_n}} \rbraces{\frac{1}{\ell_n}\sum_{i=1}^{\ell_n}\sum_{k=1}^{\ell_n}\abs{\E{Y_iY_k}}}^{1/2} \longrightarrow 0.
\end{align*}
To treat the fourth term, recall that $\elln=o(\ell_n)$ and we can thus assume without loss of generality that $\elln<\ell_n$. Hence, each block $\{(j-1)\elln+1, ..., j \elln\}$ overlaps with at most two blocks $\{(k-1)\ell_n+1, ..., k\ell_n\}$ and $\{k\ell_n+1, ..., (k+1)\ell_n\}$ and is thus split into two parts, $\{(j-1)\elln+1, ...,\tau_{j,n}\}$  and $\{\tau_{j,n}+1, ..., j \elln\}$. Consequently, we can also split the inner sum in the fourth term up via
\begin{align*}
&2  \frac{\sqrt{b_n}}{\tilde{b}_n} \sum_{j=1}^{\tilde{b}_n} \E{\abs{\sum_{i=(j-1)\tilde{\ell}_n+1}^{j\tilde{\ell}_n} \rbraces{2\sigma_iY_i \frac{1}{\ell_n}\sum_{r(i)}\sigma_rY_r}}} \frac{1}{\sqrt{\tilde{\ell}_n}}\\
\leq &  2  \frac{\sqrt{b_n}}{\tilde{b}_n} \sum_{j=1}^{\tilde{b}_n} \E{\abs{\rbraces{\sum_{i=(j-1)\tilde{\ell}_n+1}^{\tau_{j,n}}2\sigma_iY_i} \rbraces{ \frac{1}{\ell_n}\sum_{r((j-1)\tilde{\ell}_n+1)}\sigma_rY_r}}} \frac{1}{\sqrt{\tilde{\ell}_n}}\\
& + 2  \frac{\sqrt{b_n}}{\tilde{b}_n} \sum_{j=1}^{\tilde{b}_n} \E{\abs{\rbraces{\sum_{i=\tau_{j,n}+1}^{j\elln}2\sigma_iY_i} \rbraces{ \frac{1}{\ell_n}\sum_{r(j\tilde{\ell}_n)}\sigma_rY_r}}} \frac{1}{\sqrt{\tilde{\ell}_n}}\\
\end{align*}
Both these terms converge towards zero since, e.g. for the first one,
\begin{align*}
 & 2  \frac{\sqrt{b_n}}{\tilde{b}_n} \sum_{j=1}^{\tilde{b}_n} \E{\abs{\rbraces{\sum_{i=(j-1)\tilde{\ell}_n+1}^{\tau_{j,n}}2\sigma_iY_i} \rbraces{ \frac{1}{\ell_n}\sum_{r((j-1)\tilde{\ell}_n+1)}\sigma_rY_r}}} \frac{1}{\sqrt{\tilde{\ell}_n}}\\
  \leq & 2  \frac{\sqrt{b_n}}{\tilde{b}_n} \sum_{j=1}^{\tilde{b}_n} \Bigg\{\E{\rbraces{\sum_{i=(j-1)\tilde{\ell}_n+1}^{\tau_{j,n}}2\sigma_iY_i}^2}\Bigg\}^{1/2} \Bigg\{\E{\rbraces{ \frac{1}{\ell_n}\sum_{r((j-1)\tilde{\ell}_n+1)}\sigma_rY_r}^2}\Bigg\}^{1/2} \frac{1}{\sqrt{\tilde{\ell}_n}}\\
  \leq & 4 \sigma_{\sup}^2 \sqrt{\frac{b_n}{\ell_n}} \rbraces{\frac{1}{\elln}\sum_{i=1}^{\elln}\sum_{k=1}^{\elln}\abs{\E{Y_iY_k}}}^{1/2} \rbraces{\frac{1}{\ell_n}\sum_{i=1}^{\ell_n}\sum_{k=1}^{\ell_n}\abs{\E{Y_iY_k}}}^{1/2} \longrightarrow 0
\end{align*}
by the Cauchy-Schwarz inequality. For the fifth term, one obtains
\begin{align*}
2  \frac{\sqrt{b_n}}{\tilde{b}_n} \sum_{i=1}^{n} \E{ \rbraces{\frac{1}{\ell_n}\sum_{r(i)}\sigma_rY_r}^2} \frac{1}{\sqrt{\tilde{\ell}_n}}
\leq 2\sigma_{\sup}^2 \frac{\sqrt{b_n}\sqrt{\elln}}{\ell_n}\rbraces{\frac{1}{\ell_n}\sum_{i=1}^{\ell_n}\sum_{k=1}^{\ell_n}\abs{\E{Y_iY_k}}} \longrightarrow 0
\end{align*}
and hence, convergence in probability towards zero of the entire expression holds.
}
\end{proof}
Consequently, it suffices from now on to prove our results for $\hat{B}$.

\subsubsection{Limit Results under the Null Hypothesis}
Under the null hypothesis, $\sigma= \sigma_H$, and we have
$$\hat{B}= \frac{\sigma_H^2}{\tilde{b}_n} \sqrt{\frac{\pi}{2}} \sum_{j=1}^{\tilde{b}_n} \abs{T_j- \elln \overline{Y^2} } \frac{1}{\sqrt{\tilde{\ell}_n}},$$
where we defined for ease of notation
$$ T_j:= T_{j,n}:=  \sum_{i=(j-1)\elln+1}^{j\elln} Y_i^2 \quad \text{and} \quad \overline{Y^2}:=\overline{Y^2_n}:=\frac{1}{n}\sum_{i=1}^nY_i^2.$$

\begin{proposition}{
\label{Appendix: Prop: LRV esti under H}
Assume that there exist constants $\rho>1$ and $0<\delta\leq 1$ such that $\E{\abs{Y_1}^{4+2\delta}}<\infty$ and for all $k\in \N$ it holds $\beta_Y(k)\leq C k^{-\rho (2+\delta)(1+\delta)/\delta^2}$. Moreover, let $\ell_n=n^s$ and $\elln=n^q$ such that $1-s<q\delta (\rho-1)/(\rho+1)$ and  $q<s$. 
Then, under the null hypothesis,
$$\sqrt{b_n}\abs{\hat{B}-\sigma_H^2\kappa}\pConv 0  \quad \text{ as } n \rightarrow \infty.$$
}\end{proposition}

\begin{proof}
Parts of this proof rely on the proof of Proposition 3.1 in Dehling et al. \citep{Dehling.2013} but various alterations need to be made to account for the additional scaling $\sqrt{b_n}$, such that the entire proof is written out here. Consider
{\allowdisplaybreaks
\begin{align*}
\sqrt{b_n}\abs{\hat{B}/\sigma_H^2-\kappa}
=& \sqrt{b_n}\abs{ \frac{1}{\tilde{b}_n} \sqrt{\frac{\pi}{2}} \sum_{j=1}^{\tilde{b}_n} \abs{T_j- \elln \overline{Y^2} } \frac{1}{\sqrt{\tilde{\ell}_n}}-\kappa}\\
\leq & \sqrt{b_n}\abs{ \frac{1}{\tilde{b}_n} \sqrt{\frac{\pi}{2}} \sum_{j=1}^{\tilde{b}_n} \rbraces{\abs{T_j- \elln \overline{Y^2} } -\abs{T_j- \elln \E{Y^2_1} }}\frac{1}{\sqrt{\tilde{\ell}_n}}}\\
&+  \sqrt{b_n}\abs{ \frac{1}{\tilde{b}_n} \sqrt{\frac{\pi}{2}} \sum_{j=1}^{\tilde{b}_n} \abs{T_j- \elln \E{Y^2_1} } \frac{1}{\sqrt{\tilde{\ell}_n}}-\kappa}.
\end{align*}}

We will show that each of the last two terms converges towards zero in probability. For the first one, we obtain
{\allowdisplaybreaks
\begin{align*}
& \Pb\rbraces{\sqrt{b_n}\abs{ \frac{1}{\tilde{b}_n} \sqrt{\frac{\pi}{2}} \sum_{j=1}^{\tilde{b}_n} \rbraces{\abs{T_j- \elln \overline{Y^2} } -\abs{T_j- \elln \E{Y^2_1} }}\frac{1}{\sqrt{\tilde{\ell}_n}}}>\varepsilon^*}\\
\leq & \Pb \rbraces{\sqrt{b_n}\abs{ \frac{1}{\tilde{b}_n} \sqrt{\frac{\pi}{2}} \sum_{j=1}^{\tilde{b}_n} \rbraces{\abs{T_j- \elln \overline{Y^2}  -T_j+ \elln \E{Y^2_1} }}\frac{1}{\sqrt{\tilde{\ell}_n}}}>\varepsilon^*}\\
= & \Pb \rbraces{ \sqrt{b_n}\sqrt{\elln}\sqrt{\frac{\pi}{2}}\frac{1}{n} \abs{\sum_{i=1}^n\rbraces{Y_i^2-\E{Y_1^2}}}>\varepsilon^*}\\
\leq &\frac{1}{\varepsilon^*}\frac{\pi}{2} \frac{b_n\elln}{n} \E{\abs{\frac{1}{\sqrt{n}}\sum_{i=1}^n\rbraces{Y_i^2-\E{Y_1^2}}}^2}
\end{align*}}
which tends towards zero as long as $q<s$ since the expectation converges towards the long run variance $\kappa^2$ and is thus bounded. We now turn towards the second term and bound it by
\begin{align*}
& \sqrt{b_n}\abs{ \frac{1}{\tilde{b}_n} \sqrt{\frac{\pi}{2}} \sum_{j=1}^{\tilde{b}_n} \abs{T_j- \elln \E{Y^2_1} } \frac{1}{\sqrt{\tilde{\ell}_n}}-\kappa}\\
\leq &  \sqrt{b_n}\abs{ \frac{1}{\tilde{b}_n} \sqrt{\frac{\pi}{2}} \sum_{j=1}^{\tilde{b}_n} \rbraces{\abs{T_j- \elln \E{Y^2_1} }- \E{\abs{T_j- \elln \E{Y^2_1} }} }\frac{1}{\sqrt{\tilde{\ell}_n}}} \\
& +  \sqrt{b_n}\abs{ \frac{1}{\tilde{b}_n} \sqrt{\frac{\pi}{2}} \sum_{j=1}^{\tilde{b}_n} \E{\abs{T_j- \elln \E{Y^2_1} }} \frac{1}{\sqrt{\tilde{\ell}_n}}-\kappa}.
\end{align*}
Once more, we treat each of these two terms separately.
For the first, it holds
\begin{align*}
&\Pb \rbraces{\sqrt{b_n}\abs{ \frac{1}{\tilde{b}_n} \sqrt{\frac{\pi}{2}} \sum_{j=1}^{\tilde{b}_n} \rbraces{\abs{T_j- \elln \E{Y^2_1} }- \E{\abs{T_j- \elln \E{Y^2_1} }} }\frac{1}{\sqrt{\tilde{\ell}_n}}}>\varepsilon^* } \\
\leq & \frac{1}{{\varepsilon^*}^2} \sqrt{\frac{\pi}{2}} \frac{b_n}{\elln\bn^2} \E{\abs{\sum_{j=1}^{\tilde{b}_n} \abs{T_j- \elln \E{Y^2_1} }- \E{\abs{T_j- \elln \E{Y^2_1} } }}^2}.
\end{align*}
Next, we intend to apply an inequality from Yokoyama \citep{Yokoyama.1980}. To do so, note that for fixed $n$, the sequence $T_j$,  $1\leq j\leq \bn$, is stationary and likewise $\beta$-mixing with mixing coefficients smaller than or equal to $\beta_{Y}$. Moreover, for  $0<\tilde{\delta}<\delta$, it holds by an application of Yokoyama's inequality for the $Y_i^2$'s that

\begin{align*}
&\E{\abs{T_1-\elln \E{\abs{Y_1}^2}}^{2+\tilde{\delta}}}
=\E{\abs{\sum_{i=1}^\elln \rbraces{Y_i^2-\E{Y_1^2}}}^{2+\tilde{\delta}}}\\
 \leq & C \elln^{(2+\tilde{\delta})/2}\rbraces{\E{\abs{Y_1^2}^{2+\delta}}}^{(2+\tilde{\delta})/(2+\delta)}
\end{align*}

which is finite for fixed $n$. Thus, we can apply  Yokoyama's inequality to the $T_j$'s, leading to 
\begin{align*}
&  \frac{b_n}{\elln\bn^2} \E{\abs{\sum_{j=1}^{\tilde{b}_n} \abs{T_j- \elln \E{Y^2_1} }- \E{\abs{T_j- \elln \E{Y^2_1} } }}^2} \\
\leq & C  \frac{b_n}{\elln\bn^2} \bn \E{\abs{T_1- \elln \E{Y^2_1} }^{2+\tilde{\delta}}}^{2/(2+\tilde{\delta})}\\
\leq & C  \frac{b_n}{\elln\bn}  \rbraces{\elln^{(2+\tilde{\delta})/2}\rbraces{ \E{\abs{Y_1^2}^{2+\delta}}} ^{(2+\tilde{\delta})/(2+\delta)}}^{2/(2+\tilde{\delta})}=C\frac{b_n}{\bn}
\end{align*}
which converges towards zero due to $1-q>1-s$. 
For the second term, we obtain
\begin{align*}
& \sqrt{b_n}\abs{ \frac{1}{\tilde{b}_n} \sqrt{\frac{\pi}{2}} \sum_{j=1}^{\tilde{b}_n} \E{\abs{T_j- \elln \E{Y^2_1} }} \frac{1}{\sqrt{\tilde{\ell}_n}}-\kappa}\\
= & \sqrt{b_n} \sqrt{\frac{\pi}{2}}\abs{\E{\abs{T_1- \elln \E{Y^2_1} }} \frac{1}{\sqrt{\tilde{\ell}_n}}-\kappa\E{\abs{Z}}}\\
\leq & \sqrt{b_n} \sqrt{\frac{\pi}{2}}\abs{\E{\abs{T_1- \elln \E{Y^2_1} }} \frac{1}{\sqrt{\tilde{\ell}_n}}-\Var{T_1/\sqrt{\elln}}\E{\abs{Z}}}\\
 & + \sqrt{b_n} \sqrt{\frac{\pi}{2}} \E{\abs{Z}}\abs{\kappa-\Var{T_1/\sqrt{\elln}}}
\end{align*}
for a standard normally distributed random variable $Z$. Analogous to Proposition \ref{Prop: Replacing kappan by kappa}, one can show that
$$ \sqrt{b_n}\abs{\kappa-\Var{T_1/\sqrt{\elln}}} \rightarrow 0\quad \text{ as } n \rightarrow \infty$$
as long as $b_n=o(\elln)$.
To prove the proposition, it thus remains to show that
$$ \sqrt{b_n} \sqrt{\frac{\pi}{2}}\abs{\E{\abs{T_1- \elln \E{Y^2_1} }} \frac{1}{\sqrt{\tilde{\ell}_n}}-\Var{T_1/\sqrt{\elln}}\E{\abs{Z}}} \pConv 0.$$
Define $F_n(x):= \Pb\rbraces{\rbraces{T_1- \elln \E{Y^2_1}}/\rbraces{\Var{T_1/\sqrt{\elln}}\sqrt{\tilde{\ell}_n}}\leq x}$ and $\Phi(x):=\Pb\rbraces{Z \leq x}$ such that
\begin{align*}
& \sqrt{b_n} \sqrt{\frac{\pi}{2}}\abs{\E{\abs{T_1- \elln \E{Y^2_1} }} \frac{1}{\sqrt{\tilde{\ell}_n}}-\Var{T_1/\sqrt{\elln}}\E{\abs{Z}}}\\
 \leq & \Var{T_1/\sqrt{\elln}} \sqrt{\frac{\pi}{2}} \sqrt{b_n} \int_{-\infty}^\infty \abs{F_n(x)-\Phi(x)}\mathrm{d}x.
\end{align*}

As in the proof of Lemma 3.7, 
we  bound the absolute difference $\sup_x  \abs{F_n(x)-\Phi(x)}$ via Theorem 1 in \citep{Tikhomirov.1980} and Theorem 9, Chapter V, in \citep{Petrov.1975}, thereby obtaining
\begin{align*}
& \sqrt{b_n} \int_{-\infty}^\infty \abs{F_n(x)-\Phi(x)}\mathrm{d}x
\leq \sqrt{b_n} \int_{-\infty}^\infty   C \elln^{-(\delta/2)(\rho-1)/(\rho+1)}\log(C\elln)\frac{1}{1+ \abs{x}^2}    \mathrm{d}x \\
\leq & C \sqrt{b_n}\elln^{-(\delta/2)(\rho-1)/(\rho+1)}\log(C\elln)
\end{align*}
which converges towards zero as long as $1-s<q\delta\frac{\rho-1}{\rho+1}$.
\end{proof}

Lastly, it remains to replace the theoretical variance $\sigma_H^2$ by its empirical counterpart $\hat{\sigma}_H^2:= \frac{1}{n}\sum_{i=1}^n \tilde{X}_i^2$.

\begin{lemma}{
\label{Appendix:Lemma: sigmahat H}
Let the assumptions (A1) and (A2) hold and assume $\ell_n=n^s$ with $s>0.5$.
Then it holds under the null hypothesis
$$\sqrt{b_n}\abs{\sigma_H^2-\hat{\sigma}_H^2}\pConv 0 \quad \text{as } n \rightarrow \infty. $$
}\end{lemma}

\begin{proof} We can write
$$\sqrt{b_n}\abs{\sigma_H^2-\hat{\sigma}^2_H} \leq \sqrt{b_n}\abs{\sigma_H^2-\frac{1}{n}\sum_{i=1}^n \sigma_H^2 Y_i^2}+ \sqrt{b_n}\abs{\frac{1}{n}\sum_{i=1}^n \rbraces{\sigma_H^2Y_i^2-\tilde{X}_i^2}}.$$

For the first of these terms, we obtain
$$\sqrt{b_n}\abs{ \sigma_H^2-\frac{1}{n-1}\sum_{i=1}^n \sigma_H^2Y_i^2}= \sigma_H^2\frac{\sqrt{b_n}}{\sqrt{n}}\abs{\frac{1}{\sqrt{n}} \sum_{i=1 }^n(Y_i^2-1)}\pConv 0,$$
whereas the convergence in probability towards zero of the second term has already been shown in the proof of Proposition \ref{Appendix:Prop: LRV Mean approx H and A}, even with the additional scaling factor$\sqrt{\elln}$.
\end{proof}

The proof of Proposition 2.7 
follows as a mere consequence of the three former results, Proposition \ref{Appendix:Prop: LRV Mean approx H and A}, Proposition \ref{Appendix: Prop: LRV esti under H} and Lemma \ref{Appendix:Lemma: sigmahat H}, once we bound the difference  $\sqrt{b_n}\abs{\hat{\kappa}-\kappa}$ by
\begin{equation*}
\frac{1}{\hat{\sigma}_H^2}\sqrt{b_n}\abs{\hat{\kappa}\hat{\sigma}_H^2-\hat{B}}
+ \frac{1}{\hat{\sigma}_H^2}\sqrt{b_n}\abs{\hat{B}-\kappa\sigma_H^2}
+ \sqrt{b_n}\kappa\abs{1-\frac{\sigma_H^2}{\hat{\sigma}_H^2}}.
\end{equation*}
Corollary 2.8 
is in turn a consequence of Proposition 2.7 
and  Theorem 2.5 
together with an application of Slutzky's Lemma.

\subsubsection{Limit Results  under the Alternative}
\begin{proposition}{
\label{Appendix: Prop: LRV esti under A}
Let the assumptions (A1) and (A2) be fulfilled. Then it holds
$$\frac{1}{\sqrt{\tilde{\ell}_n}} \cdot\hat{B} \pConv  \int_0^1 \abs{\sigma^2(x)-\int_0^1 \sigma^2(y) \mathrm{d}y}\mathrm{d}x$$
as $n\rightarrow \infty$.
}\end{proposition}

\begin{proof}
It holds
\allowdisplaybreaks{
\begin{align*}
&\abs{ \int_0^1 \abs{\sigma^2(x)-\int_0^1 \sigma^2(y) \mathrm{d}y}\mathrm{d}x - \frac{1}{\tilde{\ell}_n} \cdot\frac{1}{\tilde{b}_n}  \sum_{j=1}^{\tilde{b}_n} \abs{\sum_{i=(j-1)\tilde{\ell}_n+1}^{j\tilde{\ell}_n} \sigma_i^2Y_i^2 - \frac{\tilde{\ell}_n}{n}\sum_{i=1}^n \sigma_i^2Y_i^2 }} \\
\leq&  \abs{\int_0^1 \abs{\sigma^2(x)-\int_0^1 \sigma^2(y) \mathrm{d}y}\mathrm{d}x -\frac{1}{\bn}\sum_{j=1}^{\bn}\abs{\sigma^2\rbraces{\frac{j}{\bn}} -\int_0^1 \sigma^2(y) \mathrm{d}y} }\\
& +\abs{\frac{1}{\bn}\sum_{j=1}^{\bn}\abs{\sigma^2\rbraces{\frac{j}{\bn}} -\int_0^1 \sigma^2(y) \mathrm{d}y}   -\frac{1}{\bn}\sum_{j=1}^{\bn}\abs{\sigma^2\rbraces{\frac{j}{\bn}} -\frac{1}{\bn} \sum_{k=1}^{\bn} \sigma^2\rbraces{\frac{k}{\bn}}}}\\
& + \abs{\frac{1}{\bn}\sum_{j=1}^{\bn}\abs{\sigma^2\rbraces{\frac{j}{\bn}} -\frac{1}{\bn} \sum_{k=1}^{\bn} \sigma^2\rbraces{\frac{k}{\bn}}}-\frac{1}{\bn}\sum_{j=1}^{\bn}\abs{\frac{1}{\elln}\sum_{i=(j-1)\elln+1}^{j\elln}\sigma_i^2 Y_i^2-\frac{1}{\bn} \sum_{k=1}^{\bn} \sigma^2\rbraces{\frac{k}{\bn}}}   }\\
& + \abs{\frac{1}{\bn}\sum_{j=1}^{\bn}\abs{\frac{1}{\elln}\sum_{i=(j-1)\elln+1}^{j\elln}\sigma_i^2 Y_i^2-\frac{1}{\bn} \sum_{k=1}^{\bn} \sigma^2\rbraces{\frac{k}{\bn}}} -\frac{1}{\bn}\sum_{j=1}^{\bn}\abs{\frac{1}{\elln}\sum_{i=(j-1)\elln+1}^{j\elln}\sigma_i^2 Y_i^2-\frac{1}{n} \sum_{i=1}^{n} \sigma_i^2Y_i^2}}.
\end{align*}
}In the following, we will argue that each of these terms converges towards zero. The first term is the difference between the outer integral and its Riemann approximation, which converges towards zero by Lebesgue's Theorem since $\sigma^2$ is a c\`adl\`ag-function. For the second term, we obtain
\begin{align*}
&\abs{\frac{1}{\bn}\sum_{j=1}^{\bn}\abs{\sigma^2\rbraces{\frac{j}{\bn}} -\int_0^1 \sigma^2(y) \mathrm{d}y}   -\frac{1}{\bn}\sum_{j=1}^{\bn}\abs{\sigma^2\rbraces{\frac{j}{\bn}} -\frac{1}{\bn} \sum_{k=1}^{\bn} \sigma^2\rbraces{\frac{k}{\bn}}}}\\
\leq & \frac{1}{\bn}\sum_{j=1}^{\bn}\abs{\sigma^2\rbraces{\frac{j}{\bn}} -\int_0^1 \sigma^2(y) \mathrm{d}y- \sigma^2\rbraces{\frac{j}{\bn}} +\frac{1}{\bn} \sum_{k=1}^{\bn} \sigma^2\rbraces{\frac{k}{\bn}}} \\
=& \abs{\int_0^1 \sigma^2(y) \mathrm{d}y-\frac{1}{\bn} \sum_{k=1}^{\bn} \sigma^2\rbraces{\frac{k}{\bn}}},
\end{align*}
converging likewise towards zero by  Lebesgue's Theorem. The third term can be bounded by
\begin{align*}
& \frac{1}{\bn}\sum_{j=1}^{\bn}\abs{\sigma^2\rbraces{\frac{j}{\bn}} -\frac{1}{\elln}\sum_{i=(j-1)\elln+1}^{j\elln}\sigma_i^2 Y_i^2} \\
\leq & \frac{1}{\bn} \sum_{j=1}^\bn \abs{  \sigma^2\rbraces{\frac{j}{\bn}}\frac{1}{\elln}\sum_{i=(j-1)\elln+1}^{j\elln}\rbraces{Y_i^2-1}} \\
& +  \frac{1}{\bn} \sum_{j=1}^\bn \abs{\frac{1}{\elln}\sum_{i=(j-1)\elln+1}^{j\elln}\rbraces{\sigma_i^2-\sigma^2\rbraces{\frac{j}{\bn}} } Y_i^2}.
\end{align*}
The first of these two sums converges in probability towards zero due to
\begin{align*}
&\Pb\rbraces{ \frac{1}{\bn} \sum_{j=1}^\bn \abs{  \sigma^2\rbraces{\frac{j}{\bn}}\frac{1}{\elln}\sum_{i=(j-1)\elln+1}^{j\elln}\rbraces{Y_i^2-1}} >\varepsilon^*}\\
\leq  &\frac{1}{\varepsilon^*}\sigma^2_{\sup} \sqrt{\frac{1}{\elln}} \E{\abs{\sqrt{\frac{1}{\elln}}\sum_{i=1}^{\elln}\rbraces{Y_i^2-1}}^2}^{1/2} \longrightarrow 0.
\end{align*}
Since $\sigma^2$ is a c\`adl\`ag-function, it has only a fixed number $r$ of jumps higher than a given level $\varepsilon>0$. Thus,
\begin{align*}
&\Pb\rbraces{\frac{1}{\bn} \sum_{j=1}^\bn \abs{\frac{1}{\elln}\sum_{i=(j-1)\elln+1}^{j\elln}\rbraces{\sigma_i^2-\sigma^2\rbraces{\frac{j}{\bn}} } Y_i^2}>\varepsilon^* }\\
 \leq &\rbraces{\frac{r}{\bn}2\sigma^2_{\sup}+\frac{\bn-r}{\bn} \varepsilon} \frac{\E{Y_1^2}}{\varepsilon^*}\leq C \varepsilon
\end{align*}
for $\bn$ large enough. By the triangle inequality, the fourth term can be bounded by
\begin{align*}
\abs{\frac{1}{\bn} \sum_{k=1}^{\bn} \sigma^2\rbraces{\frac{k}{\bn}}-\frac{1}{n} \sum_{i=1}^{n} \sigma_i^2Y_i^2}\leq \frac{1}{\bn} \sum_{k=1}^{\bn}\abs{\frac{1}{\elln}\sum_{i=(k-1)\elln+1}^{k\elln}\rbraces{\sigma^2\rbraces{\frac{k}{\bn}}-\sigma_i^2Y_i^2}}
\end{align*}
which converges towards zero as shown above for the third term.
\end{proof}

\begin{lemma}{
\label{Appendix:Lemma: sigmahat A}
Let the assumptions (A1) and (A2) hold. Then,
\begin{equation*}
\abs{\hat{\sigma}_H^2-\int_0^1 \sigma^2(z)\mathrm{d}z }\pConv 0
\end{equation*}
as $n \rightarrow \infty$.}\end{lemma}

Obviously, the above result holds likewise under the hypothesis with the limit simply being zero.

\begin{proof}
We can bound the difference of interest via
$$\abs{\hat{\sigma}_H^2-\int_0^1 \sigma^2(z)\mathrm{d}z } \leq \abs{\frac{1}{n}\sum_{i=1}^n \tilde{X}_i- \frac{1}{n}\sum_{i=1}^n \sigma_i^2Y_i^2}+ \abs{\frac{1}{n}\sum_{i=1}^n \sigma_i^2Y_i^2-\int_0^1 \sigma^2(z)\mathrm{d}z }.$$
The convergence of the first term towards zero has already been shown in the proof of Proposition \ref{Appendix: Prop: LRV esti under A}, even with the additional scaling $\sqrt{b_n}\sqrt{\elln}$.
For the second term, however, we need to show convergence towards the Riemann-integral. It can be split up via
\begin{align*}
&\abs{\hat{\sigma}_H^2-\int_0^1 \sigma^2(z)\mathrm{d}z }\\
\leq & \abs{\int_0^1 \sigma^2(z)\mathrm{d}z- \frac{1}{\bn} \sum_{j=1}^\bn \sigma^2\rbraces{\frac{j}{\bn}} }+ \abs{\frac{1}{\bn} \sum_{j=1}^{\bn} \sigma^2\rbraces{\frac{j}{\bn}}\rbraces{\frac{1}{\elln}\sum_{i=(j-1)\elln+1}^{j\elln}\rbraces{Y_i^2-1}} } \\
& + \abs{\frac{1}{\bn} \sum_{j=1}^{\bn} \frac{1}{\elln}\sum_{i=(j-1)\elln+1}^{j\elln}\rbraces{\sigma^2\rbraces{\frac{j}{\bn}}-\sigma_i^2}Y_i^2}
\end{align*}
and treated similarly to the former proof. In particular, the first  of the three expressions converges towards zero  by Lebesgue's Theorem.  In the second term,  the variance function can be bounded by $\sigma^2_{\sup}$ and the convergence in probability towards zero can be easily shown. For the third term, the c\`adl\`ag-property of $\sigma^2$ is once more employed and the difference $\abs{\sigma^2\rbraces{\frac{j}{\bn}}-\sigma_i^2}$ is bounded by some $\varepsilon>0$ on all but finitely many blocks $j \in \{1, ..., \bn\}$. The convergence then  follows as in the proof of Proposition \ref{Appendix: Prop: LRV esti under A}.
\end{proof}

A combination of Proposition \ref{Appendix:Prop: LRV Mean approx H and A}, Proposition \ref{Appendix: Prop: LRV esti under A} and Lemma \ref{Appendix:Lemma: sigmahat A} now proves that the difference
\begin{align*}
&\abs{\frac{1}{\sqrt{\elln}} \hat{\kappa}- \frac{\int_0^1 \abs{\sigma^2(x)-\int_0^1 \sigma^2(y)\mathrm{d}y}\mathrm{d}x}{\int_0^1 \sigma^2(z) \mathrm{d}z}}\\
\leq &\frac{1}{\sqrt{\elln}} \frac{1}{\hat{\sigma}_H^2} \abs{\hat{\kappa}\hat{\sigma}_H^2-\hat{B}}+ \abs{\frac{1}{\sqrt{\elln}} \frac{\hat{B}}{\hat{\sigma}_H^2}- \frac{\int_0^1 \abs{\sigma^2(x)-\int_0^1 \sigma^2(y)\mathrm{d}y}\mathrm{d}x}{\int_0^1 \sigma^2(z) \mathrm{d}z}}
\end{align*}
in Proposition 2.9 
converges in probability towards zero.

\subsection{Auxiliary Results from the Literature}
\label{Appendix: Subsec: Results from Literature}
This section collects some key results from the literature that are essential tools for our proofs.

Aside from absolute regularity assumed in our setting, there are various other mixing conditions. The $\alpha$-mixing coefficient of two $\sigma$-fields $\mathcal{A}$ and $\mathcal{B}$ is defined as
$$\alpha(\mathcal{A},\mathcal{B})= \sup \{\abs{\Pb(A)\Pb(B)-\Pb(A\cap B)}: A\in \mathcal{A}, B\in \mathcal{B}  \}$$
and a process $(X_i)_{i\in\N}$ is $\alpha$-mixing (or strongly mixing) if
$$\alpha_X(k):=\sup_{m\in\N} \alpha(\sigma\rbraces{X_i, 1\leq i\leq m}, \sigma\rbraces{X_i, k+m\leq i\leq \infty})\longrightarrow 0 \quad \text{  as } k \rightarrow \infty.$$
Note that absolute regularity is a stronger assumption than $\alpha$-mixing since  $2\alpha(\mathcal{A},\mathcal{B})\leq \beta(\mathcal{A},\mathcal{B})$ (see, e.g., Proposition 1, Section 1.1 in Doukhan \citep{Doukhan.1994}) and thus the subsequent statements also apply in our setting. The first ones are a central limit theorem for $\alpha$-mixing processes and a covariance inequality.

\begin{theorem}[Theorem 10.7 in Bradley \citep{Bradley.2007}]
Let $(X_i)_{i\in \mathbb{Z}}$ be a strictly stationary, $\alpha$-mixing sequence of random variables such that $\E{X_0}=0$. Suppose that for some $\vartheta>0$, one has that $\E{\abs{X_0}^{2+\vartheta}}<\infty$ and that the mixing-coefficients satisfy $\sum_{k=1}^{\infty} \alpha(k)^{\vartheta/(2+\vartheta)}<\infty$.
\begin{enumerate}
\item Then $\kappa^2:= \E{X_0^2}+ 2 \sum_{k=1}^\infty \E{X_0X_k}$ exists in $[0,\infty)$ and the sum is absolutely convergent.
\item If also $\kappa^2>0$, then $\sum_{i=1}^n X_i/(\sqrt{n}\kappa) \distConv \NoD{0}{1}$ as $n\rightarrow \infty$.
\end{enumerate}
\end{theorem}
\begin{theorem}[Davydov's covariance inequality; see, Theorem 3, Section 1.2, in  \citep{Doukhan.1994}]{
\label{Appendix: Auxiliary: THM Davydov}
Let $X$ and $Y$ be two random variables that are measurable with respect to the $\sigma$-fields $\mathcal{A}$ and $\mathcal{B}$, respectively. Then it holds

$$\abs{\cov{X}{Y}}\leq 8 \alpha(\mathcal{A},\mathcal{B})^{1/r} \|X\|_p \|Y\|_q,$$
for any $p,q,r\geq 1$ such that $\frac{1}{p}+\frac{1}{q}+\frac{1}{r}=1$.
}\end{theorem}

Additionally,  the following bound holds for the $\beta$-mixing coefficients:
\begin{theorem}[Theorem 1, Section 1.1, in  \citep{Doukhan.1994}]{
Let $(\mathcal{A}_n)_{n\in \N}$ and $(\mathcal{B}_n)_{n\in\N}$ be two sequences of $\sigma$-fields such that $(\mathcal{A}_n\vee \mathcal{B}_n)_{n\in \N}$ are independent. Then it holds
$$\beta(\bigvee_{n=1}^\infty \mathcal{A}_n,\bigvee_{n=1}^\infty \mathcal{B}_n) \leq \sum_{n=1}^\infty \beta (\mathcal{A}_n, \mathcal{B}_n).$$
}\end{theorem}

The following lemma is the main reason for which we employ an absolutely regular and not only an $\alpha$-mixing time series.
\begin{lemma}[Blocking and coupling technique (see, Borovkova, Burton, Dehling \citep{Borovkova.2001}, Lemma 2.4, and the references therein)]
Let $(X_i)_{i\in\N}$ be a stationary and absolutely regular time series with mixing coefficients $(\beta(k))_{k\geq 0}$. For positive integers $M$ and $N$, define the $(M,N)$-blocking of the time series as the sequence of blocks $(B_s)_{s\in \N}$ of $N$ consecutive observations $X_i$, separated by (smaller) blocks of length $M$. Then there exists a sequence of independent, identically distributed random vectors $(B_s')_{s\in \N}$ with the same marginal distributions as $(B_s)_{s\in \N}$ such that
$$ \Pb\rbraces{B_s=B_s'}=1-\beta(M) \quad \text{ for all } s \in \N.$$

\end{lemma}

Frequently in our proofs, we require inequalities bounding the moments of a sum  of some random variables.
\begin{lemma}[$c_r$-inequality]{
For two random variables $X$ and $Y$ with existing $r$-th moments for some $r> 0$, it holds
\begin{equation*}
\E{\abs{X+Y}^r} \leq \max\rbraces{ 1,2^{r-1} } \cdot \rbraces{\E{\abs{X}^r}+ \E{\abs{Y}^r}}.
\end{equation*}
}\end{lemma}

\begin{theorem}[Rosenthal-type inequality; see, Theorem 2, Section 1.4, in \citep{Doukhan.1994}] 
Let $(X_i)_{i\in \N}$ be an $\alpha$-mixing sequence of random variables and let $T$ be a finite subset of $\N$ such that $\E{X_t}=0$ for all $t\in T$.
Assume there exists an $\varepsilon>0$ and a constant $c \in 2\N$ with $c\geq \tau$ such that
$$\sum_{k=1}^\infty \rbraces{k+1}^{c-2} \alpha(k)^{\varepsilon/(c+\varepsilon)}<\infty\quad \text{as well as } \quad
 \E{\abs{X_t}^{\tau+\varepsilon}}<\infty$$ for some $\tau>0$ and all $t\in T$. Then there exists a constant $C$ depending only on $\tau$ and the mixing coefficients $\alpha(k)$ of $X$ such that
 $$\E{\abs{\sum_{t\in T}X_t}^\tau}\leq C D(\tau, \varepsilon,T),$$
 where
 $$
 D(\tau, \varepsilon,T) = \begin{cases}
L(\tau,0,T) \quad \text{ for } 0<\tau\leq 1, \varepsilon\geq 0,\\
L(\tau,\varepsilon,T) \quad \text{ for } 1<\tau \leq 2, \varepsilon>0,\\
\max\rbraces{L(\tau,\varepsilon,T), (L(2,\varepsilon,T))^{\tau/2}} \quad \text{ for } \tau>2, \varepsilon>0,\\
\end{cases}
$$
with $$L(\mu, \varepsilon,T):= \sum_{t\in T} \E{\abs{X_t}^{\mu+\varepsilon}}^{\mu/(\mu+\varepsilon)}.$$
\end{theorem}

\begin{theorem}[Theorem 1 in Yokoyama \citep{Yokoyama.1980}] 
Let $(X_i)_{i\in \N}$ be a strictly stationary $\alpha$-mixing sequence of random variables such that $\E{X_1}=\mu$.  Assume there exist constants $\delta$ with $0<\delta\leq \infty$ and  $t$ with $2\leq t <2+\delta$ such that $$\E{\abs{X_1}^{2+\delta}}<\infty \quad \text{and} \quad \sum_{k=1}^{\infty} k^{t/2-1}\alpha(k)^{\rbraces{2+\delta-t}/(2+\delta)}<\infty.$$
Then it holds
\begin{equation*}
\E{\abs{\sum_{i=1}^n(X_i-\mu)}^t} \leq C n^{t/2}\E{\abs{X_1}^{2+\delta}}^{t/(2+\delta)}.
\end{equation*}
\end{theorem}

The next two theorems enable us to control the difference between an arbitrary distribution function $F$ and the standard normal distribution $\Phi$.

\begin{theorem}[Theorem 9, Chapter V, in Petrov \citep{Petrov.1975}]
Let $\Phi$ denote the distribution function of the  standard normal distribution, let $F$ be another arbitrary distribution function and define $\Delta:= \sup_x \abs{F(x)-\Phi(x)}$. Suppose that $0<\Delta<1/\sqrt{e}$ and that $F$ has finite absolute moments of order $p$ for some $p>0$. Then there exists a constant $C_p$ depending only on $p$ such that
\begin{equation*}
\abs{F(x)-\Phi(x)} \leq \frac{C_p \Delta\log\rbraces{\frac{1}{\Delta}}^{p/2}+\lambda_p}{1+\abs{x}^p}
\end{equation*}
for all $x\in \R$, where
\begin{equation*}
\lambda_p= \abs{\int_{-\infty}^\infty\abs{x}^p\mathrm{d}F(x)-\int_{-\infty}^\infty\abs{x}^p\mathrm{d}\Phi(x)}.
\end{equation*}
\end{theorem}

\begin{theorem}[Theorem 1 in Tikhomirov \citep{Tikhomirov.1980}]
Let $(X_i)_{i\in \N}$ be a strictly stationary, $\alpha$-mixing sequence of random variables with mean zero and finite variance. Let $\Phi$ denote the distribution function of the standard normal distribution and define $$F_n(x):= \Pb\rbraces{\frac{\sum_{i=1}^nX_i}{\sqrt{\Var{\sum_{i=1}^n X_i}}}\leq x}\quad  \text{as well as} \quad \Delta_n:=\sup_x\abs{F_n(x)-\Phi(x)}.$$
Suppose that there exist constants $C_1>0$ and $\rho>1$ such that $$\alpha(k)\leq C_1k^{-\rho(2+\delta)(1+\delta)/\delta^2}$$ holds for all $k\in \N$ and some  $0<\delta\leq 1$ such that $\E{\abs{X_1}^{2+\delta}}<\infty$. Then,
$$\kappa^2=\E{X_1^2}+ 2\sum_{k=2}^\infty \E{X_1X_k}<\infty$$
 and if $\kappa^2>0$,  there exists a constant $C_2$ depending solely on $C_1, \rho $ and $\delta$ such that
\begin{equation*}
\Delta_n\leq C_2 n^{-(\delta/2)(\rho-1)/(\rho+1)}.
\end{equation*}
\end{theorem}

Lastly, we state two very useful results concerning the properties of a c\`adl\`ag function and the continuous mapping theorem, respectively.

\begin{lemma}[Lemma 1, Section 14, in \citep{billingsley1968}]{
Let $f:[0,1]\rightarrow \R$ be a c\`adl\`ag function (right-continuous with left-hand limits). Then for every $\varepsilon>0$, there exist  points $0=t_0<t_1< ... <t_r=1$ such that
$$\sup_{s,t\in[t_{i-1},t_i)} \abs{f(s)-f(t)}<\varepsilon$$
for all $i=1, ..., r$.
}\end{lemma}
As pointed out  in \citep{billingsley1968}, this clearly implies that there are only finitely many jumps that  exceed a given positive number. Moreover, every c\`adl\`ag function $f$ is bounded on $[0,1]$ and has at most countably many discontinuities.

The following theorem is a generalization of the continuous mapping theorem to sequences of functions.
\begin{theorem}[Theorem 5.5 in \citep{billingsley1968}]{
Let $X$ and $X_n$ for $n\in \N$ be real-valued random variables and $h$ and $h_n$ for $n\in \N$ be real-valued, measurable functions. If $X_n\distConv X$ and $\Pb(X\in E)=0$, where
$$ E:=\{x\in \R: \exists (x_n)_{n\in \N} \text{ such that } x_n\rightarrow x \text{ but } h_n(x_n)  \not\rightarrow h(x)  \} ,$$
then it also holds $h_n(X_n)\distConv h(X)$.

}\end{theorem}

\section{Additional Simulation Results}
\label{Appendix: Additional Simulations}
\subsection{Comparison to the Procedure of Dette, Wu and Zhou \citep{Dette.2015},  \citep{Dette.2019}}

\phantom{.}
\vspace*{-1mm}

\begin{table}[H]
    \centering
    \caption{Simulated rejection probabilities of the tests SWFD and DWZ at the nominal significance level $\alpha=0.05$ for the sample sizes $n=500, 1000, 2000$ under the null hypothesis $\mathbb{H}$ and various local alternatives $\mathbb{A}1$ to $\mathbb{A}4$ with effect sizes of magnitude $n^{-1/2}$ and for different data-generating processes.}
    \label{Table: Comparison 500, 1000, 2000 nominal}
    \begin{tabular}{|cc|cccccc|}
     \hline
         & & N(0,1) & Exp(1) & AR(1), 0.4 & AR(1), 0.7 & ARMA(2,2) & GARCH(1,1) \\ \hline
\multicolumn{8}{|c|}{$n=500$} \\ \hline
         \multirow{2}{*}{$\mathbb{H}$}  & SWFD & 0.085 & 0.112 & 0.098 & 0.134 & 0.106 & 0.180 \\
         & DWZ & 0.052 & 0.029 & 0.064 & 0.088 & 0.073 & 0.290 \\\hline
        \multirow{2}{*}{$\mathbb{A}1$} & SWFD & 0.851 & 0.496 & 0.775 & 0.614 & 0.579 & 0.644 \\
         & DWZ & 0.999 & 0.565 & 0.984 & 0.886 & 0.876 & 0.911 \\\hline
        \multirow{2}{*}{$\mathbb{A}2$} & SWFD & 0.627 & 0.343 & 0.564 & 0.465 & 0.424 & 0.481 \\
         & DWZ & 0.630 & 0.103 & 0.550 & 0.404 & 0.345 & 0.617 \\\hline
        \multirow{2}{*}{$\mathbb{A}3$} & SWFD & 0.350 & 0.242 & 0.339 & 0.328 & 0.288 & 0.352 \\
          & DWZ & 0.093 & 0.027 & 0.092 & 0.100 & 0.095 & 0.306 \\\hline
        \multirow{2}{*}{$\mathbb{A}4$} & SWFD & 0.620 & 0.334 & 0.556 & 0.449 & 0.410 & 0.471 \\
          & DWZ & 0.382 & 0.076 & 0.317 & 0.242 & 0.220 & 0.458 \\      \hline
          \multicolumn{8}{|c|}{$n=1000$} \\ \hline
        \multirow{2}{*}{$\mathbb{H}$} & SWFD & 0.077 & 0.098 & 0.090 & 0.111 & 0.104 & 0.168 \\
        & DWZ & 0.056 & 0.033 & 0.056 & 0.096 & 0.077 & 0.346 \\ \hline
        \multirow{2}{*}{$\mathbb{A}1$}  & SWFD & 0.896 & 0.496 & 0.822 & 0.604 & 0.613 & 0.619 \\
         & DWZ & 0.999 & 0.670 & 0.994 & 0.902 & 0.925 & 0.937 \\\hline
        \multirow{2}{*}{$\mathbb{A}2$} & SWFD & 0.642 & 0.319 & 0.563 & 0.405 & 0.400 & 0.438 \\
         & DWZ & 0.797 & 0.142 & 0.706 & 0.516 & 0.485 & 0.755 \\\hline
        \multirow{2}{*}{$\mathbb{A}3$} & SWFD & 0.347 & 0.228 & 0.309 & 0.277 & 0.260 & 0.319 \\
         & DWZ & 0.224 & 0.034 & 0.187 & 0.188 & 0.148 & 0.473 \\\hline
        \multirow{2}{*}{$\mathbb{A}4$} & SWFD & 0.599 & 0.274 & 0.514 & 0.359 & 0.372 & 0.398 \\
        & DWZ & 0.470 & 0.105 & 0.408 & 0.306 & 0.292 & 0.581 \\\hline
          \multicolumn{8}{|c|}{$n=2000$} \\ \hline
                \multirow{2}{*}{$\mathbb{H}$}   & SWFD & 0.073 & 0.091 & 0.074 & 0.096 & 0.084 & 0.148 \\
         & DWZ & 0.052 & 0.039 & 0.058 & 0.108 & 0.074 & 0.394 \\ \hline
        \multirow{2}{*}{$\mathbb{A}1$} & SWFD & 0.932 & 0.474 & 0.849 & 0.574 & 0.620 & 0.591 \\
         & DWZ & 1 & 0.724 & 0.996 & 0.928 & 0.941 & 0.944 \\ \hline
         \multirow{2}{*}{$\mathbb{A}2$} & SWFD & 0.808 & 0.357 & 0.700 & 0.456 & 0.460 & 0.471 \\
         & DWZ & 0.876 & 0.196 & 0.773 & 0.568 & 0.539 & 0.812 \\ \hline
         \multirow{2}{*}{$\mathbb{A}3$} & SWFD & 0.862 & 0.386 & 0.752 & 0.514 & 0.522 & 0.510 \\
       & DWZ & 0.333 & 0.052 & 0.288 & 0.261 & 0.200 & 0.581 \\ \hline
         \multirow{2}{*}{$\mathbb{A}4$} & SWFD & 0.714 & 0.291 & 0.596 & 0.376 & 0.383 & 0.410 \\
         & DWZ & 0.528 & 0.112 & 0.418 & 0.327 & 0.316 & 0.639 \\  \hline
    \end{tabular}
\end{table}

\begin{table}[H]
    \centering
    \caption{Simulated rejection probabilities of the tests SWFD and DWZ with size-correction at the significance level $\alpha=0.05$ for the sample sizes $n=500, 1000, 2000$ under various local alternatives with effect sizes of magnitude $n^{-1/2}$ and for different data-generating processes.}
    \label{Table: Comparison 500, 1000, 2000 size-corrected}
    \begin{tabular}{|cc|cccccc|}
     \hline
          & & N(0,1) & Exp(1) & AR(1), 0.4 & AR(1), 0.7 & ARMA(2,2) & GARCH(1,1) \\\hline
           \multicolumn{8}{|c|}{$n=500$} \\ \hline
           \multirow{2}{*}{$\mathbb{A}1$}  & SWFD & 0.734 & 0.318 & 0.630 & 0.356 & 0.400 & 0.336 \\
        & DWZ & 0.999 & 0.649 & 0.980 & 0.824 & 0.836 & 0.680 \\\hline
        \multirow{2}{*}{$\mathbb{A}2$}  & SWFD & 0.457 & 0.186 & 0.376 & 0.200 & 0.253 & 0.178 \\
        & DWZ & 0.618 & 0.159 & 0.492 & 0.258 & 0.252 & 0.134 \\\hline
        \multirow{2}{*}{$\mathbb{A}3$} & SWFD & 0.221 & 0.126 & 0.202 & 0.128 & 0.154 & 0.114 \\
        & DWZ & 0.088 & 0.047 & 0.069 & 0.044 & 0.061 & 0.042 \\\hline
        \multirow{2}{*}{$\mathbb{A}4$}  & SWFD & 0.481 & 0.194 & 0.399 & 0.222 & 0.255 & 0.199 \\
        & DWZ & 0.372 & 0.118 & 0.274 & 0.144 & 0.160 & 0.096 \\\hline
         \multicolumn{8}{|c|}{$n=1000$} \\ \hline
                 \multirow{2}{*}{$\mathbb{A}1$} & SWFD & 0.837 & 0.350 & 0.676 & 0.430 & 0.466 & 0.367 \\
        & DWZ & 1 & 0.854 & 1 & 0.944 & 0.968 & 0.828 \\\hline
       \multirow{2}{*}{$\mathbb{A}2$} & SWFD & 0.530 & 0.201 & 0.392 & 0.255 & 0.262 & 0.199 \\
        & DWZ & 0.971 & 0.399 & 0.916 & 0.618 & 0.666 & 0.366 \\\hline
        \multirow{2}{*}{$\mathbb{A}3$} & SWFD & 0.242 & 0.124 & 0.190 & 0.147 & 0.148 & 0.120 \\
         & DWZ & 0.786 & 0.240 & 0.674 & 0.376 & 0.398 & 0.232 \\\hline
        \multirow{2}{*}{$\mathbb{A}4$} & SWFD & 0.494 & 0.173 & 0.361 & 0.219 & 0.245 & 0.178 \\
         & DWZ & 0.562 & 0.138 & 0.453 & 0.203 & 0.234 & 0.137 \\\hline
          \multicolumn{8}{|c|}{$n=2000$} \\ \hline
                  \multirow{2}{*}{$\mathbb{A}1$}  & SWFD & 0.891 & 0.344 & 0.784 & 0.439 & 0.505 & 0.375 \\
             & DWZ & 1 & 0.760 & 0.995 & 0.875 & 0.920 & 0.727 \\ \hline
        \multirow{2}{*}{$\mathbb{A}2$} & SWFD & 0.734 & 0.246 & 0.610 & 0.310 & 0.343 & 0.251 \\
         & DWZ & 0.869 & 0.239 & 0.748 & 0.382 & 0.446 & 0.228 \\ \hline
        \multirow{2}{*}{$\mathbb{A}3$}  & SWFD & 0.805 & 0.269 & 0.674 & 0.368 & 0.397 & 0.294 \\
         & DWZ & 0.318 & 0.066 & 0.261 & 0.127 & 0.140 & 0.070 \\ \hline
        \multirow{2}{*}{$\mathbb{A}4$} & SWFD & 0.644 & 0.185 & 0.504 & 0.256 & 0.279 & 0.218 \\
        & DWZ & 0.516 & 0.136 & 0.392 & 0.192 & 0.248 & 0.138 \\ \hline
       \end{tabular}
\end{table}

\pagebreak

\subsection{Analysis of the Empirical Power}
Figure \ref{Fig: Empirical power, size-corrected} presents the size-corrected simulated rejection probabilities. While the overall picture is similar to that of Figure \ref{Fig: Empirical power}
 in Section \ref{Subsec: Simulations Empirical Power}, 
 the empirical power is generally lower than before, though the effect decreseases for larger sample sizes due to the test becoming less oversized. The test for the GARCH(1,1)-distributed data exhibits by far the highest empirical size in Figure \ref{Fig:Empirical Size}
 in Section \ref{Subsec: Simulations Empirical Size} 
 and is thus associated with a substantially lower size-corrected empirical power.

\begin{figure}[H]

 \begin{subfigure}[c]{0.49\textwidth}
  \includegraphics[width=\textwidth,  trim={0cm 0cm 0.8cm 2cm},clip]{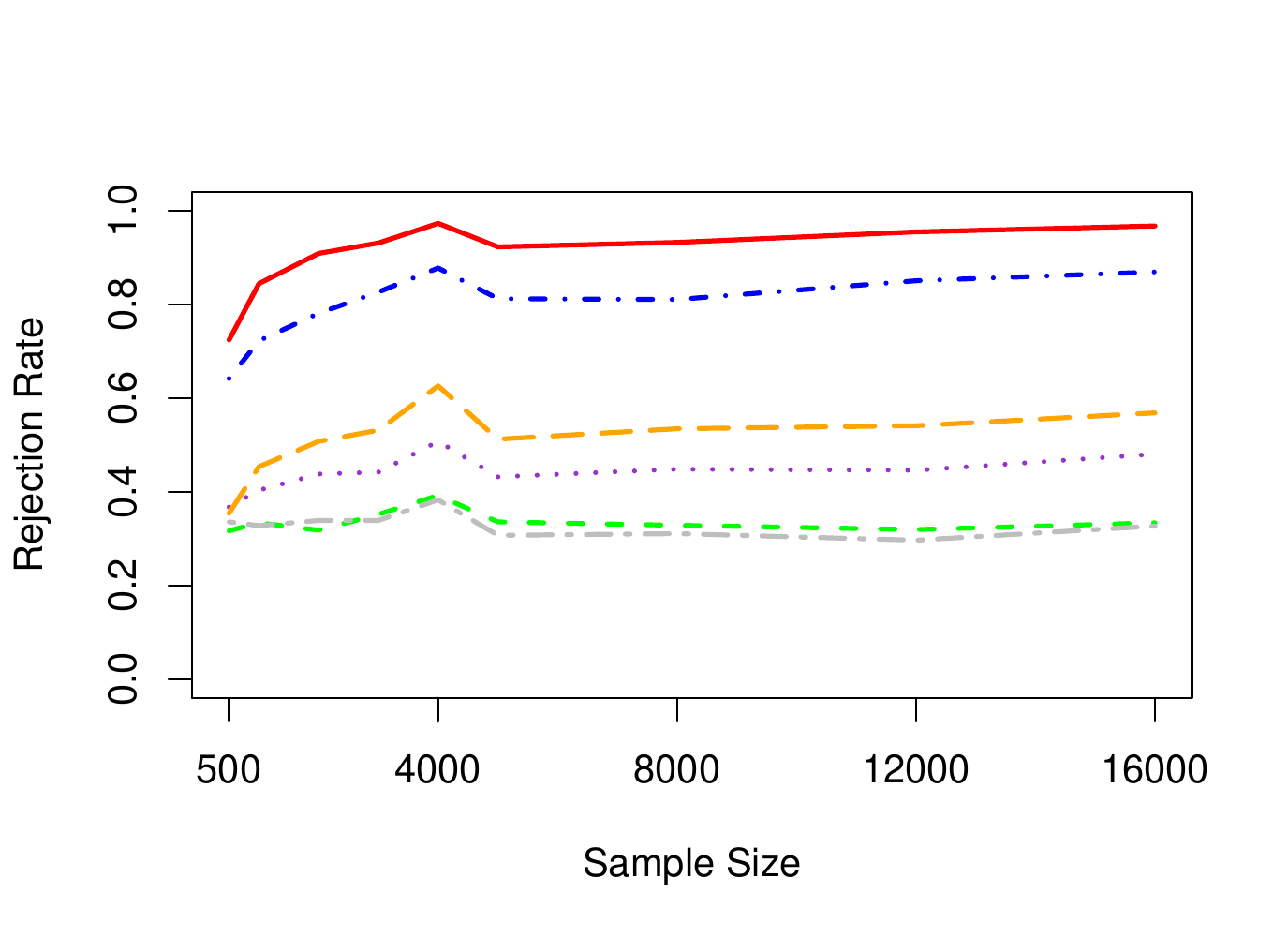}
 \end{subfigure}
   \begin{subfigure}[c]{0.49\textwidth}
  \includegraphics[width=\textwidth,  trim={0cm 0cm 0.8cm 2cm},clip]{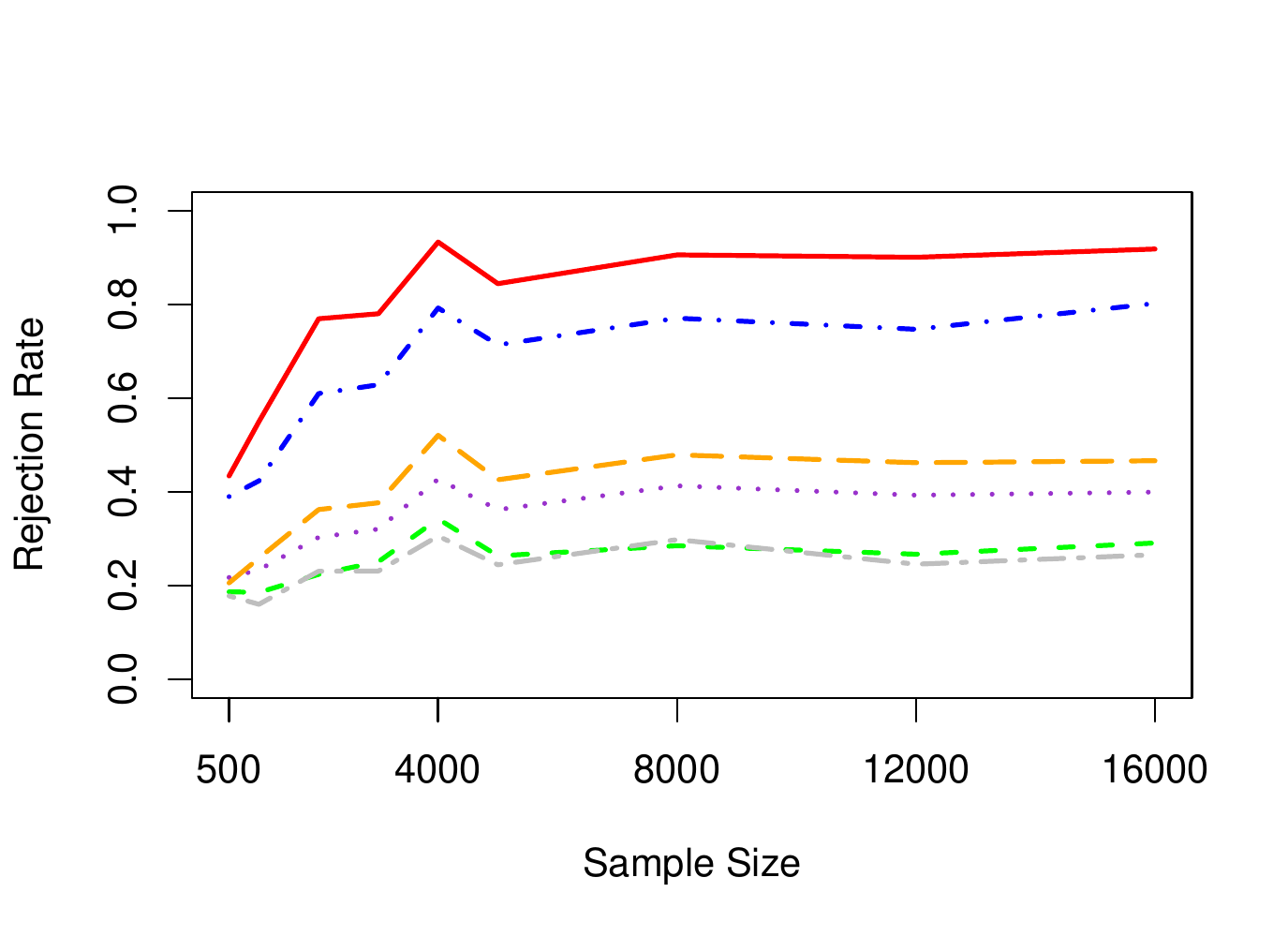}
  \end{subfigure}
   \begin{subfigure}[c]{0.49\textwidth}
  \includegraphics[width=\textwidth,  trim={0cm 0cm 0.8cm 2cm},clip]{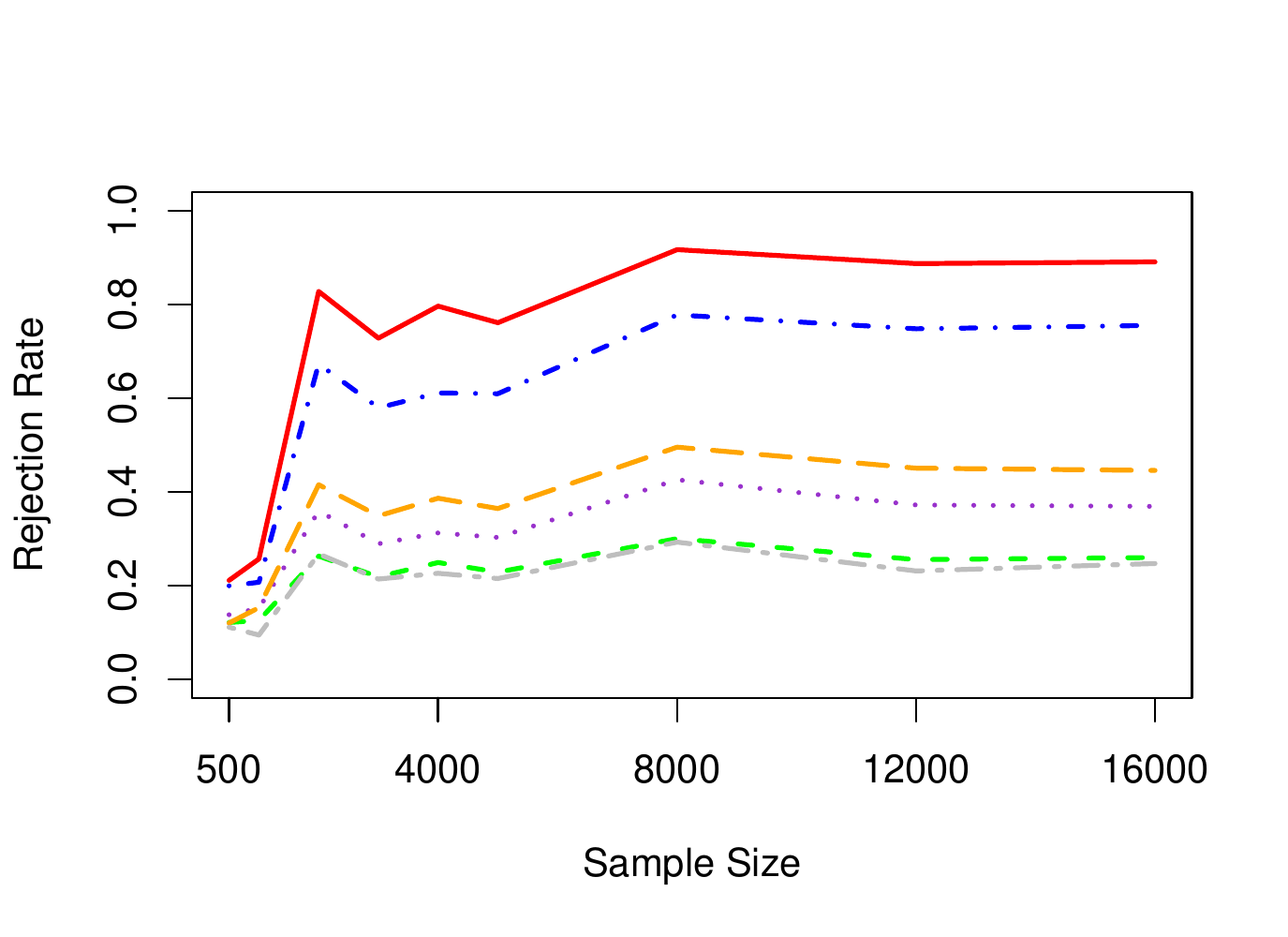}
  \end{subfigure}
   \begin{subfigure}[c]{0.49\textwidth}
  \includegraphics[width=\textwidth,  trim={0cm 0cm 0.8cm 2cm},clip]{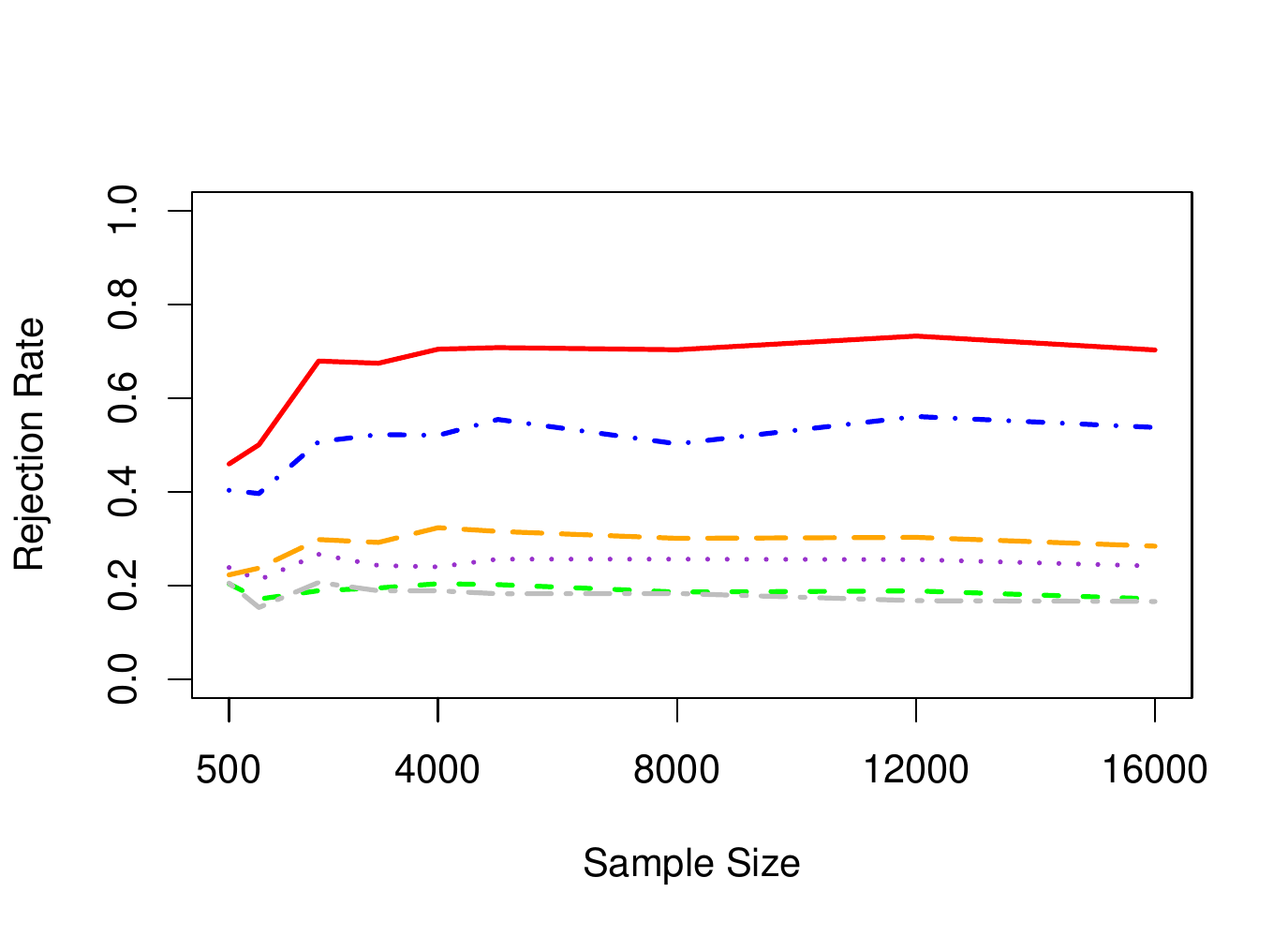}
  \end{subfigure}
     \begin{subfigure}[c]{\textwidth}

  \includegraphics[width=\textwidth, trim={0cm 7cm 0cm 6.5cm},clip]{Pictures/Legend_horizontal_thick.pdf}
  \end{subfigure}
\caption{Empirical power of the SWFD test for the local alternatives $\mathbb{A}1$(top left)-$\mathbb{A}4$(bottom right) with effect sizes of magnitude $n^{-1/2}$ as a function of the sample size $n$, size-correction at significance level $\alpha=0.05$. }
  \label{Fig: Empirical power, size-corrected}
  \end{figure}
  
 \pagebreak

\subsection{Analysis of Performance for non-centred Data}

\phantom{.}

  \begin{table}[H]
    \centering
    \caption{Simulated rejection probabilities of the test SWFD with size-correction at the significance level $\alpha=0.05$ for the sample size $n=3000$ under various local alternatives with effect sizes of magnitude $n^{-1/2}$. Results for different data-generating processes and different mean functions $\mu$ are given. }
    \label{Table: Rej prob different mu size-corr}
    \begin{tabular}{|ccccccc|}
    \hline
         & N(0,1) & Exp(1) & AR(1), 0.4 & AR(1), 0.7 & ARMA(2,2) & GARCH(1,1) \\ \hline
         \multicolumn{7}{|c|}{$\mu(x)=x$} \\ \hline
        $\mathbb{A}1$ & 0.938 & 0.332 & 0.822 & 0.469 & 0.524 & 0.351 \\
        $\mathbb{A}2$ & 0.786 & 0.232 & 0.643 & 0.320 & 0.366 & 0.262 \\
        $\mathbb{A}3$ & 0.744 & 0.221 & 0.578 & 0.300 & 0.328 & 0.234 \\
        $\mathbb{A}4$ & 0.678 & 0.179 & 0.516 & 0.259 & 0.280 & 0.201 \\
        \hline
         \multicolumn{7}{|c|}{$\mu(x)=\sin(2\pi x)$} \\ \hline
        $\mathbb{A}1$ & 0.926 & 0.346 & 0.799 & 0.443 & 0.539 & 0.339 \\
        $\mathbb{A}2$ & 0.794 & 0.261 & 0.611 & 0.303 & 0.392 & 0.242 \\
        $\mathbb{A}3$ & 0.651 & 0.194 & 0.487 & 0.262 & 0.348 & 0.188 \\
        $\mathbb{A}4$ & 0.680 & 0.183 & 0.500 & 0.251 & 0.307 & 0.180 \\
        \hline
		\multicolumn{7}{|c|}{$\mu(x)=0\cdot\1_{\{0\leq x<1/2\}}+ 1\cdot\1_{\{1/2\leq x\leq 1\}}$ for time series $(Z_i)_{i\in\N}$}  \\ \hline
        $\mathbb{A}1$ & 0.768 & 0.281 & 0.862 & 0.902 & 0.566 & 0.318 \\
        $\mathbb{A}2$ & 0.578 & 0.200 & 0.685 & 0.739 & 0.414 & 0.223 \\
        $\mathbb{A}3$ & 0.522 & 0.182 & 0.640 & 0.680 & 0.359 & 0.204 \\
        $\mathbb{A}4$ & 0.473 & 0.150 & 0.576 & 0.638 & 0.307 & 0.184 \\
        \hline
    \end{tabular}
\end{table}


\begin{thebibliography}{}

\bibitem{AbrahamWei.1984}
\textsc{Abraham, B. and Wei, W. W. S.} (1984).
Inferences about the parameters of a time series model with changing variance.
\textit{Metrika}
\textbf{31} 183-194.
\MR{0754960}

\bibitem{Anderegg.2014}
\textsc{Anderegg, W. R. L. and Goldsmith, G. R.} (2014).
Public interest in climate change over the past decade and the effects of the 'climategate' media event. 
\textit{Environ. Res. Lett.}
\textbf{9} 54005.

\bibitem{AndreouGhysels.2002}
\textsc{Andreou, E. and Ghysels, E.} (2002).
Detecting multiple breaks in financial market volatility dynamics.
\textit{J. Appl. Econ.}
\textbf{17} 579-600.

\bibitem{Aue.2009}
\textsc{Aue, A., H\"{o}rmann, S., Horv\'{a}th, L. and Reimherr, M.} (2009).
Break detection in the covariance structure of multivariate time series models.
\textit{Ann. Statist.}
\textbf{37} 4046-4087.
\MR{2572452}

\bibitem{BaufaysRasson.1985}
\textsc{Baufays, P. and Rasson, J. P.} (1985).
Variance changes in autoregressive models.
\textit{Time Series Analysis: Theory and Practice 7}. Anderson, D. (editor), North-Holland, New York, 119-127.
\MR{0787813}

\bibitem{Bloomfield.1994}
\textsc{Bloomfield, P., Hurd, H. L. and Lund, R. B.} (1994).
Periodic correlation in stratospheric ozone data. 
\textit{J. Time Ser. Anal.}
\textbf{15} 127-150.
\MR{MR1263886}

\bibitem{Borovkova.2001}
\textsc{Borovkova, S., Burton, R. and Dehling, H.} (2001).
Limit Theorems for Functionals of Mixing Processes with Applications to U-Statistics and Dimension Estimation.
\textit{Trans. Amer. Math. Soc.}
\textbf{353} 4261--4318.
\MR{1851171}

\bibitem{Bradley.2005}
\textsc{Bradley, R. C.} (2005).
Basic Properties of Strong Mixing Conditions. A Survey and Some Open Questions.
 \textit{Probab. Surv.}
 \textbf{2}
107--144.
\MR{2178042}

\bibitem{Bradley.2007}
\textsc{Bradley, R. C.} (2007).
\textit{Introduction to Strong Mixing Conditions.}
 Kendrick Press, Heber City.
 \MR{2325294}
 
 \bibitem{Burivalova.2018}
\textsc{Burivalova, Z., Butler, R. A. and Wilcove, D. S.} (2018).
Analyzing Google search data to debunk myths about the public's interest in conservation.
\textit{Front. Ecol. Environ.} 
\textbf{16} 509-514.
 
\bibitem{Carlstein.1986}
\textsc{Carlstein, E.} (1986).
The use of subseries values for estimating the variance of a general statistic from a stationary sequence.
\textit{Ann. Statist.} \textbf{14} 1171-1179.
\MR{0856813}

\bibitem{ChenGupta.1997}
\textsc{Chen, J. and Gupta, A.K.} (1997).
Testing and locating variance changepoints with applications to stock prices.
\textit{J. Amer. Statist. Assoc.}
\textbf{92} 739-747.
\MR{1467863}

\bibitem{Dahlhaus.1997}
\textsc{Dahlhaus, R.} (1997).
Fitting time series models to nonstationary processes.
\textit{Ann. Statist.} \textbf{25} 1-37.
\MR{MR1429916}

\bibitem{Davis.1995}
\textsc{Davis, R. A., Huang, D. W. and Yao, Y.} (1995).
Testing for a change in the parameter values and order of an autoregressive model.
\textit{Ann. Statist.} \textbf{23} 282-304.
\MR{1331669}

\bibitem{Dette.2015}
\textsc{Dette, H., Wu, W. and Zhou, Z.} (2015).
Change point analysis of second order characteristics in non-stationary time series.
\textit{arXiv:1503.08610}.

\bibitem{Dette.2019}
\textsc{Dette, H., Wu, W. and Zhou, Z.} (2019).
Change point analysis of correlation in non-stationary time series.
\textit{Statist. Sinica}
\textbf{29} 611-643.
\MR{3931381}

\bibitem{Eichinger.2018}
\textsc{Eichinger, B. and Kirch, C.} (2018).
A MOSUM procedure for the estimation of multiple random change points.
\textit{Bernoulli}
\textbf{24} 526-564.
\MR{MR3706768}

\bibitem{GaleanoPena.2007}
\textsc{Galeano, P. and Pe\~na, D.} (2007).
Covariance changes detection in multivariate time series.
\textit{J. Stat. Plann. Inf.}
\textbf{137} 194-211.
\MR{2292851}

\bibitem{Gao.2019}
\textsc{Gao, Z., Shang, Z., Du, P. and Robertson, J. L.} (2019).
Variance Change Point Detection Under a Smoothly-Changing Mean Trend with Application to Liver Procurement.
\textit{J. Amer. Statist. Assoc.}
\textbf{114} 773-781.
\MR{3963179}

\bibitem{Gerstenberger.2015}
\textsc{Gerstenberger, C. and D. Vogel} (2015).
On the efficiency of Gini's mean difference.
\textit{Stat. Methods Appl.}
\textbf{24} 569-596.
\MR{3421674}

\bibitem{Gerstenberger.2019}
\textsc{Gerstenberger, C., Vogel, D. and Wendler, M.} (2020).
Tests for scale changes based on pairwise differences.
\textit{J. Amer. Statist. Assoc.}  
\textbf{115} 1336-1348. 
\MR{4143469}

\bibitem{Gombay.1996}
\textsc{Gombay, E., Horv\'{a}th, L. and Hu\v{s}kov\'{a}, M.} (1996).
Estimators and tests for change in variances.
\textit{Statist. Decisions}
\textbf{14} 145-159.
\MR{1406521}

\bibitem{Inclan.1994}
\textsc{Incl\'{a}n, C. and  Tiao, G. C.} (1994).
Use of cumulative sums of squares for retrospective detection of changes of variance.
\textit{J. Amer. Statist. Assoc.}
\textbf{89} 913-923.
\MR{1294735}

\bibitem{Lee.2001}
\textsc{Lee, S. and Park, S.} (2001).
The cusum of squares test for scale changes in infinite order moving average processes.
\textit{Scand. J. Statist.}
\textbf{28} 625-644.
\MR{1876504}

\bibitem{Lindner2009}
\textsc{Lindner, A. M.} (2009).
Stationarity, Mixing, Distributional Properties and Moments of GARCH(p, q)--Processes.
\textit{Handbook of Financial Time Series}. Springer, Berlin Heidelberg, 43-69.

\bibitem{Mokkadem.1988}
 \textsc{Mokkadem, A.} (1988).
Mixing properties of ARMA processes.
\textit{Stochastic Process. Appl.}
\textbf{29} 309-315.
\MR{0958507}

\bibitem{RCoreTeam.2019}
\textsc{R Core Team} (2019).
\textit{R: A Language and Environment for Statistical Computing}
Vienna, Austria.


\bibitem{Wichern.1976}
\textsc{Wichern, D. W., Miller, R. B. and D.-A. Hsu} (1976).
Changes of Variance in First-Order Autoregressive Time Series Models - With an Application.
\textit{Appl. Statist.}
\textbf{25} 248-256.

\bibitem{Wied.2012}
\textsc{Wied, D., Arnold, M., Bissantz, N. and Ziggel, D.} (2012).
A New Fluctuation Test for Constant Variances with Applications to Finance.
\textit{Metrika}
\textbf{75} 1111--1127.
\MR{2989322}

\bibitem{Wornowizki.2017}
\textsc{Wornowizki, M., Fried, R. and Meintanis, S. G.} (2017).
Fourier Methods for Analyzing Piecewise Constant Volatilities. 
\textit{AStA Adv. Stat. Ana.}
\textbf{101} 289--308.
\MR{3679347}

\bibitem{Wu.2007}
\textsc{Wu, W. B. and Zhao, Z.} (2007).
Inference of trends in time series. 
\textit{J. R. Stat. Soc. B} \textbf{69} 391-410.
\MR{2323759}

 \end{thebibliography}

\begin{thebibliography}{}


\bibitem{billingsley1968}
\textsc{Billingsley, P.} (1968). \textit{Convergence of
Probability Measures.}
Wiley, New York.
\MR{0233396 }

\bibitem{Borovkova.2001}
\textsc{Borovkova, S., Burton, R. and Dehling, H.} (2001).
Limit Theorems for Functionals of Mixing Processes with Applications to U-Statistics and Dimension Estimation.
\textit{Trans. Amer. Math. Soc.}
\textbf{353} 4261--4318.
\MR{1851171}

\bibitem{Bradley.2007}
\textsc{Bradley, R. C.} (2007).
\textit{Introduction to Strong Mixing Conditions.}
 Kendrick Press, Heber City.
 \MR{2325294}

\bibitem{Dehling.2013}
\textsc{Dehling, H., Fried, R., Sharipov, O. S., Vogel, D. and Wornowizki, M.} (2013).
Estimation of the variance of partial sums of dependent processes.
\textit{Statist. Probab. Lett.}
\textbf{83} 141--147.
\MR{2998735}

\bibitem{Dette.2015}
\textsc{Dette, H., Wu, W. and Zhou, Z.} (2015).
Change point analysis of second order characteristics in non-stationary time series.
\textit{arXiv:1503.08610}.

\bibitem{Dette.2019}
\textsc{Dette, H., Wu, W. and Zhou, Z.} (2019).
Change point analysis of correlation in non-stationary time series.
\textit{Statist. Sinica}
\textbf{29} 611-643.
\MR{3931381}

\bibitem{Doukhan.1994}
\textsc{Doukhan, P.} (1994).
 \textit{Mixing: Properties and Examples.}
Springer, New York.
\MR{1312160}

\bibitem{Petrov.1975}
\textsc{Petrov, V. V.} (1975).
\textit{Sums of Independent Random Variables.}
Springer, Berlin.
\MR{0388499}


\bibitem{Tikhomirov.1980}
\textsc{Tikhomirov, A. N.} (1980).
On the Convergence Rate in the Central Limit Theorem for Weakly Dependent Random Variables.
\textit{Theory Probab. Appl.}
\textbf{25} 790--809.
\MR{0595140}

\bibitem{Yokoyama.1980}
 \textsc{Yokoyama, R.} (1980).
Moment bounds for stationary mixing sequences.
\textit{Probab. Theory Related Fields}
\textbf{52} 45--57.
\MR{0568258 }
\end{thebibliography}
\end{document}